\documentclass[11pt,reqno,oneside]{amsart}

\usepackage{graphicx}
\usepackage{amssymb}
\usepackage{epstopdf}

\usepackage[a4paper, total={5.56in, 9in}]{geometry}
\usepackage{cleveref}

\usepackage{amsmath,amsfonts,amsthm,mathrsfs,amssymb,cite}
\usepackage[usenames]{color}

\usepackage{xcolor}
\newcommand{\Red}[1]{{\color{red}#1}}

\newtheorem{thm}{Theorem}[section]

\newtheorem{cor}[thm]{Corollary}
\newtheorem{lem}{Lemma}[section]
\newtheorem{prop}{Proposition}[section]
\theoremstyle{definition}
\newtheorem{defn}{Definition}[section]

\theoremstyle{remark}

\newtheorem{rem}{Remark}[section]
\numberwithin{equation}{section}

\numberwithin{equation}{section}

\newcounter{saveeqn}
\newcommand{\B}{B}%A ball in 3D

\newcommand{\T}{T}%Transpose

\newcommand{\EM}[1]{\mathbf{#1}}

\allowdisplaybreaks

\title[Geometrical characterizations of  elastic sources and mediums]{Geometrical characterizations of radiating and non-radiating elastic sources and mediums with applications}
\author{Huaian Diao}
\address{School of Mathematics and Key Laboratory of Symbolic Computation and Knowledge Engineering of Ministry of Education, Jilin University, Changchun 130012, People's Republic of China}
\email{diao@jlu.edu.cn, hadiao@gmail.com}

\author{Xiaoxu Fei}
\address{School of Mathematics and Statistics, Central South University, Changsha 410083, People's Republic of China}
\address{Department of Mathematics, City University of Hong Kong, Kowloon, Hong Kong SAR, People's Republic of China}
\email{feixx0921@163.com}

\author{Hongyu Liu}
\address{Department of Mathematics, City University of Hong Kong, Kowloon, Hong Kong SAR, People's Republic of China}
\email{hongyu.liuip@gmail.com, hongyliu@cityu.edu.hk}

\date{} % Activate to display a given date or no date (if empty),
\begin{document}
\maketitle

	\begin{abstract}
     % In this paper, we investigate the time-harmonic elastic wave scattering problems in two scenarios: a scatterer acting as an active source or as an inhomogeneous medium, both of which with compact supports. First, we demonstrate that a scatterer, whether acting as an active source or an inhomogeneous medium, cannot be non-raditaing or invisible when its support is sufficiently small.  
    %  Subsequently, we derive  quantitative local and geometric results under the condition of  a high-curvature point at the boundary of the scatterer, indicating that invisibility cannot occur.
    %  More precisely, we establish a mathematical relationship between  the source intensity or the density of an inhomogeneous medium   at the $K$-curvature point and the curvature $K$.  Leveraging the aforementioned findings, we derive the local geometric properties of the transmission eigenfunctions around the $K$-curvature point. Additionally, we establish  unique determination results through a single far-field measurement for the two aforementioned scenarios.

In this paper, we investigate two types of time-harmonic elastic wave scattering problems. The first one involves the scattered wave generated by an active elastic source with compact support. The second one concerns elastic wave scattering caused by an inhomogeneous medium, also with compact support.  We derive several novel quantitative results concerning the geometrical properties of the underlying scatterer, the associated source or incident wave field, and the physical parameters. In particular, we show that a scatterer with either a small support or high-curvature boundary points must radiate at any frequency. These qualitative characterizations allow us to establish several local and global uniqueness results for determining the support of the source or medium scatterer from a single far-field measurement. Furthermore, we reveal new geometric properties of elastic  transmission eigenfunctions. To derive a quantitative relationship between the intensity of a radiating or non-radiating source and the diameter of its support, we utilize the Helmholtz decomposition, the translation-invariant $L^2$-norm estimate for the Lam\'e operator, and global energy estimates. Another pivotal technical approach combines complex geometric optics (CGO) solutions with local regularity estimates, facilitating microlocal analysis near admissible $K$-curvature boundary points.

	 
	
	\end{abstract}

\section{Introduction}
%In this paper, we consider time-harmonic elastic scattering in two scenarios: the scatterer can be either an active source or an inhomogeneous medium, both with compact supports. We demonstrate that a scatterer with sufficiently small support (sufficiently small relative to the underlying wavelength and scattering intensity) cannot be non-radiating (or invisible). 
%We subsequently localize the ``smallness" result to cases involving a high-curvature point on the  boundary of the support. As applications, these findings are utilized to investigate transmission eigenvalue problems and inverse problems.

We consider time-harmonic elastic wave scattering problems involving either an active elastic source or an inhomogeneous elastic medium, both with compact support. In this context, we refer to the support of the source term as the {\em scatterer}, as it represents the region responsible for generating the scattered field. Likewise, an inhomogeneous medium is also termed a {\em scatterer}. In this section, we introduce the mathematical models that serve as the foundation for our subsequent analysis and review relevant results from the existing literature. Finally, we summarize the main contributions of this paper at the end of the section.

%Since the elastic scattering problems involving vector-valued functions, our analysis requires more complex and precise calculations.  In Subsection \ref{sec:1.1}, we present the mathematical formulation for the both aforementioned scenarios.

\subsection{Mathematical setup}\label{sec:1.1}

%In the physical setup, $\Omega$ is the support of an inhomogeneous elastic scatterer. 

First, we are concerned with the time-harmonic elastic source scattering problem  in a homogeneous medium. Let $\Omega$ be a bounded Lipschitz domain in $\mathbb R^n\ (n\in\{2,3\})$ with a connected complement $\mathbb R^n\setminus \overline{\Omega}$. Let $\mathbf f:\mathbb R^n\to \mathbb C^n$ be a complex-valued vector function with a compact support such that  $ {\rm supp}(\mathbf f) \subset \Omega$. Denote $\mathbf f=\chi_{\Omega}\boldsymbol\varphi,$ where $\chi_\Omega$ is the characteristic function of $\Omega$, $ \boldsymbol\varphi \in L^\infty(\mathbb R^n)^n$  signifies the intensity of the elastic source 
 at various points in $\Omega$ and $\boldsymbol\varphi\not=\mathbf 0$ in a neighborhood of $\partial \Omega$.  Consider the  time-harmonic elastic source scattering problem:
\begin{equation}\label{eq:lame}
\mathcal L\mathbf u+\omega^2\mathbf u=\mathbf f,
\end{equation}
where the operator
$\mathcal L:=\mu \Delta +(\lambda+\mu)\nabla\nabla\cdot,
\ \omega>0$ is the angular frequency. The Lam\'e constants $\lambda$ and $\mu$ satisfy the  strong convexity condition:
\begin{equation}\label{eq:strong conv}
\mu>0 \ \mbox{and}\ n\lambda+2\mu >0,\ \mbox{for}\ n=2,3.
\end{equation}
%The vector $\mathbf u\in H_{loc}^2(\mathbb R^n)$ denotes the outgoing elastic filed produced by the source $\mathbf f$ satisfying the time-harmonic system \eqref{eq:lame}.
By  the Helmholtz decomposition (cf.\cite{FKS}), the external elastic source $\mathbf f\in L^q(\Omega)^n,1<q<\infty$ has the decomposition: 
\begin{equation}\label{eq:dehelm}
\mathbf f=\mathbf f_p+\mathbf f_s,
\end{equation}
where
$\mathbf f_p,\mathbf f_s\in L^{q}(\Omega)^n$ with $\nabla \times \mathbf f_p=\EM 0, \nabla \cdot \mathbf f_s=0$. Additionally, the   displacement vector field  $\mathbf u=(u_\ell)^n_{\ell=1}$ to  \eqref{eq:lame} can be decomposed into the pressure wave $\mathbf u_p$ and shear wave $\mathbf u_s$, satisfying
\begin{equation}\notag
\mathbf u=\mathbf u_{p}+\mathbf u_{ s}
\end{equation}
and
\begin{equation}\label{eq:deu}
\begin{cases}
&\Delta \mathbf u_p+\kappa_p^2\mathbf u_p=\mathbf f_p,\ \nabla \times \mathbf u_p=\mathbf 0, \\
&\Delta \mathbf u_s+\kappa_s^2\mathbf u_s=\mathbf f_s,\ \nabla \cdot \mathbf u_s=0, 
\end{cases}
\end{equation}
where $\kappa_{  p},\kappa_{ s}$ are the wave numbers for the pressure and shear waves given by
\begin{equation}\label{eq:wave number}
\kappa_{  p}=\frac{\omega}{\sqrt{\lambda+2\mu}}\ \mbox{and}\ \kappa_{ s}=\frac{\omega}{\sqrt{\mu}}.
\end{equation}
Moreover, the pressure wave $\mathbf u_p$ and shear wave $\mathbf u_s$ are characterized by
\begin{equation}\label{eq:decomp}
\mathbf u_{ p}=-\frac{1}{\kappa_{ p}^2}\nabla(\nabla \cdot \mathbf u)\ \mbox{and}\ \mathbf u_{s}=\left \{ \begin{array}{lr}
\frac{1}{\kappa_{ s}^2} \nabla\times\nabla\times\mathbf u\ (\mathbb R^3),&\\
\frac{1}{\kappa_{ s}^2} {\bf curl}\ \mathrm {curl}\ \mathbf u\ (\mathbb R^2),&
\end{array}\right.
\end{equation}
where the two-dimensional operators ${\bf curl}$ and curl are defined by
$$\mathrm {curl}\  \mathbf u=\partial_1u_2-\partial_2u_1,\ {\bf curl }\ u=(\partial_2 u,-\partial_1u)^{\top},$$
with $\mathbf u=(u_1,u_2)$ and $u$ being vector-valued and scalar-valued functions, respectively.
It is important to note that the aforementioned decompositions also play a significant role in our subsequent analysis.

The pressure wave $\mathbf u_p$ and the shear wave $\mathbf u_s$ satisfy the Kupradze radiation condition
\begin{equation}\label{eq:rad}
\begin{cases}
\lim_{\vert \mathbf x\vert  \to\infty} \frac{1}{\vert \mathbf x\vert ^{(n-1)/2}}(\frac{\partial \mathbf u_{p}}{\partial \vert \mathbf x\vert }-\mathrm i\kappa_{ p} \mathbf u_{ p})=\mathbf 0,\\
\lim_{\vert \mathbf x\vert \to\infty} \frac{1}{\vert \mathbf x\vert ^{(n-1)/2}}(\frac{\partial \mathbf u_{ s}}{\partial \vert \mathbf x\vert }-\mathrm i\kappa_{ s} \mathbf u_{ s})=\mathbf 0.
\end{cases}
\end{equation}
It is known that for a given  elastic source $\mathbf f\in L^\infty (\Omega)^n$, there exists an unique  displacement field $\mathbf u \in H^2_{\rm loc}(\mathbb R^n)^n$ to \eqref{eq:lame} and \eqref{eq:rad}(see \cite{BLY}). Furthermore,  $\mathbf u$ satisfies the following asymptotic expansion (cf.\cite{Da,P1998})
\begin{equation}\notag
\mathbf u=\frac{e^{\mathrm i\kappa_p\vert \mathbf x\vert }}{\vert \mathbf x\vert ^{\frac{n-1}{2}}} u_p^\infty(\hat{\mathbf x})\hat{\mathbf x}+\frac{e^{\mathrm i\kappa_p\vert \mathbf x\vert }}{\vert \mathbf x\vert ^{\frac{n-1}{2}}}\mathbf  u_s^\infty(\hat{\mathbf x})+\mathcal O(\vert \mathbf x\vert ^{-\frac{n+1}{2}}),\ \mathbf u_s^\infty({\hat{\mathbf x}})\cdot\hat{\mathbf x}=0
\end{equation}
which holds for all directions  $\hat{\mathbf x}:=\frac{\mathbf x}{\vert \mathbf x\vert }\in \mathbb S^{n-1}$ $(n=2,3)$ as $\vert \mathbf x\vert\to \infty$. 
Here, the scalar-valued function $u_p^\infty(\hat{\mathbf x})$ and the vector-valued function $\mathbf u_s^\infty(\hat{\mathbf x})$ are  both defined on the unit sphere $\mathbb S^{n-1}$ characterizing the  far-field pattern of the pressure wave $\mathbf u_p$ and the shear wave $\mathbf u_s$, respectively.
Denote the total far-field pattern $\mathbf u^\infty$ associated with $\mathbf u$ as follows
$$
\mathbf u^\infty:= u_p^\infty(\hat{\mathbf x}){\hat{\mathbf x}}+\mathbf u_s^\infty(\hat{\mathbf x}),
$$ then it is easy to know that the correspondence between $\mathbf u$ and $\mathbf u^\infty$ are one to one by Rellich's Lemma\cite{P1998}.

Next we consider the elastic scattering problem by the interaction of an incident wave $\mathbf u^i$ and an inhomogeneous medium scatterer $\Omega$, where $\Omega\Subset \mathbb R^n,n\in \{2,3\}$ denotes the support of the  inhomogeneous  elastic medium scatterer with the physical configurations $(\lambda,\mu,V)$ within a uniform and homogeneous background space $\mathbb R^n\backslash \overline \Omega $, where $V\in L^\infty (\mathbb R^n  )$ and $supp(V)\subset\Omega $. Here $\lambda$ and $\mu$ are Lam\'e constants satisfying \eqref{eq:strong conv} and $(1+V)$ denotes the inhomogeneous medium  density. 
Let $\mathbf u^{i}$ be an incident field which is a complex-valued entire solution to the Lam\'e system
$$\mathcal L \mathbf u^i+\omega^2\mathbf u^i=\mathbf 0\ \mbox{in}\ \mathbb R^n,\ n\in \{2,3\}.$$
Due to the interaction of the incident wave $\mathbf u^i$ with the inhomogeneous scatterer $\Omega$, the scattered displacement field  $\mathbf u$ is generated and  the total     field $\mathbf u^t$ is the superposition of $\mathbf u^i$ and $\mathbf u$, namely $\mathbf u^t:=\mathbf u^i+\mathbf u$, where
$\mathbf u^t$ satisfies the following Lam\'e system:
\begin{equation}\label{eq:mesca}
\mathcal L\mathbf u^t+\omega^2(1+V)\mathbf u^t=\mathbf 0\ \mbox{in}\ \mathbb R^n,\ n\in \{2,3\}.
\end{equation}  
Furthermore, the  {displacement}  field $\mathbf u$ satisfies    Helmholtz decomposition:
$$\mathbf u=\mathbf u_p+\mathbf u_s,$$
where $\mathbf u_p$ and $\mathbf u_s$ are the pressure wave and shear wave defined by \eqref{eq:decomp}, respectively. The {displacement}  field $\mathbf u$ fulfills the Kupradze radiation condition \eqref{eq:rad}. 
The well-posedness of \eqref{eq:mesca} can be found in \cite{P1998}, which  guarantees the unique solution $\mathbf u^t\in H^1_{\rm loc}(\mathbb R^n)^n$ to the scattering problem \eqref{eq:mesca}.  
Moreover, the {displacement}  field $\mathbf u$ has the asymptotic expansion
\begin{equation}\notag
\mathbf u=\frac{e^{\mathrm i\kappa_p\vert \mathbf x\vert }}{\vert \mathbf x\vert ^{\frac{n-1}{2}}} u_p^\infty(\hat{\mathbf x})\hat{\mathbf x}+\frac{e^{\mathrm i\kappa_s\vert \mathbf x\vert }}{\vert \mathbf x\vert ^{\frac{n-1}{2}}}\mathbf  u_s^\infty(\hat{\mathbf x})+\mathcal O(\vert \mathbf x\vert ^{-\frac{n+1}{2}}),\quad  \mathbf u_s^\infty({\hat{\mathbf x}})\cdot\hat{\mathbf x}=0.
\end{equation}
Hence $\mathbf u^\infty:= u_p^\infty(\hat{\mathbf x})\cdot {\hat{\mathbf x}}+\mathbf u_s^\infty$ is  the far-field pattern associated with  $\mathbf u$.

Throughout this paper, we denote a source  by $(\Omega; \mathbf{f})$ and a medium scatterer by $(\Omega; \lambda, \mu, V)$, where $\Omega$ represents a scatterer, and $\mathbf{f}$ or $(\lambda, \mu, V)$ characterizes its physical configuration. Specifically, in the case of a medium scatterer, the incident wave $\mathbf{u}^i$ is also regarded as part of the physical setup, since its interaction with the inhomogeneous medium $\Omega$ gives rise to a total displacement field $\mathbf{u}$, which behaves similarly to a source. The inverse problems studied in this work aim to recover the shape and location of $\Omega$, independent of its physical configuration $\mathbf{f}$ or $(\lambda, \mu, V)$, by using information from the far-field pattern $\mathbf{u}^\infty$ obtained from a \emph{single far-field measurement}. This is formally described by the following mapping:
\begin{equation}\label{eq:inv}
\mathbf{u}^\infty(\hat{\mathbf{x}})\ (\forall \hat{\mathbf{x}} \in \mathbb{S}^{n-1}) \longmapsto \partial \Omega.
\end{equation}
Here, $\mathbf{u}^\infty$ denotes the far-field pattern arising from either a source or medium scattering configuration. A single far-field measurement means that the frequency $\omega$ is fixed for the source scattering problem, and in the case where $\Omega$ is a medium scatterer, the far-field pattern is generated by a \emph{single} incident wave $\mathbf{u}^i$ and fixed frequency $\omega$. It is important to note that the inverse problem defined by \eqref{eq:inv} is \emph{nonlinear} and is formally determined by only one far-field measurement. Establishing uniqueness in recovering the shape and position of the scatterer $\Omega$ under such minimal data is a highly challenging problem in inverse scattering theory, a field with a long and rich history (cf.\ \cite{CX, DR2018}). To date, uniqueness results with finitely many far-field measurements have only been achieved under additional \emph{a-priori} assumptions on the scatterer's size or geometry (e.g., spherical or polyhedral shapes) \cite{CX, DR2018}.

Another focus of our study is the phenomenon of \emph{non-radiation}, %or \emph{invisibility}, 
characterized by the condition $\mathbf{u}^\infty \equiv \mathbf{0}$. In this case, the source $(\Omega; \mathbf{f})$ and the inhomogeneous medium scatterer are  said to be \emph{non-radiating} or \emph{radiationless}. %and the inhomogeneous medium scatterer $(\Omega; \lambda, \mu, V)$ is considered \emph{invisible} or \emph{transparent}. 
In particular, when the medium scatterer is radiationless, Rellich's lemma implies that 
\[
\mathbf{u}^s = 0 \quad \text{in } \mathbb{R}^n \setminus \overline{\Omega}, \quad n \in \{2, 3\}.
\]
Consequently, the total field coincides with the incident field in the exterior of the scatterer: $\mathbf{u}^t = \mathbf{u}^i$ in $\mathbb{R}^n \setminus \overline{\Omega}$. In such a case, setting $(\mathbf{w}, \mathbf{v}) = (\mathbf{u}^t|_{\Omega}, \mathbf{u}^i|_{\Omega})$ leads to the following elastic transmission eigenvalue problem:
\begin{equation}\label{eq:trans}
\begin{cases}
\mathcal{L} \mathbf{w} + \omega^2(1 + V)\mathbf{w} = \mathbf{0} \quad &\text{in } \Omega, \\
\mathcal{L} \mathbf{v} + \omega^2 \mathbf{v} = \mathbf{0} \quad &\text{in } \Omega, \\
\mathbf{w} = \mathbf{v}, \quad T_\nu \mathbf{w} = T_\nu \mathbf{v} \quad &\text{on } \Gamma,
\end{cases}
\end{equation}
where $\Gamma = \partial \Omega$. The traction operator $T_\nu$ is defined as
\begin{align*}
T_\nu \mathbf{u} = 
\begin{cases}
2\mu \dfrac{\partial \mathbf{u}}{\partial \nu} + \lambda \nu \nabla \cdot \mathbf{u} + \mu \nu^\perp(\partial_2 u_1 - \partial_1 u_2), & n = 2, \\
2\mu \dfrac{\partial \mathbf{u}}{\partial \nu} + \lambda \nu \nabla \cdot \mathbf{u} + \mu \nu \times (\nabla \times \mathbf{u}), & n = 3,
\end{cases}
\end{align*}
where $\nu = (\nu_1, \nu_2)^\top$ is the unit outer normal on $\partial \Omega$ and $\nu^\perp = (-\nu_2, \nu_1)^\top$ in two dimensions, and $\nu = (\nu_1, \nu_2, \nu_3)^\top$ in three dimensions.
It is evident that $(\mathbf{w}, \mathbf{v}) = (\mathbf{0}, \mathbf{0})$ is a trivial solution to \eqref{eq:trans}. If there exists a positive number $\omega \in \mathbb{R}_{+}$ along with a nontrivial pair $(\mathbf{w}, \mathbf{v}) \in H^1(\Omega)^n \times H^1(\Omega)^n$ satisfying \eqref{eq:trans}, then $\omega$ is called an \emph{elastic transmission eigenvalue}, and $(\mathbf{w}, \mathbf{v})$ are the corresponding \emph{transmission eigenfunctions}.

\subsection{Connections to existing studies and  main contributions}

Non-radiating sources have significant practical applications across various areas of physics and engineering, including geological prospecting, electromagnetic detection, and cloaking technology. It is well known that non-radiation is excluded when the source support possesses certain geometric features. The study of non-radiating sources has a long and rich history, with extensive literature dedicated to this topic. In both acoustic and elastic source scattering, the sources must radiate when the support contains a convex or non-convex corner or edge on the boundary \cite{B2018,BLY}. For acoustic sources with smooth boundaries, non-radiation is also precluded if the support is sufficiently small or if the boundary includes high-curvature points \cite{BL2021}. A similar development for radiating acoustic sources is presented in \cite{Hu1}, where the support includes weakly singular boundary points. Furthermore, the characterization of non-radiating sources for elastic waves in anisotropic and inhomogeneous media-encompassing both volume and surface sources-has been addressed in \cite{KW21}.

In the context of inhomogeneous medium scattering, non-radiation in inverse scattering problems implies that the scatterer becomes undetectable by external measurements, as the incident probing wave remains unperturbed and the resulting scattered field vanishes identically. However, non-radiation is ruled out for scatterers that contain polygonal, polyhedral, or circular conical corners, as such geometric features inherently induce scattering \cite{CX,PSV,EH18,BPS}. In addition, it has been rigorously shown that acoustic media with smooth boundaries containing points of sufficiently high curvature cannot be invisible \cite{BL2021}. Similar results concerning the characterization of non-radiating scatterers in isotropic or anisotropic, and inhomogeneous acoustic media, can be found in \cite{CV,SS21,CVX23,KSS}.

Transmission eigenvalue problems, which originate from the study of non-radiating phenomena in medium scattering, are a class of non-elliptic and non-self-adjoint eigenvalue problems \cite{CCH}. Spectral studies of transmission eigenvalues involve the analysis of their existence, discreteness, infiniteness, and Weyl-type laws. Developments regarding the vanishing behavior of acoustic and elastic transmission eigenfunctions at convex or non-convex corners are discussed in \cite{B2018,BLY}. When the domain involved in a transmission eigenvalue problem contains a high-curvature point, a delicate relationship is revealed between the curvature at that point and the associated transmission eigenfunctions \cite{BL2021}.

In inverse scattering theory, the unique identifiability of determining the shape of scatterers from a single far-field measurement is a long-standing and challenging problem. The problem of uniquely reconstructing the location and shape of a scatterer from a single far-field measurement is commonly referred to as the Schiffer problem \cite{DR2018,DLbook}. To date, uniqueness results have only been established under certain a priori assumptions on the size or shape of the scatterer; see \cite{AR,CK,DR2018,LZ08} for detailed discussions.  For inhomogeneous elastic media, uniqueness in the determination of shape using multiple far-field measurements has been addressed in \cite{Hahner1,P1998}. Furthermore, uniqueness results by a single far-field measurement for convex polygonal elastic medium scatterers under generalized transmission boundary conditions have been established in \cite{DLS2021}.

%DTLT
%Unique identifiability for both inverse acoustic and electromagnetic scatterers under a priori geometric knowledge from a few far-field measurements is discussed in \cite{BL2021,CDLZ1,CDLZ2,DLZZ}. 
%Moreover, some related stable determinations of scatterers from finite measurements have been established in \cite{LT22,SSini,RSS}.%LRX19

%The enclosure method was used to establish the extraction formula for determining the location and the shape of unknown  convex hull of the unknown polygonal cavity  and crack within an elastic body from a single set of boundary data\cite{Ike1,Ike2}. 

\medskip
 The main contributions of this paper can be summarized in several key points:

\begin{enumerate}
    \item[(i)] 
    Under the H\"older regularity assumption on the elastic source within its support, we rigorously demonstrate that the source must radiate at any frequency when its intensity satisfies certain quantitative conditions (see \eqref{eq:small}) involving the diameter of the support and the frequency. This result indicates that if the support of the elastic source is sufficiently ``small'' (in terms of the wavelength), then radiation necessarily occurs. This provides a geometric characterization---specifically in terms of the diameter---of the support for radiating elastic sources. Further details can be found in Theorem \ref{thm:small} and Remark \ref{rem:radiating}. Conversely, if an elastic source is non-radiating, then the diameter of its support must possess a positive lower bound that depends solely on the source intensity and the frequency (cf.\ Corollary \ref{cor1}). This offers a corresponding geometric characterization for non-radiating elastic sources in terms of the support's diameter.

    \iffalse
    
    Under the H\"older regularity assumption of the elastic source in the  support,  when the intensity of the elastic source fulfill  quantitative relationships (see \eqref{eq:small} in Theorem \ref{thm:small}) with respect to the diameter of the support and the  frequency, we rigorously show that  the elastic source must radiate at any frequency. The result established in Theorem \ref{thm:small} implies that when the support of the elastic source is ``small", then the radiating phenomena must occur, which establish the geometrical characterization (namely, the diameter) of the support with respect to the radiating source. The more detailed discussion can be founded in Theorem \ref{thm:small} and Remark \ref{rem:21}. On the other hand, if the elastic source is non-radiating, then the diameter of the support must has a positive lower bound only depending on the intensity of the elastic source and  the frequency (cf. Corollary \ref{cor1}), which serves as a geometrical characterization for non-radiating elastic source with respect to the diameter of the support.

    We establish an intricate  quantitative characterization between  the intensity of a non-radiating source and the diameter of its support.  The results demonstrate that scatterer with sufficiently small support cannot be non-radiating completely.
    
    \fi 
    
    \item[(ii)] 
    
    When the support of the source contains the admissible \( K \)-curvature points (Definition~\ref{def:ad-hi-cur}) on its boundary, we establish a rigorous relationship between the source's radiation intensity and the curvature (Theorem~\ref{thm:source rad}). This result further implies that if an admissible \( K \)-curvature point has sufficiently large curvature, the source must radiate.
Conversely, for a non-radiating source whose support includes an admissible \( K \)-curvature point, we prove that the intensity admits an upper bound that decays with increasing curvature (Corollary~\ref{cor:non-rad}). Thus, the bound becomes arbitrarily small when the curvature is sufficiently large.

    \item[(iii)] We extend the results from (i) and (ii) for the source scattering problem to the inhomogeneous medium scattering problem. 
Under the two aforementioned geometric structures, we derive quantitative characterizations of both the physical parameters (related to the density of a radiating inhomogeneous medium scatterer) and the incident wave; see Theorems~\ref{thm:medsmall} and~\ref{thm:medKpoint1} for details.

%If the domain for defining \eqref{eq:trans}  is small relative to the wavelength of the corresponding transmission eigenvalue,

Furthermore, we establish geometric properties of transmission eigenfunctions associated with the transmission eigenvalue problem \eqref{eq:trans} in two distinct scenarios:

\begin{itemize}
    \item Let $\varepsilon = d(\Omega)\omega$, where $\omega \in \mathbb{R}_+$ is a transmission eigenvalue for \eqref{eq:trans}, and $d(\Omega)$ denotes the diameter of the domain $\Omega$. We demonstrate that the $L^\infty$-norm of the transmission eigenfunction on the boundary of $\Omega$ is bounded by quantities involving $V$ from \eqref{eq:trans} and $\varepsilon$ (Theorem~\ref{thm:tans1}). For sufficiently small $\varepsilon$, i.e., when the diameter of $\Omega$ is sufficiently small for a fixed transmission eigenvalue $\omega$, we rigorously prove that the corresponding transmission eigenfunction is nearly vanishing on the boundary of $\Omega$.
    
    \item If the boundary of the domain contains an admissible $K$-curvature point, we establish a precise relationship between the value of the transmission eigenfunction at this point and the curvature $K$ (Theorem~\ref{thm:trans2}). This result implies that the transmission eigenfunction is nearly vanishing at the admissible $K$-curvature point when the curvature $K$ is sufficiently large.
\end{itemize}

    \item[(iv)]  
    %established in Theorems \ref{thm:in2loc} and \ref{I41},

   Using the non-radiating properties of elastic sources and media established in Corollaries \ref{cor:non-rads} and \ref{cor:non-rad} and Theorems \ref{thm:tans1} and \ref{thm:trans2}, we derive novel local and global unique identification results for determining the shape of radiating elastic sources and medium scatterers from a single far-field pattern under general conditions defined by admissible classes. The global uniqueness result determines the number of scatterer components. The local uniqueness result shows that if two scatterers produce identical far-field patterns, their difference cannot include an admissible $K$-curvature point or a small component.

   % Finally, as applications, we investigate the inverse source and medium problems. By introducing several new admissible classes in Definitions \ref{def:ad1}--\ref{def:ad-me2}, we establish the local and global uniqueness results for the shape determination of scatterers by a single far-field pattern measurement, independent of the physical configurations. More details can be found in Section \ref{sec:inverse}.
 % Finally, as applications, we investigate the inverse source and medium problems. By introducing several admissible classes in Definitions \ref{def:ad1}--\ref{def:ad-me2}, we establish local and global uniqueness results for the shape determination of scatterers using a single far-field pattern measurement, independent of the physical configurations. More details can be found in Section \ref{sec:inverse}.
\end{enumerate}

 In what follows,  we present an overview of the principal methodological approaches developed in this paper. %providing detailed insights into contributions (i) and (ii).
 For the main contribution (i),  to investigate the geometrical properties of radiating elastic sources established in Theorem~\ref{thm:small}, we analyze the Lam\'e system \eqref{eq:lame}, a complex system of coupled partial differential equations governing vector-valued functions associated with pressure and shear waves. Using Helmholtz decomposition, we decouple the Lam\'e system into two vector-valued Helmholtz equations, enabling the derivation of translation-invariant $L^2$-estimates in Lemma~\ref{lem:ufL2} and the energy estimate \eqref{eq:energy} in Lemma~\ref{lem:uholder} for the scattered displacement $\mathbf{u}$ associated with a non-radiating source $\mathbf{f}$. These estimates depend on the diameter of $\Omega$, the operating frequency $\omega$, and the $L^2$-norm of $\mathbf{f}$ in $\Omega$. Leveraging Lemmas~\ref{lem:ufL2} and~\ref{lem:uholder}, we prove Theorem~\ref{thm:small} using a contradiction argument and the integral identity \eqref{eq:psim}. The geometrical properties of a non-radiating elastic source $\mathbf{f}$, with respect to the diameter of $\Omega$ and the frequency $\omega$, follow directly from Theorem~\ref{thm:small} and are detailed in Corollary~\ref{cor:non-rads}.

For the main contribution (ii), we utilize the so-called complex geometrical optics (CGO) solutions to  give a subtle asymptotic analysis near the high-curvature boundary point with respect to the parameter in CGO solutions. The presence of admissible $K$-curvature points significantly enhances the complexity of the asymptotic analysis. 
 To investigate the radiating properties of elastic sources with admissible $K$-curvature boundary points in Theorem \ref{thm:source rad}, we use a  contradiction argument  that requires analysis of the non-radiating property established in Corollary \ref{cor:non-rad}. 
The  investigation of non-radiating properties in elastic sources requires comprehensive analysis of the global behavior of displacement field.
%To investigate the non-radiating properties of elastic sources, we analyze the global behavior of the displacement field.
We  derive a global regularity result in Lemma \ref{lem:rise reg} and a global $L^2$ estimate in Proposition \ref{prop:ubeta} for the   displacement field with non-radiating source. 
 In Proposition \ref{prop:ubeta}, using geometric techniques, we prove that the Jacobi matrix of a non-radiating displacement field $\EM u$ vanishes at admissible $K$-curvature points.
%In Proposition \ref{prop:ubeta}, we use geometric techniques to prove that the Jacobi matrix of a non-radiating  displacement field $\mathbf{u}$ vanishes at an admissible $K$-curvature point. 
To obtain the aforementioned subtle asymptotic analysis, we employ  the CGO solutions in Lemma \ref{lem:cgo} to establish an   integral identity in Lemma \ref{lem:integral cur}. %and then perform asymptotic analysis aforementioned of non-radiating sources near admissible K-curvature points.  
Through careful selection of the CGO parameter $\tau$, %we derive the quantitative relationship between the intensity of non-radiating sources  and   curvature at the  admissible $K$-curvature point, 
we  prove Theorem \ref{thm:source rad} and Corollary \ref{cor:non-rad}.

%For the main contribution (iii), the analytical approach for the medium scattering problem follows a methodology similar to that of the source case. After direct transformation, the medium scattering problem \eqref{eq:mesca} can be formally converted into the source scattering problem with the source term $-\omega^2 V \EM  u^t$. However, since $-\omega^2 V \EM  u^t$ contains the unknown total field, the calculations and analysis become significantly more complex. Proposition \ref{prop:ust2}  plays a pivotal role in our subsequent analysis by enabling the transition from estimating unknown fields (including both displacement and total fields) to working with the known incident field. acement and total fields) to working with the known incident field.
For the main contribution (iii), our analytical approach to the medium scattering problem follows methodology analogous to the source case. The medium scattering problem \eqref{eq:mesca} can be reformulated as an equivalent source scattering problem. The presence of the unknown total field in source term introduces significant analytical challenges. 
 By combining the well-posedness results for the direct problem (see Proposition \ref{prop:ust2}) with Theorems \ref{thm:small} and \ref{thm:source rad}, we prove Theorems \ref{thm:medsmall} and \ref{thm:medKpoint1}, which     characterize  the radiating properties for mediums. Furthermore, the geometric properties of transmission eigenfunctions established in Theorems \ref{thm:tans1} and \ref{thm:trans2} follow directly from  Corollaries \ref{cor:non-rads} and \ref{cor:non-rad}.   For the main contribution (iv),  by using a contradiction argument incorporating the non-radiating properties for the sources and medium,   we establish both local and global uniqueness results for scatterer shape determination by a single far-field pattern measurement in Theorems \ref{thm:in2loc}--\ref{thm:medinv2}.

The paper is organized as follows.  In Section \ref{sec:2}, 
%we present the precise mathematical relationship between the non-radiating source intensity  and the diameter of the support. 
we derive a precise quantitative relationship between the intensity of a  radiating elastic source and the diameter of its support (in terms of the wavelength). 
%the presents the small support (smallness) results for the elastic source scattering problem.
Section \ref{sec:kpoint} presents  a rigorous mathematical result  when the support of source has high-curvature boundary points.
In section \ref{sec:3}, we extend the findings from sections \ref{sec:2}--\ref{sec:kpoint} to the inhomogeneous medium scattering problem and investigate the geometric properties of transmission eigenfunctions.
In section \ref{sec:inverse}, we establish the local and global uniqueness results for inverse source and medium scattering problems by a single far-field pattern measurement. 
%extends these results to inhomogeneous medium scattering problem and presents the geometric properties of the transmission eigenfunctions and the uniqueness determination results using a single far-field pattern.

\section{Radiating property of elastic sources with small support
%The source with small support must be radiating
}\label{sec:2}
%\subsection{Small support elastic sources must scatter.}

In this section, %we study the geometric properties of non-radiating  sources with small support. 
we shall derive a mathematical relationship between the intensity of a radiating or non-radiating elastic source and the size of its support, implying that a source with a relatively small support compared to its intensity and frequency cannot be non-radiating in $\mathbb{R}^n$ $(n=2,3)$. The specific details are presented in the following theorem.
Before this, we introduce the H\"older seminorm for a vector-valued function  $\boldsymbol\varphi=(\varphi_1 (\mathbf x),\dots,\varphi_n (\mathbf x))^{\top}$ with $\mathbf x \in \Omega$ as follows:
$$[\boldsymbol\varphi]_{\alpha,\Omega}:=\max\{[\varphi_i]_{\alpha,\Omega}\}_{i=1}^n,$$
where
$$[\varphi_i]_{\alpha,\Omega}:=\sup_{\mathbf x\not=\mathbf y,
\mathbf x,\mathbf y\in \Omega}\frac{|\varphi_i(\mathbf x)-\varphi_i(\mathbf y)\vert}{|\mathbf x-\mathbf y|^\alpha}.$$

\begin{thm}\label{thm:small}
Let $\Omega\subset \mathbb R^n$, $n\in\{2,3\}$ be a bounded Lipshcitz domain with a connected complement $\mathbb R^n\setminus {\overline{\Omega}}$. For some    component  $\Omega_c$ of $\Omega$, suppose that $\boldsymbol\varphi\in L^\infty(\mathbb R^n)^n$ satisfies
\begin{equation}\label{eq:phi23}
\begin{cases}
	\boldsymbol\varphi \in C^{\delta}(\overline{\Omega_c})^2, &\delta \in(0,1] ,\\
	\boldsymbol\varphi \in C^{\delta}(\overline{\Omega_c})^3, &\delta \in(0,1/2].
\end{cases}
\end{equation}
 %$\boldsymbol\varphi \in C^{\delta}(\overline{\Omega_c})^n,\delta\in(0,1/2]$. 
 Let $\mathbf u\in H_{loc}^2(\mathbb R^n)^n$ be the unique outgoing solution to \eqref{eq:lame} associated with $\mathbf f=\chi_{\Omega}\boldsymbol\varphi$. Let $\varepsilon=d(\Omega_c)\omega$ be the diameter of $\Omega_c$ in units of $\omega^{-1}$, where $d(\Omega_c)$ is the Euclidean diameter of $\Omega_c$. If 
\begin{equation}\label{eq:small}
\frac{\sup_{\partial \Omega_c}|\boldsymbol\varphi|}{\omega^{-\delta}[\boldsymbol\varphi]_{\delta,\Omega_c}+\|\boldsymbol\varphi\|_{L^\infty(\Omega_c)}}>C\varepsilon^{\delta}(1+(1+\varepsilon)\varepsilon^{n/2}),
\end{equation}
where $C=C(n,\lambda,\mu)$ is a positive constant depending on $n,\lambda,\mu$, then $\mathbf u^{\infty}\not=\mathbf 0$.
\end{thm}

%\begin{rem} The assumption on the regularity assumption for $\boldsymbol\varphi$ in $\overline{\Omega_c} $ states that $ \boldsymbol\varphi \in C^{\delta}(\overline{\Omega_c})^n$ for $n=2$ or $n=3$, where the positive number $\delta\in(0,1/2]$. Indeed, for the two-dimensional case, we can relax that $\delta\in (0,1]$, while for the three-dimensional case, we need that $\delta\in (0,1/2]$. In order \end{rem}

\begin{rem}\label{rem:21}
    It is noted that in Theorem \ref{thm:small},  we emphasize  the diameter of any component $\Omega_c$ of $\Omega$, rather than the diameter of $\Omega$. Here, a component $\Omega_c$ of $\Omega$ means that $\overline {\Omega_c}\cap \overline{\Omega\setminus{\Omega_c}}=\emptyset.$  %By Theorem \ref{thm:small}, we know that if $\Omega$ has a  component where $\boldsymbol \varphi$ and the diameter satisfy \eqref{eq:small}, the far-field pattern cannot be zero. 
   
    %On the other hand, Theorem \ref{thm:small} reveals that  if the intensity of the source is fixed and $\sup|\boldsymbol {\varphi}|$ has a positive lower bound on the boundary, for sufficiently small support, the source cannot be non-radiating. 
    
   % Moreover, given the intensity of the source, i.e. the denominator of \eqref{eq:small} fixed, if the support (or a component of it)  is sufficiently small (in units of $\omega^{-1}$), and $\sup|\boldsymbol {\varphi }|$ has a suitable positive lower bound on the boundary of the support (or a component of it), then the source must be radiating. 
\end{rem}
\begin{rem}\label{rem:radiating}
%It is clear that Theorem \ref{thm:small} reveals that  for the fixed intensity of the source on any component, if the support of the source is sufficiently small, then \eqref{eq:small} holds nuturally, which indicates that  the source must be radiating. 
Theorem \ref{thm:small} reveals that, for a fixed source intensity on any component, when the support of the source is sufficiently small, \eqref{eq:small} holds naturally, indicating that the source must be radiating.	
\end{rem}

Corollary~\ref{cor:non-rads} follows immediately from Theorem~\ref{thm:small}, demonstrating that the non-radiating source intensity with frequency   admits an upper bound proportional to the scatterer diameter.
\begin{cor}\label{cor:non-rads}
Under the same setup as in Theorem \ref{thm:small}, if the source $\EM f=\chi_{\Omega}\boldsymbol{{\varphi}}$ is non-radiating, then it follows that
	\begin{equation}\label{eq:non-small}
\frac{\sup_{\partial \Omega_c}|\boldsymbol\varphi|}{\omega^{-\delta}[\boldsymbol\varphi]_{\delta,\Omega_c}+\|\boldsymbol\varphi\|_{L^\infty(\Omega_c)}}\leq C\varepsilon^{\delta}(1+(1+\varepsilon)\varepsilon^{n/2}),
\end{equation}
where $C$ is a positive constant independent of $\varepsilon$ and $\omega$.
\end{cor}

%By Theorem \ref{thm:small}, we know that the radiation must occur if the support of the source is sufficiently small. This enables us to investigate the size of the support of non-radiating sources.  Corollary \ref{cor1} shows that the diameter of any component of a non-radiating source cannot be sufficiently small, i.e., must have a positive lower bound (depending on the source intensity).

According to Theorem \ref{thm:small}, the radiation must occur if the support of a source is sufficiently small. This allows us to investigate the size of the support of a non-radiating source. Corollary \ref{cor1} demonstrates that the diameter of any component of the support of a non-radiating source cannot be arbitrarily small; instead, it must have a positive lower bound that depends on the source intensity.

%demonstrates that  the size of any component of a non-radiating source must have a positive lower bound depending on  its intensity.

\begin{cor}\label{cor1}
Assume that the source $\mathbf f=\chi_{\Omega}\boldsymbol \varphi $ is non-radiating, where $\boldsymbol{\varphi}$ is given in \eqref{eq:phi23}. It holds that
%the diameter of any component $\Omega_c$ of $\Omega$  must have a positive lower bound associated with its intensity:
\begin{equation}\label{eq:non-rad}
	{\rm diam}(\Omega_c)\geq 
\min \left (1, \left ( {c_{n,\lambda,\mu}} \frac{\sup_{\partial \Omega_c}|\boldsymbol\varphi|}{\omega^{-\delta}[\boldsymbol\varphi]_{\delta,\Omega_c}+\|\boldsymbol\varphi\|_{L^\infty(\Omega_c)}} \right)^{\frac{1}{\delta}}\right)\omega^{-1},
\end{equation}
where  $\Omega_c$  is  any component of $\Omega$ and $c_{n,\lambda,\mu}$ is a positive constant only depending on $n,\lambda,\mu.$
\end{cor}

\begin{proof}
Since the source is non-radiating, \eqref{eq:non-small} directly follows from Corollary \ref{cor:non-rads}. If $d(\Omega_c) < \omega^{-1} $, then it follows that $ \varepsilon < 1 $, and
 \begin{equation}\label{eq:non-rad2}
\frac{\sup_{\partial \Omega_c}|\boldsymbol\varphi|}{\omega^{-\delta}[\boldsymbol\varphi]_{\delta,\Omega_c}+\|\boldsymbol\varphi\|_{L^\infty(\Omega_c)}}<c_{n,\lambda,\mu}\varepsilon^\delta. 
\end{equation}
Combining \eqref{eq:non-small} with \eqref{eq:non-rad2}, it is clear to obtain \eqref{eq:non-rad}. Therefore, the proof is complete.
\end{proof}
%\Blu {Lemma \ref{lem:omegac} plays a crucial role on proving Theorem \ref{thm:small}, which demonstrates that $\Omega$ is non-radiating if and only if  any component $\Omega_c$ of $\Omega$ is non-radiating. }
%\begin{lem}\label{lem:omegac}
%Let $\Omega\subset \mathbb R^n,n=2,3$ be a bounded Lipschitz domain with  a connected complement  and $\Omega_c$ be any component of $\Omega$. Let $\boldsymbol \varphi\in L^\infty(\mathbb R^n)^n,\mathbf u\in H^2_{loc}(\mathbb R^n)^n$ and $\mathbf u_c\in H^2_{loc}(\mathbb R^n)^n $ satisfy that $\mathcal L\mathbf u+\omega^2\mathbf u=\chi_{\Omega}\boldsymbol \varphi$ and $\mathcal L \mathbf u+\omega^2\mathbf u=\chi_{\Omega_c}\boldsymbol\varphi$. Let $\mathbf u^{\infty}$ and $\mathbf u_c^\infty$ be the far-field patterns corresponding to $\Omega$ and $\Omega_c$, respectively.

%If $\mathbf u$ and $\mathbf u_c$ satisfy  Sommerfeld radiation conditions, then $\mathbf u^\infty=\mathbf 0$ if and only if $\mathbf u_c^\infty=(\mathbf u-\mathbf u_c)^\infty =\mathbf 0.$ 
%\end{lem}

Before proving Theorem \ref{thm:small}, we introduce the following critical lemmas that provide essential a-priori estimates. Lemma \ref{lem:ufL2} establishes a translation-invariant $L^2$ estimate for the Lam\'e system, analogous to those for the Laplacian operator in \cite{BS}.
\begin{lem}\label{lem:ufL2}
    Let $\Omega\subset \mathbb R^n$, $n\in \{2,3\}$ be a bounded Lipschitz domain. 
Consider the scattering problem \eqref{eq:lame} with the source term $\EM f=\chi_{\Omega}\boldsymbol{\varphi}$, where $\boldsymbol {\varphi}\in L^2(\mathbb R^n)^n$. Then one has the estimate as follows:
\begin{align}
    \|\EM u\|_{L^2(\Omega)^n}\leq C_{n,\lambda,\mu}d(\Omega)\omega^{-1}\|\EM f\|_{L^2(\Omega)^n}.\label{eq:uf}
\end{align}
\end{lem}
\begin{proof}
Since  $\EM u$ satisfies the decomposition $\EM u=\EM u_p+\EM u_s$  and  \eqref{eq:deu}, it follows from \cite[Lemma 2.1]{BS} that the following a-priori estimate holds for  $\|\mathbf u_\iota\|_{L^2(\Omega)}$ ($\iota=p,s$): %in the case where $\Omega_s=\Omega_r=\Omega$
\begin{equation}\label{eq:uf1}
\|\mathbf u_\iota\|_{L^2(\Omega)^n}\leq C_n\kappa_\iota^{-1}d(\Omega)\|\mathbf f_\iota\|_{L^2(\Omega)^n}.
\end{equation}
Recalling that $\kappa_p$ and $\kappa_s$ are defined by \eqref{eq:wave number}, it yields that
\begin{align}
\|\EM  u\|_{L^2(\Omega)^n}&\leq \|\EM u_p\|_{L^2(\Omega)^n}+\|\EM u_s\|_{L^2(\Omega)^n}\leq C_nd(\Omega)(\kappa_p^{-1}\|\mathbf{f}_{{p}}\|_{L^2(\Omega)^n}+\kappa_s^{-1}\|\mathbf{f}_{{s}}\|_{L^2(\Omega)^n})\notag\\
&\leq C_{n,\lambda,\mu}d(\Omega)\omega^{-1} (\|\mathbf{f}_{{p}}\|_{L^2(\Omega)^n}+\|\mathbf{f}_{{s}}\|_{L^2(\Omega)^n}).\notag
\end{align}
Moreover,  the Helmholtz decomposition \eqref{eq:dehelm} for $\mathbf{f}$ in $L^2$ satisfies $\|\mathbf{f}_{{p}}\|_{L^2(\Omega)^n} + \|\mathbf{f}_{{s}}\|_{L^2(\Omega)^n} \leq 2\|\mathbf{f}\|_{L^2(\Omega)^n}$ (see \cite{FKS}). Therefore, we obtain \eqref{eq:uf}.

The proof is complete.
\end{proof}

In the following lemma, we consider a non-radiating source $  \mathbf{f}  $, where the scattered displacement $  \mathbf{u}  $ associated with $  \mathbf{f}  $ in the Lam\'e system \eqref{eq:lame} satisfies $  \mathbf{u} \equiv \mathbf{0}  $ in $  \mathbb{R}^n \setminus \Omega  $. We establish a H\"older regularity result for $  \mathbf{u}  $ in $  \mathbb{R}^n  $. Additionally, we derive a relationship between the global energy estimate \eqref{eq:energy} and the diameter of $  \Omega  $, the operating frequency $  \omega  $ and  the $  L^2  $-norm of $  \mathbf{f}  $ in $  \Omega  $, where the energy is expressed in terms of the $  L^\infty  $-norm and a H\"older-type seminorm of $  \mathbf{u}  $.

\begin{lem}\label{lem:uholder}
Let $\Omega\subset \mathbb R^n$, $n\in \{2,3\}$ be a bounded Lipshitz domain with a connected complement. Let $\mathbf u\in L_{loc}^2(\mathbb R^n)^n$ be an outgoing solution to $\mathcal L\mathbf u+\omega^2\mathbf u=\EM f$, where $\EM f=\chi_\Omega\boldsymbol\varphi$ with $\boldsymbol\varphi\in L^2(\mathbb R^n)^n$. If $\mathbf u=\mathbf 0$ in $\mathbb R^n\setminus \overline{\Omega}$, then one has $\mathbf u\in C^{\delta}(\mathbb R^n)^n$ and
\begin{equation}\label{eq:energy}
\|\mathbf u\|_{L^\infty(\mathbb R^n)^n}+\omega^{-\delta}[\mathbf u]_{\delta,\mathbb R^n}\leq C_{n,\lambda,\mu}\omega^{\frac{n}{2}-1}(\omega^{-1}+d(\Omega))\|\EM f\|_{L^2(\Omega)^n},
\end{equation}
   {where $\delta\in(0,1]$ for $n=2$,  and $\delta\in (0,1/2]$ for $n=3$}, and $C_{n,\lambda,\mu}\in \mathbb R_+$ is  a  positive constant only depending on $n,\lambda$ and $\mu$.
\end{lem}
\begin{proof}

Utilizing   Helmholtz decomposition, one can find that $\mathbf{f}=\chi_{\Omega}\boldsymbol{\varphi}(\EM x)$ and $\mathbf{u}$  satisfy \eqref{eq:dehelm} and \eqref{eq:deu}.
%\begin{equation}\notag
%\mathbf f=\mathbf f_{\mathbf p} +\mathbf f_{\mathbf s},
%\end{equation}
%where $\mathbf f_{\mathbf p}$ and $\mathbf f_{\mathbf s}$ fulfill that
%\begin{equation}\notag
%\begin{cases}
%\Delta \mathbf u_{\mathbf p}+\kappa_p^2\mathbf u_{\mathbf p}=\mathbf f_{\mathbf p},\\
%\Delta \mathbf u_{\mathbf s}+\kappa_s^2\mathbf u_{\mathbf s}=\mathbf f_{\mathbf s},
%\end{cases}
%\end{equation}
%with
%\begin{equation}\notag
%\nabla \times \mathbf f_{\mathbf p}=\mathbf 0,\  \nabla \cdot \mathbf f_{\mathbf s}=0.
%\end{equation}
Denote
\begin{align}\notag
%\mathbf f(\mathbf x)&=\chi_{\Omega}\boldsymbol\varphi(\mathbf x),\notag\\
\mathbf U_{\iota}(\mathbf y)=\mathbf u_{\iota}(\mathbf y/\kappa_{\iota}),\quad
\mathbf F_{\iota}(\mathbf y)=\mathbf f_{\iota}(\mathbf y/\kappa_{\iota}),\quad
\Omega_{\kappa_{\iota}}=\{\mathbf y\in \mathbb R^n|\mathbf y/\kappa_{\iota}\in \Omega\},\notag
\end{align}
where and also in what follows $\iota\in \{ p, s\}$. Then one has $\Delta \mathbf U_{\iota}+\mathbf U_{\iota}=\mathbf F_{\iota}$ in $\mathbb R^n$ with $\mathbf U_{\iota} \in L_{loc}^2(\mathbb R^n)^n$, $\mathbf F_{\iota}\in L^2(\mathbb R^n)^n$ and $\mathbf U_{\iota}=\mathbf F_{\iota}=\mathbf 0$ in $\mathbb R^n\setminus\overline{\Omega_{\kappa_{\iota}}}$.

By virtue of  Fourier transform, we obtain $(-|\xi|^2+1)\mathbf {\hat U}_{\iota}(\xi)=\mathbf{ \hat F}_{\iota}(\xi)$ for $\xi \in \mathbb R^n$. For given $\mathbf U_{\iota}=\mathbf 0$ in $\mathbb R^n\setminus \overline{\Omega_\iota}$, it implies that $\mathbf U_{\iota}$ belongs to $L^2(\mathbb R^n)^n$. % rather than  $\mathbf U_{\iota}\in L_{loc}^2(\mathbb R^n)^n $. 
Applying the Plancherel theorem, one can derive that each component of $\mathbf U_{\iota}$ and $\mathbf F_{\iota}$, denoted by $ U_{\iota_{i}}$ and $F_{\iota_i}$ satisfy that $U_{\iota_{i}}\in H^2(\mathbb R^n)$ $(i=1,\dots,n)$ and 
\begin{align}
\|U_{\iota_{i}}\|_{H^2(\mathbb R^n)}=&\left\| (1+|\cdot|^2) {\hat{ U}}_{\iota_{i}}\right\|_{L^2(\mathbb R^n)}=\left\|2{\hat U}_{\iota_{i}}- {\hat F}_{\iota_{i}}\right\|_{L^2(\mathbb R^n)}\notag\\
&\leq \| F_{\iota_{i}}\|_{L^2(\mathbb R^n)}+2\| U_{\iota_{i}}\|_{L^2(\mathbb R^n)}.\notag
\end{align}
 Hence, it yields that
\begin{align}\notag
\|\mathbf U_{\iota}\|_{H^2(\mathbb R^n)^n}&=\left(\sum_{i=1}^n\|   U_{\iota_{i}} \|_{H^2(\mathbb R^n)^n}^2\right)^{\frac{1}{2}}\leq \left( \sum_{i=1}^n\left( 2\|  U_{\iota_i}\|_{L^2(\mathbb R^n) }+\|  F_{\iota_i}\|_{L^2(\mathbb R^n) }\right)^2\right)^{\frac{1}{2}}\notag\\
&\leq \sqrt{2}(2\|\mathbf U_{\iota}\|_{L^2(\mathbb R^n)^n}+\|\mathbf F_{\iota}\|_{L^2(\mathbb R^n)^n}).\notag
\end{align}

By virtue of Sobolev embedding theorem, we know that there exists  a positive constant $C_n$ such that $H^2(\mathbb{R}^n)^n \hookrightarrow C^{\delta}(\mathbb{R}^n)^n$, { $\delta\in (0,1]$ for $n=2$ and $\delta\in (0,1/2]$ for $n=3$} with
$\|\mathbf U_{\iota}\|_{C^{\delta}}\leq C_n\|\mathbf U_{\iota}\|_{H^2}$.
Therefore,  one has
$$\|\mathbf U_{\iota}\|_{C^{\delta}(\mathbb R^n)^n}\leq C_n(2\|\mathbf U_\iota\|_{L^2(\Omega_{\kappa_{\iota}})^n}+\|\mathbf F_{\iota}\|_{L^2(\Omega_{\kappa_{\iota}})^n}).$$
Returning to the non-scaled variable $\mathbf{x}$, one can demonstrate that:
\begin{equation}\label{eq:1}
\|\mathbf u_\iota\|_{L^\infty(\mathbb R^n)^n}+\kappa_{\iota}^{-\delta}[\mathbf u_\iota]_{\delta,\mathbb R^n}\leq C_n\kappa_{\iota}^{n/2}(\kappa_{\iota}^{-2}\|\mathbf f_{\iota}\|_{L^2(\Omega)^n}+2\|\mathbf u_\iota\|_{L^2(\Omega)^n}).
\end{equation}
%Moreover, we have the following  a-priori estimate for $\|\mathbf u_\iota\|_{L^2(\Omega)}$ for the case where $\Omega_s=\Omega_r=\Omega$(\cite[Lemma 2.1]{BS}):
%\begin{equation}\label{eq:2}
%\|\mathbf u_\iota\|_{L^2(\Omega)^n}\leq C_n\kappa_\iota^{-1}d(\Omega)\|\mathbf %f_\iota\|_{L^2(\Omega)^n}.
%\end{equation}
Recalling that $\iota\in\{p,s\}$, then combining \eqref{eq:1} with \eqref{eq:uf1}, one obtains
\begin{align}
\|\mathbf u_{ p}\|_{L^\infty(\mathbb R^n)^n}+&\|\mathbf u_{ s}\|_{L^\infty(\mathbb R^n)^n}+\kappa_{ p}^{-\delta}[\mathbf u_{ p}]_{\delta,\mathbb R^n}+\kappa_{ s}^{-\delta}[\mathbf u_{ s}]_{\delta,\mathbb R^n}\notag \\
&\leq C_{n}
%\omega^{\frac{n}{2}-1}(\omega^{-1}+d(\Omega))\|\mathbf f\|_{L^2(\Omega)^n}
\bigg[(\kappa_{ p}^{\frac{n}{2}-1}(\kappa_{ p}^{-1}+d(\Omega))\|\mathbf f_{ p}\|_{L^2(\Omega)^n}\notag\\
&+\kappa_{ s}^{\frac{n}{2}-1}(\kappa_{ s}^{-1}+d(\Omega))\|\mathbf f_{s}\|_{L^2(\Omega)^n}\bigg].\label{eq:kpsf}
\end{align}
Following the proof argument of Lemma \ref{lem:ufL2}, we can simplify \eqref{eq:kpsf} to
\begin{align}
    \|\mathbf u_{ p}\|_{L^\infty(\mathbb R^n)^n}&+\|\mathbf u_{ s}\|_{L^\infty(\mathbb R^n)^n}+\omega^{-\delta}[\mathbf u_{ p}]_{\delta,\mathbb R^n}+\omega^{-\delta}[\mathbf u_{ s}]_{\delta,\mathbb R^n}\notag\\
&\leq C_{n,\lambda,\mu}
\omega^{\frac{n}{2}-1}(\omega^{-1}+d(\Omega))\|\mathbf f\|_{L^2(\Omega)^n}\notag.
%\left(\kappa_{ p}^{\frac{n}{2}-1}(\kappa_{ p}^{-1}+d(\Omega))\|\mathbf f_{ p}\|_{L^2(\Omega)^n}+\kappa_{ s}^{\frac{n}{2}-1}(\kappa_{ s}^{-1}+d(\Omega))\|\mathbf f_{s}\|_{L^2(\Omega)^n}\right).\label{eq:kpsf}
\end{align}
Recalling that $\mathbf u=\mathbf u_p+\mathbf u_s$, one observes that  $\|\mathbf u\|_{L^\infty(\mathbb R^n)^n}\leq \|\mathbf u_p\|_{L^\infty(\mathbb R^n)^n}+\|\mathbf u_s\|_{L^\infty(\mathbb R^n)^n}$ and $[\mathbf u]_{\delta,\mathbb R^n}\leq [\mathbf u_p]_{\delta,\mathbb R^n}+[\mathbf u_s]_{\delta,\mathbb R^n}$.  
%By the strong convexity condition $n\lambda+2\mu>0$, it is clear that $\kappa_s>\kappa_p$ and $\kappa_s^{-1}<\kappa_p^{-1}$.

The proof is complete.
\end{proof}

%The next key lemma is  called Betti's formula for the L\'ame operator.
%\begin{lem}\cite[Lemma 2.5]{DLS2021}\label{lem:lame green}
%Suppose that $\Omega\subset \mathbb R^n$ is a bounded Lipschitz domain, let $\mathbf w,\mathbf v\in H^2(\Omega)$ satisfy $\mathcal L\mathbf u\in L^2(\Omega),\mathcal L \mathbf v\in L^2(\Omega),$ then one has
%\begin{equation}\notag
%\int_{\Omega} \mathbf u\cdot (\mathcal L \mathbf v)-\mathbf v\cdot(\mathcal L\mathbf u)\mathrm d\mathbf x=\int_{\partial \Omega}\mathbf u\cdot (T_\nu \mathbf v)-\mathbf v\cdot (T_\nu \mathbf u)\mathrm d\sigma,
%\end{equation}
%particularly, if $\mathbf u=T_\nu \mathbf u=\mathbf 0 \ \mbox{on} \ \Gamma \subset \partial \Omega,$ then
%\begin{equation}\notag
%\int_{\Omega} \mathbf u\cdot(\mathcal L \mathbf v)-\mathbf v\cdot( \mathcal L\mathbf u)\mathrm d\mathbf x=\int_{\partial \Omega\setminus \Gamma} \mathbf u\cdot(T_\nu \mathbf v)-\mathbf v\cdot(T_\nu \mathbf u)\mathrm d\sigma.
%\end{equation}
%\end{lem}

Now we present the proof of Theorem \ref{thm:small}.
\begin{proof}[Proof of Theorem \ref{thm:small}]
We first prove the theorem with the case of  $\Omega$ having  a single component, i.e. $\Omega=\Omega_c.$
We will prove the theorem by contradiction. Assume that $\mathbf u^\infty\equiv \mathbf 0$. By Rellich's Lemma, it is evident that  $\mathbf u=\mathbf 0\in \mathbb R^n\setminus\overline{\Omega},\ n\in\{2,3\}$. Therefore, for given that $\mathbf{u} \in H_{\text{loc}}^2(\Omega)^n$, we can conclude that $\mathbf{u} \in H_0^2(\Omega)^n$.
Let  $\boldsymbol \psi := \boldsymbol\varphi|_{\Omega} - \omega^2\mathbf{u}|_{\Omega}$ and  $\mathbf{p} \in \partial \Omega$. Since $\mathbf{u}=\mathbf 0$ on $\partial \Omega$, we derive that
\begin{align}
	\boldsymbol \psi(\mathbf p)m(\Omega)=\boldsymbol\varphi(\mathbf p)m(\Omega)=\int_{\Omega}\boldsymbol\psi(\mathbf p)\mathrm d\mathbf x,\label{eq:psim}
\end{align}
where $m(\Omega)$  denotes  the measure of  $\Omega$.
Recalling that $\mathbf{u} \in H^2_0(\Omega)^n$ implies $T_\nu \mathbf{u} = \mathbf 0$ on $\partial \Omega$. Therefore,  we obtain
\begin{equation}\label{eq:psi0}
\int_{\Omega}\boldsymbol\psi(\mathbf x)\cdot \mathbf e_{i} \mathrm d\mathbf x=\int_{\Omega}(\boldsymbol \varphi-\omega^2\mathbf u)\cdot \mathbf e_i \mathrm d\mathbf x=\int_{\Omega}\mathcal L\mathbf u\cdot \mathbf e_i\mathrm d\mathbf x=\int_{\partial \Omega}\{T_{\nu}\mathbf u\}_i\mathrm d \sigma=  0,
\end{equation}
where $\mathbf e_i=(0,\dots,1,\dots,0)^\top\ (i=1,\dots,n)$ is the orthogonal basis of $\mathbb R^n$ and $\{T_\nu \mathbf u\}_i$ is the $i-th$ component of $T_\nu \mathbf u.$
From \eqref{eq:psi0}, we can deduce that $\int_{\Omega}\boldsymbol \psi(\mathbf x)\mathrm d\mathbf x=\mathbf 0$ since each   $\int_{\Omega} \psi_i(\mathbf x)\mathrm d\mathbf x=0$, where $\boldsymbol \psi=(\psi_1,\dots,\psi_n)^\top.$
It can directly show that
$$|\boldsymbol\varphi(\mathbf p)m(\Omega)|=\left|\int_{\Omega}(\boldsymbol\psi(\mathbf p)-\boldsymbol\psi(\mathbf x))\mathrm d\mathbf x\right|\leq[\boldsymbol\psi]_{\delta,\Omega}\int_{\Omega}|\mathbf p-\mathbf x|^{\delta}\mathrm d\mathbf x\leq[\boldsymbol\psi]_{\delta,\Omega}m(\Omega)(d(\Omega))^{\delta}.$$
By virtue of  Lemma \ref{lem:uholder} and with the fact that $\|\boldsymbol\varphi\|_{L^2(\Omega)^n}\leq\sqrt{m(\Omega)}\|\boldsymbol\varphi\|_{L^\infty(\Omega)^n}$, we can calculate that
\begin{align}
|\boldsymbol\varphi(\mathbf p)|&\leq (d(\Omega))^{\delta}[\boldsymbol\psi]_{\delta,\Omega}\leq(d(\Omega))^{\delta}([\boldsymbol\varphi]_{\delta,\Omega}+\omega^2[\mathbf u]_{\delta,\Omega})\notag\\
&\leq (d(\Omega))^{\delta}\left([\boldsymbol\varphi]_{\delta,\Omega}+C_{n,\lambda,\mu}\omega^2\omega^{\frac{n}{2}+\delta-1}(\omega^{-1}+d(\Omega))\sqrt{m(\Omega)}\|\boldsymbol\varphi\|_{L^\infty(\Omega)^n}\right).\notag
\end{align}
Since $\Omega$ is bounded, it can be enclosed within a sphere of radius $d(\Omega)$, implying that $m(\Omega)\leq C_n(d(\Omega))^n=C_n\varepsilon^n\omega^{-n}$. Therefore, we can deduce that  
\begin{align}
|\boldsymbol \varphi(\mathbf p)|
%&\leq (\varepsilon\omega^{-1})^{\delta}\left( [\boldsymbol\varphi]_{\delta,\Omega}+ C_{n,\lambda,\mu}\varepsilon^{\frac{n}{2}}\omega^{-\delta}(\sqrt{\mu})^{\frac{n}{2}+\delta-1}(\sqrt{\lambda+2\mu}+\varepsilon)\|\boldsymbol\varphi\|_{L^\infty(\Omega)^n}\right)\notag\\
&\leq C_{n,\lambda,\mu}(\varepsilon\omega^{-1})^{\delta}\left( [\boldsymbol\varphi]_{\delta,\Omega}+ \omega^2\omega^{\frac{n}{2}+\delta-1} (\omega^{-1}+\varepsilon\omega^{-1})\varepsilon^{\frac{n}{2}}\omega^{-\frac{n}{2}}\|\boldsymbol\varphi\|_{L^\infty(\Omega)^n}\right)\notag\\
%&\leq C_{n,\lambda,\mu}\varepsilon^\delta\left (\omega^{-\delta}[\boldsymbol \varphi]_{\delta,\Omega}+\varepsilon^{\frac{n}{2}}(1+\varepsilon)\|\boldsymbol \varphi \|_{L^\infty(\Omega)^n}\right)\notag\\
&\leq C_{n,\lambda,\mu}\varepsilon^\delta(1+\varepsilon^{\frac{n}{2}}(1+\varepsilon))\left(\omega^{-\delta}[\boldsymbol \varphi]_{\delta,\Omega}+\| \boldsymbol \varphi\|_{L^\infty(\Omega)^n}\right)\notag.
\end{align}
By taking a supremum over $\mathbf p\in \partial \Omega$, we obtain a contradiction with \eqref{eq:small} in Theorem \ref{thm:small}. 

Next, we are going to prove the case of $\Omega$ having multiple components. By contradiction, assume that $\mathbf u^\infty=\mathbf 0.$ For  $\mathbf u_c\in H^2_{loc}(\mathbb R^n)^n$ satisfying $\mathcal L\mathbf u_c+\omega^2\mathbf u_c=\chi_{\Omega_c}\boldsymbol \varphi$, it is easily seen that $\mathbf u^\infty =\mathbf 0$ if and only if $\mathbf u_c^\infty =(\mathbf u-\mathbf u_c)^\infty =\mathbf 0$.  This follows from the fact that  the complement $\mathbb R^n\setminus \overline \Omega$ of $\Omega$ is connected, implying that $\mathbf u_c\in H^2_{0}(\Omega_c)$ 
%\Blu {We first show a crucial conclusion that  $\mathbf u^\infty =\mathbf 0$ iff $\mathbf u_c^\infty =(\mathbf u-\mathbf u_c^\infty )^\infty=\mathbf 0.$ If $\mathbf u_c^\infty =(\mathbf u-\mathbf u_c^\infty )^\infty=\mathbf 0$, then $\mathbf u^\infty =\mathbf 0$ trivially.} 
Therefore, applying a similar argument regarding the single component  to $\mathbf u_c$ and $\Omega_c$, one has 
$$ 
\frac{\sup_{\partial \Omega_c}|\boldsymbol\varphi|}{\omega^{-\delta}[\boldsymbol\varphi]_{\delta,\Omega_c}+\|\boldsymbol\varphi\|_{L^\infty(\Omega_c)}}\leq C\varepsilon^{\delta}(1+(1+\varepsilon)\varepsilon^{n/2}),
$$
which contradicts \eqref{eq:small}. 

The proof is complete.
\end{proof}

\section{Radiating property of elastic sources  with admissible $K$-curvature points}\label{sec:kpoint}

In this section, we explore the geometric characterization of radiating and non-radiating elastic sources at admissible $ K $-curvature points. These points are defined in Definition \ref{def:ad-hi-cur}, as introduced by \cite{BL2021}, and are fundamental to our subsequent analysis. A schematic illustration is provided in Figure \ref{fig:k}. Throughout this paper, we refer to admissible $ K $-curvature points with sufficiently large curvature as high-curvature points.

\begin{defn}\label{def:ad-hi-cur}
Let $\Omega$ be a bounded domain in $\mathbb R^n$ $(n=2,3)$ with a connected complement and $K,L,M,\varsigma$ be positive constants. A fixed point $\mathbf q\in \partial \Omega$ is termed an admissible $K$-curvature point with parameters $K,L,M,\varsigma$ if the following conditions are satisfied:
\begin{itemize}
\item[(a)]Under the translation invariance, the point $\mathbf q$ is the origin $\mathbf x=\mathbf 0$ and $\mathbf e_n=(0,\dots,0,1)$ is the interior unit normal vector to $\partial \Omega$ at $\mathbf 0;$
\item[(b)]Set $\rho=\sqrt{M}/K$ and $b=1/K$. There is a $C^3$-function $\gamma:\ B(\mathbf 0,\rho) \to \mathbb R_{+}\cup\{\mathbf 0\}$, where  $B(\mathbf 0,\rho)\subset \mathbb R^{n-1}$ is the disc of dimension $n-1$ with radius $\rho$ and center $\mathbf 0$ such that if
    \begin{equation}\label{eq:orho h}
    \Omega_{\rho,b}=B(\mathbf 0,\rho)\times (-b,b)\cap \Omega,
    \end{equation}
    then
    \begin{equation}\notag
    \Omega_{\rho,b}=\{ \mathbf x\in \mathbb R^n\vert \vert \mathbf x^{'}\vert <\rho,\ -b<x_n<b,\ \gamma(\mathbf x^{'})<x_n<b \},
    \end{equation}
    where $\mathbf x^{'}=(x_1,\dots,x_{n-1}),\ \mathbf x=(\mathbf x^{'},x_n)$.
\item[(c)]The function $\gamma$ satisfies
\begin{equation}\notag
\gamma(\mathbf x^{'})=K\vert \mathbf x^{'}\vert ^2+\mathcal O(\vert \mathbf x^{'}\vert^3),\ \mathbf x^{'}\in B(\mathbf 0,\rho).
\end{equation}
\item[(d)] The parameters $M\geq1$ and $0<K_{-}\leq K\leq K_{+}<\infty$ such that
\begin{align}
&K_{-}\vert \mathbf x^{'}\vert ^2 \leq \gamma(\mathbf x^{'})\leq K_{+}\vert \mathbf x^{'}\vert ^2,\quad\vert \mathbf x^{'}\vert <\rho,\notag\\
&M^{-1}\leq \frac{K_{\pm}}{K}\leq M,\quad K_{+}-K_{-}\leq LK^{1-\varsigma}\label{eq: K_+}.
\end{align}
\item[(e)] The intersection $D=\overline{\Omega_{\rho,b}}\cap (\mathbb R^{n-1}\times \{b\})$ is a Lipschitz domain.
\end{itemize}
\end{defn}

\begin{figure}[ht!]
\centering
\includegraphics[width=3cm]{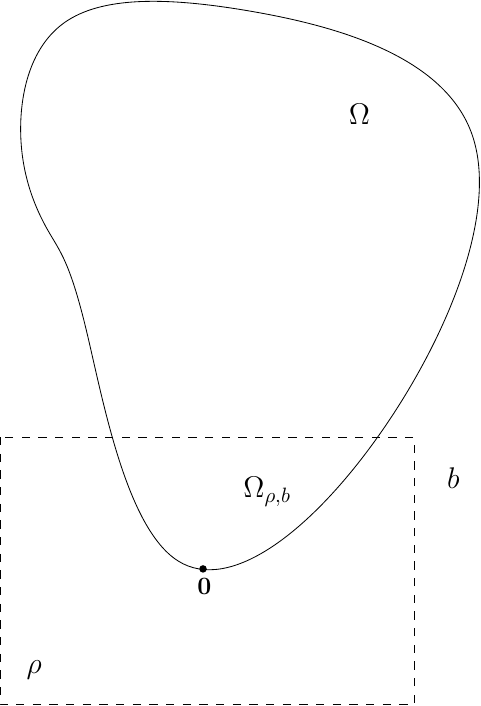}
\caption{The schematic illustration of the the boundary neighborhood $\Omega_{\rho,b}$ of an admissible $K$-curvature point.}
\label{fig:k}
\end{figure}

Theorem~\ref{thm:source rad} establishes a quantitative relationship between the intensity of a radiating elastic source at an admissible $K$-curvature point and the curvature, demonstrating that radiation necessarily occurs for scatterers containing high-curvature points.

%More precisely,   Theorems~\ref{thm:source rad}  derives a  precise relationship between source intensity at admissible $K$-curvature points and the curvature parameter $K$.   %Theorem~\ref{thm:sta} rigorously proves that for scatterers  with  admissible $K$-curvature boundary points, the corresponding far-field pattern admits a positive lower bound  determined by both $K$ and the source intensity. The proofs of Theorems \ref{thm:source rad} and \ref{thm:sta} are deferred to Subsections \ref{sub:proof1} and \ref{sub:proof2}, respectively.

\begin{thm}\label{thm:source rad}
Let $\Omega\subset\mathbb R^n$ $(n=2,3)$ be a bounded domain with diameter at most $D$ and  with  a connected complement. Suppose that $\mathbf q\in \partial \Omega$ is an admissible $K$-curvature point with parameters $K,L,M,\varsigma$ and $K$ is sufficiently large ($K\geq e$). Consider the source scattering problem \eqref{eq:lame} with the source term  $\mathbf f=\chi_\Omega \boldsymbol\varphi,\ \boldsymbol\varphi \in C^\alpha (\overline{\Omega})^n\cap H^1(\Omega)^n$, where $\alpha \in (0,1)$ for $n=2$ and $\alpha \in (1/3, 1)$ for $n=3$. Then for any   frequency $\omega\in \mathbb R_{+}$, if there holds that
\begin{equation}\label{eq:f1}
\frac{\vert \boldsymbol \varphi(\mathbf q)\vert}{\max(1,\|\boldsymbol{\varphi}\|_{C^\alpha(\overline{\Omega})^n},\|\boldsymbol{\varphi}\|_{H^1(\Omega)^n})} \geq \mathcal C(\ln K)^{\frac{n+1}{2}}K^{-\frac{1}{2}\min(\alpha,\varsigma)} \ \mbox{for}\ n=2
\end{equation}
or 
\begin{align}\label{eq:f2}
\frac{\vert \boldsymbol \varphi(\mathbf q)\vert}{\max(1,\|\boldsymbol{\varphi}\|_{C^\alpha(\overline{\Omega})^n},\|\boldsymbol{\varphi}\|_{H^1(\Omega)^n})}  \geq \mathcal C(\ln K)^{\frac{n+1}{2}}K^{-\frac{1}{2}\min(\alpha,\varsigma)+\frac{1}{6}}\  \mbox{for}\ n=3,
\end{align}
where $\min(\alpha,\varsigma)\in(1/3,1) \  \mbox{for}\ n=3 $,    $\mathcal C=\mathcal C(n,D,M,L,\alpha,\lambda,\mu,\varsigma,\omega)\in \mathbb R_{+}$ is  a  constant not depending on $K$. 
Then $(\Omega;\mathbf f)$ cannot be non-radiating.  
\end{thm}
\begin{rem}
% The assumption $\min(\alpha,\varsigma)\in(1/3,1)$ for $n=3$ is to  guarantee the exponent of the curvature $K$ is negative such that the right-hand side of \eqref{eq:f2} decays with respect to the curvature $K$.
 The condition $\min(\alpha,\varsigma) \in (1/3,1)$ for $n=3$ ensures the curvature exponent in \eqref{eq:f2} is negative, thereby guaranteeing a decay property  of the right-hand side when  $K$ is increasing.
\end{rem}

To prove Theorem \ref{thm:source rad}, we employ a contradiction argument that requires analysis of the geometric properties of non-radiating sources, as established in Corollary \ref{cor:non-rad}. In the subsequent analysis, we focus on key results regarding non-radiating sources.

%we only need to prove Corollary \ref{cor:non-rad}. Therefore, in what following, we concentrate on several results related to a non-radiating source $\mathbf f$ in Lemma !!!~~

\begin{cor}\label{cor:non-rad} 
Let $n,D,L,M,\mathcal G,\alpha,\lambda,\mu,\varsigma$ and $\omega$ be a-priori positive constants.  
 Let $\mathbf q\in \partial \Omega$ be an admissible $K$-curvature point as defined in Definition \ref{def:ad-hi-cur} and the source term   $\mathbf f=\chi_{\Omega}\boldsymbol \varphi$ with $ \boldsymbol \varphi \in C^\alpha(\overline{\Omega})^n\cap H^1(\Omega)^n$, where $\alpha \in (0,1)$ for $n=2$ and $\alpha \in (1/3,1)$ for $n=3$. Suppose $\max(\|\boldsymbol{\varphi}\|_{C^{\alpha}(\overline{\Omega})},\boldsymbol{\varphi}\|_{H^1( \Omega)})\leq \mathcal G$. If the source is non-radiating, then there holds that  
 \begin{equation}\notag %\label{eq:k1}
\vert \boldsymbol \varphi(\mathbf q)\vert \leq \mathcal R(\ln K)^{\frac{n+1}{2}}K^{-\frac{1}{2}\min(\alpha,\varsigma)}\quad \mbox{for}\quad n=2,
\end{equation}
or
\begin{equation}\notag%\label{eq:k2}
\vert \boldsymbol \varphi(\mathbf q)\vert \leq \mathcal R(\ln K)^{\frac{n+1}{2}}K^{-\frac{1}{2}\min(\alpha,\varsigma)+\frac{1}{6}},\quad  \min(\alpha,\varsigma)\in(1/3,1)\quad \mbox{for}\quad n=3,
\end{equation}
where $\mathcal R=\mathcal R(n,D,L,M,\mathcal G,\alpha,\lambda,\mu,\varsigma,\omega)$ is a positive constant depending on a-priori constants.
\end{cor}

%\Red{
%\begin{rem}
%	If  $\inf_{\partial \Omega}|\boldsymbol{\varphi}|>\Xi_m$ for some positive constant $\Xi_m\in \mathbb R_+$, then the source $(\Omega;\EM f)$ must be radiating for sufficiently. 
%\end{rem}
%}

%\subsection{Proof of Theorem \ref{thm:source rad}}\label{sub:proof1}

Before proving Corollary \ref{cor:non-rad}, we present some auxiliary lemmas that are essential for the proofs.
Lemma \ref{lem:cgo} provides an expression of the complex geometric optics (CGO) solution $\mathbf{u}_0$, as introduced in \cite{P1998}.
\begin{lem}\label{lem:cgo}
Let $\mathbf x\in \mathcal K_{-}$ where $\mathcal K_{-}=\{\mathbf x\in \mathbb R^n| x_n>K_{-}|\mathbf x^{'}|\}$, and given $\mathbf d,\mathbf d^\perp\in \mathbb S^{n-1}$ such that $\mathbf d\cdot \mathbf d^\perp=0$ and $\mathbf d\cdot \hat{\mathbf x}\in(-1,0)$ with $\hat{\mathbf x}=\frac{\mathbf x}{\vert \mathbf x\vert}$. Denote the Complex Geometric Optics (CGO) solution as follows:
\begin{equation}\label{eq:cgo}
\mathbf u_0=\eta e^{\xi\cdot \mathbf x},
\end{equation}
where
\begin{equation}\label{eq:def eta}
\mathbf \xi=\tau\mathbf d+\mathrm i\sqrt{\kappa_s^2+\tau^2}\mathbf d^\perp,\quad \eta=-\mathrm i\sqrt{1+\frac{\kappa_s^2}{\tau^2}}\mathbf d+\mathbf d^\perp,
\end{equation}
$\kappa_s$ is defined in \eqref{eq:wave number} and $\kappa_s<\tau.$
Then $\mathbf u_0$ satisfies
\begin{equation}\notag
\mathcal L \mathbf u_0+\omega^2 \mathbf u_0=\mathbf 0.
\end{equation}
\end{lem}

%\begin{rem}
%By virtue of the definition of $\eta$ and the selection of $\kappa_s$ and $\tau$, it follows that $\eta$ is bounded such that $\sqrt{2} < |\eta| < \sqrt{3}$.
%\end{rem}

Next, we shall present a crucial integral identity \eqref{eq:integral sum} in Lemma \ref{lem:integral cur}.
\begin{lem}\label{lem:integral cur}
Let $\Omega\subset \mathbb R^n$ $(n=2,3)$ be a bounded Lipschitz domain and $\Omega_{\rho,b}$ be given by \eqref{eq:orho h} with $\mathbf 0\in \overline {\Omega}_{\rho,b}\cap \partial \Omega$, where $\mathbf 0$ is  an admissible K-curvature point with parameters $K,L,M,\varsigma.$ Let  $\mathbf u_0$ be given by \eqref{eq:cgo}. Assume that $\mathbf u\in  H^2(\Omega_{\rho,b})^n\cap C(\overline{\Omega}_{\rho,b})^n$ and $\boldsymbol\varphi\in L^\infty(\Omega_{\rho,b})^n$ satisfy
\begin{equation}\notag
\mathcal L \mathbf u+\omega^2 \mathbf u=\boldsymbol\varphi\   \mbox{in}\  \Omega_{\rho,b}\quad \mbox{and}\quad 
 \mathbf u=\T_\nu \mathbf u=\mathbf 0\ \mbox{on}\  \overline{\Omega}_{\rho,b}\cap \partial \Omega.
\end{equation}
Then we have the following integral identity
\begin{equation}\label{eq:integral sum}
\boldsymbol\varphi(\mathbf 0)\cdot \eta \int_{x_n>K\vert \mathbf x^{'}\vert ^2}e^{\xi \cdot \mathbf x}\mathrm d\mathbf x=\Sigma_{j=1}^4\mathcal I_4,
%\begin{aligned}
%\boldsymbol\varphi(\mathbf 0)\cdot \eta&\int_{x_n>K\vert \mathbf x^{'}\vert ^2}e^{\xi \cdot \mathbf x}\mathrm d\mathbf x
%=\boldsymbol\varphi(\mathbf 0)\cdot \eta\int_{x_n>\max(b,K\vert \mathbf x^{'}\vert ^2)} e^{\xi \cdot \mathbf x}\mathrm d\mathbf x\\
% &+\boldsymbol\varphi(\mathbf 0)\cdot \eta  \left [\int_{K\vert \mathbf x^{'}\vert ^2<x_n<b} e^{\xi \cdot \mathbf x}\mathrm d\mathbf x-\int_{\Omega_{\rho,b}}e^{\xi \cdot \mathbf x}\mathrm d \mathbf  x\right ]\\
 %&-\int_{\Omega_{\rho,b}}\mathbf u_0\cdot (\boldsymbol\varphi(\mathbf x)-\boldsymbol\varphi(\mathbf 0)-\omega^2(\mathbf u(\mathbf x)-\mathbf u(\mathbf 0)))\mathrm d\mathbf x\\
% &+\int_{\partial \Omega_{\rho,b}\setminus \partial \Omega}\mathbf u_0 \cdot (T_{\nu}\mathbf u)-\mathbf u\cdot( T_{\nu}\mathbf u_0 )\mathrm d\sigma.
%\end{aligned}
\end{equation}
where
\begin{align}
\mathcal I_1&=\boldsymbol\varphi(\mathbf 0)\cdot \eta\int_{x_n>\max(b,K\vert \mathbf x^{'}\vert ^2)} e^{\xi \cdot \mathbf x}\mathrm d\mathbf x,\notag\\
 \mathcal I_2&=\boldsymbol\varphi(\mathbf 0)\cdot \eta  \left [\int_{K\vert \mathbf x^{'}\vert ^2<x_n<b} e^{\xi \cdot \mathbf x}\mathrm d\mathbf x-\int_{\Omega_{\rho,b}}e^{\xi \cdot \mathbf x}\mathrm d \mathbf  x\right ]\notag\\
\mathcal I_3&=-\int_{\Omega_{\rho,b}}\mathbf u_0\cdot (\boldsymbol\varphi(\mathbf x)-\boldsymbol\varphi(\mathbf 0))\mathrm d\mathbf x,\notag\\
 \mathcal I_4&=\int_{\partial \Omega_{\rho,b}\setminus \partial \Omega}\mathbf u_0 \cdot (T_{\nu}\mathbf u)-\mathbf u\cdot( T_{\nu}\mathbf u_0 )\mathrm d\sigma.\label{eq:I1-4}
\end{align}
\end{lem}
Drawing upon the results in \cite{BL2021}, we have the following inequalities related to the asymptotic analysis with respect to the curvature $K$ and the parameter $\tau$.
\begin{lem}\label{lem:estimate}
Let  $\xi\in \mathbb C^n$ be defined in  \eqref{eq:def eta} and $K,K_{\pm}$ be given by \eqref{eq: K_+}. Denote $\mathcal K=\{\mathbf x\in \mathbb R^n\vert K|\mathbf x^{'}|^2<x_n\}$, $\mathcal K_b=\{\mathbf x\in \mathbb R^{n}|K\vert \mathbf x^{'}\vert ^2<x_n<b\}$ for $b\in \mathbb R_{+}\cup\{\infty\}$ and $\mathcal K_{\pm}=\{\mathbf x\in \mathbb R^n| K_{\pm}|\mathbf x^{'}|^2<x_n<b\}$. Then it yields that
\begin{align}
&\int_{\mathcal K\setminus \mathcal K_b}e^{-\tau x_n}\mathrm d\mathbf x\leq C_n\frac{1+(\tau b)^{\frac{n-1}{2}}}{\tau^{\frac{n+1}{2}}K^{\frac{n-1}{2}}}e^{-\tau b},\notag\\
&\int_{\mathcal K_b}e^{-\tau x_n}\vert x\vert ^{\alpha}\mathrm d\mathbf x \leq C_{n,\alpha}(b+K^{-1})^{\frac{\alpha}{2}}b^{\frac{n+\alpha+1}{2}}K^{-\frac{n-1}{2}},\notag\\
&\int_{\mathcal K} e^{\xi \cdot \mathbf x}\mathrm d \mathbf x =-\frac{1}{\xi_n}\left( \frac{\pi}{-\xi_nK} \right) ^{\frac{n-1}{2}}\exp\{-\frac{\xi^{'}\cdot \xi^{'}}{4\xi_nK}\},\notag\\
&\int_{\mathcal K_{-}\setminus \mathcal K_{+}} e^{-\tau x_n}\mathrm d\mathbf x =\frac{\sigma(\mathbb S^{n-2})}{n-1}\left( \left ( \frac{1}{K_{-}} \right)^{\frac{n-1}{2}}-\left( \frac{1}{K_{+}} \right)^{\frac{n-1}{2}}\right)\frac{1}{\tau^{\frac{n+1}{2}}}\gamma(\tau b,\frac{n+1}{2})\notag,
\end{align}
where $\xi^{'}=(\xi_1,\dots,\xi_{n-1})$ and $\gamma(t,c)=\int_{0}^{t}e^{-x}x^{c-1}\mathrm d x,c\in \mathbb C$ is the lower incomplete gamma function. The positive constant $C_n$  depends only on the dimension $n$ and the positive constant $C_{n,\alpha}$ depends on $n$ and $\alpha$.
\end{lem}

For non-radiating sources $\mathbf{f}$, the following lemma establishes improved regularity from $H^2(\Omega)^n$ to $C^{1,\beta}(\overline{\Omega})^n$ for the displacement field $\mathbf{u}$. Moreover, it provides a global estimate bounding the H\"older norm $\|\mathbf{u}\|_{C^{1,\beta}(\overline{\Omega})^n}$ by the $H^1(\Omega)^n$-norm of $\mathbf{f}$. 
\begin{lem} \label{lem:rise reg}
Let $\Omega\subset \mathbb R^n,n\in\{2,3\}$ be a bounded Lipschitz domain of the diameter at most $D\in \mathbb R_{+}$ with a connected complement. Recall that  \( \kappa_s \) and  \( \kappa_p \) are defined by \eqref{eq:wave number}. Suppose that $\EM u\in H^2(\Omega)^n$  is a solution to \eqref{eq:lame} and $\EM f\in H^1(\Omega)^n$ is non-radiating,
\iffalse 
\begin{align}\label{eq:uH02}
\begin{cases}
    \mathcal L \EM u+\omega^2\EM u=\EM f \ \mbox{in}\ \Omega,\\
    \EM u=T_\nu \EM u=\EM 0\ \ \mbox{on}\ \partial \Omega,
    \end{cases}
\end{align}
    where $\EM f=\chi_\Omega \boldsymbol{\varphi} $ with $\boldsymbol{\varphi} \in H^1(\Omega)^n$. 
    \fi
then  one has $\EM u\in C^{1,\beta}(\overline{\Omega})^n,\beta\in (0,1]$ and  for $n=2$ and $\beta\in (0,1/2]$ for $n=3$. Moreover, there holds that
    \begin{align}
        \|\EM u\|_{C^{1,\beta}(\overline{\Omega})^n}\leq C\|\EM f\|_{H^1(\Omega)^n},\notag    \end{align}
    where $C=C_{n, D,\omega,\lambda,\mu,\beta}$ is a positive constant.
\end{lem}
\begin{proof}
    First, denote $\tilde {\EM u}$ as the zero extension of $\EM u$ outside of $\Omega$ on a ball $B_{R}$ where $R>D$ is a radius such that $\Omega\Subset B_R$ and  neither \( \kappa_s^2 \) nor \( \kappa_p^2 \) is a Dirichlet eigenvalue for \(-\Delta\) in the ball \( B_R \) since only $0$. Moreover, there exist $D<R_1<R$ such that $\Omega\Subset B_{R_1}\Subset B_R$.  
        It is clear that $\tilde{\EM u}\in H^2(B_R)$ satisfies 
    \begin{align}\notag
        \mathcal L\tilde {\EM u} +\omega^2 \tilde {\EM u}=\EM f\ \mbox{in}\ \B_{R},\ \tilde {\EM u}=\EM 0\ \mbox{on}\ \partial \B_{R}.
    \end{align}
   Recall that $\EM u$ and $\EM f$ satisfy  Helmholtz decomposition \eqref{eq:dehelm} and \eqref{eq:deu}, we derive an analogous result for $ \tilde{  {\EM u}}$:
    \begin{align}\label{eq:helmde2}
        \Delta \tilde{  {\EM u}}_p+\kappa_p^2\tilde{  {\EM u}}_p=\EM f_p,\ 
        \Delta \tilde{  {\EM u}}_s+\kappa_s^2\tilde{  {\EM u}}_s=\EM f_s,
    \end{align}
    where $\tilde{\EM u}=\tilde{\EM u}_p+\tilde{\EM u}_s$ and $\EM f=\EM f_p+\EM f_s.$ 
    Since $\tilde {\EM u}=\EM 0$ in $B_R\setminus \overline{ \Omega}$, we have $\tilde{\EM u}_p=\tilde{\EM u}_s=\EM 0$   on $\partial B_R$.  %Since neither \( \kappa_s^2 \) nor \( \kappa_p^2 \) is a Dirichlet eigenvalue for \(-\Delta\) in the ball \( B_R \),
    Then by \cite[Corollary 8.7]{GilTru83} and \eqref{eq:helmde2}, we obtain that  
    \begin{align}
        \|\tilde{\EM u}_\iota\|_{H^1(B_{R})^n}\leq C \|\EM f_\iota\|_{L^2(\Omega)^n}\ (\iota=p,s).\notag
    \end{align}
    Moreover, by virtue of the fact that $\|\EM f_p\|_{L^2(\Omega)^n}+\|\EM f_s\|_{L^2(\Omega)^n}\leq 2\|\EM f\|_{L^2(\Omega)^n}$, we have
    \begin{align}
        \|\tilde{\EM u}\|_{H^1(B_{R})^n}&\leq   \|\tilde{\EM u}_p\|_{H^1(B_{R})^n}+\|\tilde{\EM u}_s\|_{H^1(B_{R})^n}
         \leq C \|\EM f\|_{L^2(\Omega)^n}.\label{eq:uH1f}
    \end{align}
     It is noted that  there exists a curve $\Gamma$ such that $\Gamma$ is $C^{2,1}$ and $\Gamma \cap \partial B_{R_1}\not=\emptyset$ but $\Gamma\cap B_{R_1}=\emptyset$.
     \iffalse
     see Figure \ref{fig:holder} for a schematic illustration.
     \begin{figure}[ht!]
\centering
\includegraphics[width=4cm]{Ball.pdf}
\caption{The schematic illustration for $\Omega,B_{R},B_{R_1}$ and $\Gamma$.}
\label{fig:holder}
\end{figure}
\fi
    Since  $\tilde {\EM u }=T_\nu \tilde{\EM u}=\EM 0 \in H^{\frac{5}{2}}(\Gamma)^n$, then   by \cite[Theorem 4.18]{MC}, one has $\tilde {\EM u}\in H^3(B_{R_1})^n$ and 
    \begin{align}\label{eq:uH3}
        \|\tilde {\EM u}\|_{H^3(B_{R_1})^n}\leq C (\|\tilde {\EM u}\|_{H^1(\Omega)^n}+\|\EM f\|_{H^1(\Omega)^n}).
    \end{align}
   Therefore, combining   \eqref{eq:uH1f} and\eqref{eq:uH3}, we have
   \begin{align}\notag
       \| {\EM u}\|_{H^3(\Omega)^n}=\| {\tilde{\EM u}}\|_{H^3(B_{R_1})^n}\leq C\|\EM f\|_{H^1(\Omega)^n}.
   \end{align}
    Then by Sobolve embedding theorem, it follows that 
    $$\|{\EM u}\|_{ C^{1,\beta}(\overline{\Omega})^n}\leq C\|  {\EM u}\|_{H^3(\Omega)^n}\leq C\|\EM f\|_{H^1(\Omega)^n},$$%=C\|\boldsymbol{\varphi}\|_{H^1(\Omega)^n},
where $\beta\in (0,1]$ for $n=2$ and $\beta\in (0,1/2]$ for $n=3$ and $C=C_{n,R,R_1,D,\omega,\lambda,\mu,\beta}$ is a positive constant. 

The proof is complete.
\end{proof}

%\begin{rem}
%    In order to leverage \cite[Corollary 8.7]{GilTru83} for \eqref{eq:helmde2}, we further require that the radius \( R > D \)  is such that neither \( \kappa_s^2 \) nor \( \kappa_p^2 \) is a Dirichlet eigenvalue for \(-\Delta\) in the ball \( B_R \). This is possible because $0$ is the only common eigenvalue of $-\Delta$ among all large disks.
%\end{rem}

%\begin{rem}
 %Indeed, from Lemma \ref{lem:rise reg}, we know that $\mathbf u\in C^{1,\beta^{'}},\beta^{'}\in(0,1/2]$ implying that $\mathbf u\in C^{1+\beta}$ with $\beta\in (0,\beta^{'}).$ Moreover, we maintain uniformity by considering $\beta^{'}\in (0,1/2]$ for both $n=2$ and $n=3$.
%\end{rem}
%In the following we establish a global point-wise estimation \eqref{eq:u00} for $|\EM u(\EM x)|$ in $\Omega$ when $\EM f$ is non-radiating and  derive a vanishing property of Jacobi matrix of $\EM u$ at the admissible $K$-curvature point. To prove  Proposition \ref{prop:ubeta}, we utilize the geometric characterization of domain in the neighborhood of the admissible $K$-curvature point. 

Subsequently, we establish a global point-wise estimate \eqref{eq:u00} for $|\EM u(\EM x)|$ in $\Omega$ when $\EM f$ is non-radiating and derive the vanishing property of the Jacobi matrix of $\EM u$ at the admissible $K$-curvature point. The following proposition is essential for estimating the integral term $\mathcal{I}_4$ in Lemma~\ref{lem:integral cur}. The proof of Proposition~\ref{prop:ubeta} exploits the   geometric characterization of the domain near admissible $K$-curvature boundary points.

\begin{prop}\label{prop:ubeta} Let $\Omega\subset \mathbb R^n$ $(n=2,3)$ be a bounded Lipschitz domain and $\Omega_{\rho,b}$ be given by \eqref{eq:orho h} with $\mathbf 0\in \overline {\Omega}_{\rho,b}\cap \partial \Omega$, where $\mathbf 0$ is  an admissible K-curvature point with parameters $K,L,M,\varsigma.$  Assume that $\EM u\in H^2(\Omega)^n$ is a solution to \eqref{eq:lame} and $\EM f$ is non-radiating, then there holds that 
      $$\nabla \mathbf u(\mathbf 0)=\mathbf 0$$
    and
    \begin{equation}\label{eq:u00}
        \vert\mathbf u (\mathbf x) \vert \leq \|\mathbf u\|_{C^{1,\beta}(\overline{\Omega})^n}|\mathbf{x}|^{1+\beta},\ \forall \mathbf x\in \overline{\Omega}, 
    \end{equation}
     where $\beta\in (0,1]$ for $n=2$ and $\beta\in (0,1/2]$ for $n=3$.
\end{prop}
\begin{proof}
    According to Lemma \ref{lem:rise reg}, we have $\mathbf u$ belongs to $C^{1,\beta}(\overline{\Omega} )^n$, implying that $\mathbf u$ satisfies the following expansion:
    \begin{equation}\notag
    \mathbf u(\mathbf x)=\mathbf u(\mathbf 0)+\nabla \mathbf u (\mathbf 0) \mathbf x+\boldsymbol \delta (\mathbf x),\ \vert \boldsymbol \delta (\mathbf x)\vert\leq \|\mathbf u\|_{C^{1,\beta}(\overline{\Omega})^n}|\mathbf x|^{1+\beta},\ \forall \mathbf x\in \overline{\Omega}.
    \end{equation}
   % Since $\EM f$ is non-radiating, he field satisfies $\mathbf{u} = T_\nu \mathbf{u} = \mathbf{0}$ on $\partial\Omega$. Therefore, we only need to show that $\nabla \mathbf u(\mathbf 0)=\mathbf 0$,  \eqref{eq:u00} follows immediately from the vanishing condition $\mathbf{u}(\mathbf{0}) = \mathbf{0}$.
    Since the source $\mathbf{f}$ is non-radiating, it yields that  $\mathbf{u} = T_\nu \mathbf{u} = \mathbf{0}$ on $\partial\Omega$. To prove \eqref{eq:u00}, it therefore suffices to show $\nabla\mathbf{u}(\mathbf{0}) = \mathbf{0}$, which follows immediately from the vanishing condition $\mathbf{u}(\mathbf{0}) = \mathbf{0}$.    \medskip
    
    \textbf{Case I:} $n=2$. We know that the boundary $\overline{\Omega}_{\rho,b}\cap \partial \Omega$ can be characterized by the curve $\mathbf r(x_1)=(x_1,\gamma(x_1))$, where the function $\gamma$ is given in Definition \ref{def:ad-hi-cur} and $\mathbf r$ is parameterized by $x_1$.  Let $s$ be a parameter characterizing the arc length of the curve $\mathbf r(x_1)$. Subsequently, the tangential derivative of $u_i\ (i=1,2)$ along the curve is represented by $\frac{\mathrm d u_i}{\mathrm d s}$.  By virtue of the condition $\mathbf u=\mathbf 0$ on $\partial \Omega$, we  have 
    $$0=\frac{\mathrm d u_i}{\mathrm d s}=\frac{\mathrm du_i}{ \mathrm dx_1}\frac{1}{|\mathbf r^{'}(x_1)|}=\nabla u_i\cdot  \mathbf t,$$
    where $ \mathbf r^{'}(x_1)=(1,\gamma^{'}(x_1))$ and  $\mathbf t$ is a tangent vector  at $\mathbf x\in \partial \Omega$.
  Hence, one has 
    \begin{equation}\label{eq:u01}
 \partial_1u_1(\mathbf 0)=\partial_1u_2(\mathbf 0)=0,
    \end{equation}
    since $\mathbf t=(1,0)^\top$ at $\mathbf x=\mathbf 0.$
    
    Furthermore, we have the expression of $T_{\nu}\mathbf u$ for $n=2$ as follows:
 \begin{equation}\notag
 \begin{aligned}
 T_{\nu}\mathbf u&=2\mu\frac{\partial \mathbf u}{\partial \nu}+\lambda \nu\nabla\cdot \mathbf u+\mu\nu^\perp(\partial_2u_1-\partial_1u_2)\\
         &=2\mu(\nu_1\partial_1+\nu_2\partial_2)\mathbf u+\lambda \nu(\partial_1u_1+\partial_2u_2)+\mu\nu^\perp(\partial_2u_1-\partial_1u_2).
 \end{aligned}
 \end{equation}
 Recalling that $T_\nu \mathbf u=\mathbf 0$ on $\partial \Omega$ and $\nu=(0,-1)$ and $\nu^\perp =(1,0)$ at $\mathbf x=\mathbf 0$, we then have the formula for $T_\nu \mathbf u$ at $\mathbf x=\mathbf 0$ as follows:
\begin{equation}\label{eq:tu=0}
T_\nu\mathbf u\big\vert_{\mathbf x=\mathbf 0}=-
\left(
\begin{array}{c}
\mu(\partial_2u_1(\mathbf 0)+\partial_1u_2(\mathbf 0))\\
\lambda\partial_1u_1(\mathbf 0)+(\lambda+2\mu)\partial_2u_2(\mathbf 0)
\end{array}
\right)=\mathbf 0.
\end{equation}
Combining \eqref{eq:u01} with \eqref{eq:tu=0}, it is clear that $\partial_2u_1(\mathbf 0)=\partial_2u_2(\mathbf 0)=0$. Therefore, $\nabla \mathbf u(\mathbf 0)=\mathbf 0$ for $n=2.$

\medskip

\textbf{Case II:} $n=3$.  Similarly, one has $\overline{\Omega}_{\rho,b}\cap\partial \Omega$ can be represented by the surface $\mathbf r(x_1,x_2)=(x_1,x_2,\gamma(x_1,x_2))$, where $\gamma$ is also given by Definition \ref{def:ad-hi-cur}. For a given point $\mathbf a\in \overline{\Omega}_{\rho,b}\cap\partial \Omega$, let $\boldsymbol{\varkappa}:[-1,1]\to \mathbf r(x_1,x_2)$ be a smooth curve denoted by $\varkappa(t)=(\varkappa_1(t),\varkappa_2(t),\varkappa_3(t))\in \mathbf r(x_1,x_2),t\in[-1,1]$  such that  $\boldsymbol{\varkappa}(  0)=\mathbf a$ and $\varkappa^{'}(  0)=\mathbf t$, where $\mathbf t$ is a tangent vector at $\mathbf a$. Then the tangential derivative of $u_i,i=1,2,3$ along $\mathbf t$ at $\mathbf a$ is given by
$$\nabla u_i\cdot \mathbf t=\frac{\mathrm d}{\mathrm d t}u_i \circ \varkappa(t)|_{t=0}.$$
By virtue of $\mathbf u(\mathbf 0)=\mathbf 0$, we then obtain 
$$\frac{\mathrm d}{\mathrm d t}u_i \circ \varkappa(t)|_{t=0}=\frac{\mathrm d}{\mathrm dt}u_i(\varkappa_1(t),\varkappa_2(t),\varkappa_3(t))|_{t=0}= 0.$$
Let $\mathbf t_1=(1,0,0)^{\top}$ and  $\mathbf t_2=(0,1,0)^{\top}$ be two tangent vectors at $\mathbf 0$, it yields that 
\begin{equation}\label{eq:u02}
    \partial_1 u_i(\mathbf 0)=\partial_2 u_i(\mathbf 0)=0\ (i=1,2,3).
\end{equation}

When $n=3$, one has the following formula for $T_\nu \mathbf u$:
\begin{equation}\notag
\begin{aligned}
T_\nu \mathbf u &=2\mu\frac{\partial \mathbf u}{\partial \nu}+\lambda \nu \nabla\cdot \mathbf u+\mu\nu\times (\nabla \times \mathbf u)\\
&=2\mu(\nu_1\partial_1+\nu_2\partial_2+\nu_3\partial_3)\mathbf u+\lambda \nu(\partial_1u_1+\partial_2u_2+\partial_3u_3)+\mu\nu\times (\nabla \times \mathbf u),
\end{aligned}
\end{equation}
where
\begin{equation}\notag
\nu\times (\nabla \times \mathbf u)=
\left(
\begin{array}{c}
\nu_2(\partial_1u_2-\partial_2u_1)-\nu_3(\partial_3u_1-\partial_1u_3)\\
\nu_3(\partial_2u_3-\partial_3u_2)-\nu_1(\partial_1u_2-\partial_2u_1)\\
\nu_1(\partial_3u_1-\partial_1u_3)-\nu_2(\partial_2u_3-\partial_3u_2)
\end{array}
\right).
\end{equation}
Recalling that $T_\nu \mathbf u=\mathbf 0$ on $\partial \Omega $ and $\nu=(0,0,-1)$ at $\mathbf 0$, then one has
\begin{equation}\label{eq:tu3}
T_\nu\mathbf u|_{\mathbf x=\mathbf 0}=-\left(\begin{array}{c}
\mu(\partial_1u_3(\mathbf 0)+\partial_3u_1(\mathbf 0))\\
\mu(\partial_2u_3(\mathbf 0)+\partial_3u_2(\mathbf 0))\\
\lambda(\partial_1u_1(\mathbf 0)+\partial _2u_2(\mathbf 0))+(\lambda+2\mu )\partial _3u_3(\mathbf 0)
\end{array}\right)=\mathbf 0.
\end{equation}
Combining \eqref{eq:u02} and \eqref{eq:tu3} yields that  $\partial_3u_i=0$ ($i=1,2,3$), implying that $\nabla \mathbf u(\mathbf 0)=\mathbf 0$.

%Moreover, recalling the expansion \eqref{eq:uexp} and $\mathbf u(\mathbf 0)=\mathbf 0$, we find that $|\mathbf u(\mathbf x)|=|\mathbf r(\mathbf x)|\leq\|\mathbf u \|_{C^{1,\beta}(\overline{\Omega{\rho,b}})}\vert \mathbf x\vert^{1+\beta}$. 
The proof is complete.
\end{proof}

Combining Lemma~\ref{lem:rise reg} with Proposition~\ref{prop:ubeta}, we obtain the following estimate for the boundary integral term $\mathcal I_4$ in Lemma \ref{lem:integral cur}.
\begin{prop}\label{prop:Tnu}
%Let $\Omega \subset \mathbb R^n$, $n\in\{2,3\}$ be a bounded Lipschitz domain and $\Omega_{\rho,b}$ be given in Definition \ref{def:ad-hi-cur} with an admissible $K$-curvature point $\mathbf q\in\partial \Omega$ with parameters $K,L,M,\varsigma$. Suppose that  $\mathbf u\in H^2(\Omega_{b,h})^n$ satisfies $\mathcal L\mathbf u+\omega^2\mathbf u=\chi_\Omega\boldsymbol\varphi$, $\boldsymbol\varphi\in H^1(\Omega_{\rho,b})^n$ and $\mathbf u=T_\nu \mathbf u=\mathbf 0 \ \mbox{on}\ \overline{\Omega}_{\rho,b}\cap\partial \Omega$. Let  $\mathbf u_0=\eta e^{\xi \cdot \mathbf x}$ be in the form of \eqref{eq:cgo} satisfying \eqref{eq:def eta}.  Denote 
Let $\mathcal I_4$ be given in Lemma \ref{lem:integral cur}, then there holds that 
\iffalse
\begin{equation}\label{eq:Tnu}
I=\int_{\partial \Omega_{\rho,b}\setminus \partial \Omega}\mathbf u_0\cdot(T_\nu \mathbf u)-\mathbf u\cdot(T_\nu \mathbf u_0)\mathrm d \sigma,
\end{equation}
\fi

\begin{equation}\label{eq:estTnu}
\vert \mathcal I_4\vert \leq Ce^{-\tau b}K^{-(\beta+\frac{n+1}{2})}(K+\tau )\|\EM u\|_{C^{1,\beta}(\overline{\Omega}_{\rho,b})^n},
\end{equation}
where $\beta\in (0,1]$ for $n=2$ and $\beta\in (0,1/2] $ for $n=3$, and $C$ is a positive constant independent of $K$ and the parameter $\tau$.
\end{prop}

 \begin{proof}
  Without loss of generality, assume that the admissible $K$-curvature boundary point $\mathbf{q}$ is located at the origin $\mathbf{0}$, we will prove this proposition  for the cases of $n=2$ and $n=3$, respectively.

$\mathbf{Case\  I}:n=2 $. Let $\mathbf d=(0,-1)^{\top}$ and $\mathbf d^\perp=(1,0)^{\top}$, it follows from  \eqref{eq:def eta} that $  \eta=(1,\mathrm i\sqrt{1+\frac{\kappa_s^2}{\tau^2}})^\top$ and $\xi=(\mathrm i \sqrt{\tau^2+\kappa_s^2},-\tau)^\top$. 

 Similar to \eqref{eq:tu=0}, the following formulas hold for $T_\nu\mathbf u$ and $T_\nu\mathbf u_0$ at $x_2=b$:
\begin{equation}\label{eq:tuh}
T_\nu \mathbf u(x_1,b)=
\left(
\begin{array}{c}
\mu(\partial_2u_1(x_1,b)+\partial_1u_2(x_1,b))\\
\lambda\partial_1u_1(x_1,b)+(\lambda+2\mu)\partial_2u_2(x_1,b)
\end{array}
\right)
\end{equation}
and
\begin{equation}\label{eq:tvh}
T_{\nu}\mathbf u_0(x_1,b)=-2\mu\tau
\left(
\begin{array}{c}
2+\frac{\kappa_s^2}{\tau^2}\\
-2\mathrm i\sqrt{1+ \frac{\kappa_s^2}{\tau^2}}
\end{array}
\right)e^{\mathrm i\sqrt{\tau^2+\kappa_s^2}x_1-\tau b},
\end{equation}
since $\nu=(0,1)^{\top}$ when $x_2=b.$
Substituting \eqref{eq:tuh} and \eqref{eq:tvh} into \eqref{eq:Tnu}, one has
\begin{equation}\label{eq:In2}
\begin{aligned}
\mathcal I_4=e^{-\tau b}&\int_{\partial \Omega_{\rho,b}\setminus \partial \Omega} e^{\mathrm i\sqrt{\kappa_s^2+\tau^2}x_1}\bigg\lbrack
\mu (\partial_2u_1(x_1,b)+\partial_1u_2(x_1,b))\\
&+ \mathrm i\sqrt{1+\frac{\kappa_s^2}{\tau^2}}(\lambda\partial_1u_1(x_1,b)+(\lambda +2\mu)\partial_2u_2(x_1,b)) \bigg\rbrack\\
&+\mu\tau e^{\mathrm i\sqrt{\kappa_s^2+\tau^2}x_1}\left[\left(2 +\frac{\kappa_s^2}{\tau^2}\right)u_1(x_1,b)+2\mathrm i\sqrt{1+\frac{\kappa_s^2}{\tau^2}}u_2(x_1,b)
\right]\mathrm d\sigma.
\end{aligned}
\end{equation}

By virtue of $\mathbf u\in C^{1,\beta}(\overline{\Omega}_{\rho,b})^n$ with $\mathbf u=(u_1,u_2)^\top$ in  Lemma \ref{lem:rise reg}, we have $\partial_i u_j\in C^{\beta}(\overline{\Omega}_{\rho,b})^n$ for $i,j=1,2$. Combining \eqref{eq:tuh}-\eqref{eq:In2} with Proposition \ref{prop:ubeta}, and  taking $b=1/K$, we obtain the following estimates:
\begin{align}\notag
&\mu\tau\left\vert \left(2 +\frac{\kappa_s^2}{\tau^2}\right)u_1(x_1,b)+2\mathrm i\sqrt{1+\frac{\kappa_s^2}{\tau^2}}u_2(x_1,b)\right\vert \notag\\
&\quad \quad \quad \quad \quad\leq 3\mu \tau \left ( \|u_1\|_{C^{1,\beta}(\overline{\Omega}_{\rho,b})^n}+\|u_2\|_{C^{1,\beta}(\overline{\Omega}_{\rho,b})^n}\right)(x_1^2+b^2)^{\frac{1+\beta}{2}},\notag\\
&\left \vert \lambda \partial_1u_1(x_1,b)+ (\lambda +2\mu )\partial_2u_2(x_1,b)\right \vert \notag\\
&\quad \quad \quad \quad \quad \leq(\lambda+2\mu) (\|\partial_1u_1\|_ {C^{\beta}(\overline{\Omega}_{\rho,b})^n}+\|\partial_2u_2\|_{C^{\beta}(\overline{\Omega}_{\rho,b})^n})( x_1^2+b^2) ^{\frac{{\beta}}{2}},\notag\\
&\left \vert \mu (\partial_2u_1(x_1,b)+\partial_1u_2(x_1,b))\right \vert\notag\\
& \quad \quad \quad \quad \quad \leq \mu(\|\partial_2 u_1\|_{C^{\beta}(\overline{\Omega}_{\rho,b})^n}+\|\partial_1 u_2\|_{C^{\beta}(\overline{\Omega}_{\rho,b})^n})(x_1^2+b^2) ^{\frac{{\beta}}{2}}.\notag
\end{align}
Recall that $\frac{K}{K_{-}}\leq M$, then one has the following inequality
\begin{equation}
\begin{aligned}
\vert \mathcal I_4\vert
&\leq C e^{-\tau b}\|\EM u\|_{C^{1,\beta}(\overline{\Omega}_{\rho,b})^n}\left(\int_{\partial \Omega_{\rho,b}\setminus \partial \Omega} (x_1^2+b^2)^{\frac{\beta}{2}}\mathrm d\sigma+ \tau \int_{\partial \Omega_{\rho,b}\setminus \partial \Omega} (x_1^2+b^2)^{\frac{1+\beta}{2}}\mathrm d\sigma\right )\notag\\
&\leq C \sigma(\mathbb S^{n-1}) e^{-\tau b} \|\EM u\|_{C^{1,\beta}(\overline{\Omega}_{\rho,b})^n}\left (\sqrt {\frac{b}{K_{-}}}\right)^{n-1}(K^{-\beta}+\tau K^{-(1+\beta)})\\
&\leq Ce^{-\tau b}K^{-(\beta+\frac{n+1}{2})}(K+\tau )\|\EM u\|_{C^{1,\beta}(\overline{\Omega}_{\rho,b})^n},
\end{aligned}
\end{equation}
where $C=C_{n,M,\lambda, \mu,\beta}$ is a positive constant independent of $K$ and $\tau$.

 $\mathbf {Case\ II}:n=3.$ Let $\mathbf d=(0,0,-1)^\top$ and $ \mathbf d^\perp=(\sin \varphi,\cos \varphi,0)^\top$, then one has
\begin{equation}\notag
\eta=\left(
\begin{array}{c}
\sin \varphi\\
\cos \varphi\\
\mathrm i\sqrt{1+\frac{\kappa_s^2}{\tau^2}}
\end{array}\right),\quad
\xi=\left(
\begin{array}{c}
\mathrm i\sqrt{\tau^2+\kappa_s^2}\sin \varphi\\
\mathrm i\sqrt{\tau^2+\kappa_s^2}\cos\varphi\\
-\tau
\end{array}
\right).
\end{equation}

Since   $\nu=(0,0,1)^{\top}$ when $x_3=b$,  then one has
\begin{equation} \label{eq:tub3}
T_\nu\mathbf u|_{x_3=b}=\left(\begin{array}{c}
\mu(\partial_1u_3(\mathbf x^{'},b)+\partial_3u_1(\mathbf x^{'},b))\\
\mu(\partial_2u_3(\mathbf x^{'},b)+\partial_3u_2(\mathbf x^{'},b))\\
\lambda(\partial_1u_1(\mathbf x^{'},b)+\partial _2u_2(\mathbf x^{'},b))+(\lambda+2\mu )\partial _3u_3(\mathbf x^{'},b)
\end{array}\right)
\end{equation}
and
\begin{equation}\label{eq:tu03}
T_\nu\mathbf u_0|_{x_3=b}=-\mu\tau \left(\begin{array}{c}
(2+\frac{\kappa_s^2}{\tau^2})\sin\varphi\\
(2+\frac{\kappa_s^2}{\tau^2})\cos\varphi\\
2\mathrm i\sqrt{1+\frac{\kappa_s^2}{\tau^2} }
\end{array}\right)e^{\mathrm i\sqrt{\tau^2+\kappa_s^2}(x_1\sin\varphi+x_2\cos \varphi)-\tau b}.
\end{equation} 
Substitute \eqref{eq:tub3} and \eqref{eq:tu03} into $\mathcal I_4$ in \eqref{eq:I1-4}, it yields that
\begin{equation}\notag
\begin{aligned}
\mathcal I_4&=e^{-\tau b}\int_{\partial \Omega_{\rho,b}\setminus \partial \Omega} e^{\mathrm i\sqrt{\tau^2+\kappa_s^2}(\sin \varphi x_1+\cos \varphi x_2)} \bigg\lbrack \mu \sin \varphi(\partial_3u_1+\partial_1u_3)+\mu\cos \varphi (\partial_3u_2+\partial_2u_3)\\
&+\mu\tau\left((2+\frac{\kappa_s^2}{\tau^2})( u_1\sin\varphi+u_2\cos\varphi )+2\mathrm i \sqrt{1+\frac{\kappa_s^2}{\tau^2}}u_3\right)\\
&+\mathrm i\sqrt{1+\frac{\kappa_s^2}{\tau^2}}\left (\lambda (\partial_1u_1+\partial_2u_2)+(\lambda+2\lambda)\partial_3u_3\right )\bigg\rbrack\mathrm d\sigma,
\end{aligned}
\end{equation}
Then we can directly derive the result in \eqref{eq:estTnu} using a similar analytical manner as in the case of $n=2$.

The proof is complete.
 \end{proof}

We now present the proof of Theorem \ref{thm:source rad}. We prove this theorem  by contradiction, arriving   the integral identity \eqref{eq:integral sum}. By leveraging the auxiliary Lemmas \ref{lem:integral cur}--\ref{lem:rise reg} and Propositions \ref{prop:ubeta}--\ref{prop:Tnu}, and by appropriately setting the parameter $\tau$, we derive the quantitative results \eqref{eq:f1} and \eqref{eq:f2}.\begin{proof}[The proof of Theorem \ref{thm:source rad}] 
Contradictly, assume that $\EM f=\chi_{\Omega}\boldsymbol \varphi$ represents a non-radiating source, namely,  $\mathbf u^\infty=\mathbf 0$, it follows that $\mathbf u=T_{\nu }\mathbf u=\mathbf 0 \ \mbox{on}\ \partial \Omega$ by  Rellich's Lemma. Without loss of generality, assume that the admissible $K$-curvature $\EM p$ is $\EM 0$. 
 {First, we assume that $\EM f=\chi_\Omega \boldsymbol{\varphi}$ is a real-valued function. }
Since $\boldsymbol{\varphi}\in C^\alpha(\overline{\Omega}_{\rho,b})^n$, then one has %Then for $\mathcal I_3$ define in \eqref{eq:integral sum}, by utilizing Lemma \ref{lem:rise reg}, we have
\begin{align}
   \boldsymbol{\varphi}(\EM x)=\boldsymbol{\varphi}(\EM 0) +(\boldsymbol{\varphi}(\EM x)-\boldsymbol{\varphi}(\EM 0)), \quad |(\boldsymbol{\varphi}(\EM x)-\boldsymbol{\varphi}(\EM 0)|
    &\leq \|\boldsymbol{\varphi} \|_{C^\alpha}|\EM x|^\alpha\label{eq:phiu1}.
   % &\leq C (\|\boldsymbol{\varphi} \|_{C^\alpha }+\omega^2\|\EM u\|_{C^{1,\beta}})|\EM x|^\alpha.\notag\\
   % & \leq C (\|\boldsymbol{\varphi} \|_{C^\alpha }+\|\boldsymbol{\varphi}\|_{H^1})|\EM x|^\alpha\label{eq:phiu1}
\end{align}
Therefore, the integral \eqref{eq:integral sum} in Lemma \ref{lem:integral cur} holds immediately.  Furthermore, substitute the estimates in Lemma \ref{lem:estimate},  Proposition \ref{prop:Tnu} and \eqref{eq:phiu1}  into  \eqref{eq:integral sum}, it follows that
\begin{align}
\pi^{\frac{n-1}{2}}&K^{-\frac{n-1}{2}}\tau^{-\frac{n+1}{2}}e^{-\frac{\tau^2+\kappa_s^2}{4\tau K}}\vert \boldsymbol \varphi(\mathbf 0)\cdot \eta \vert
\leq \vert  \boldsymbol \varphi(\mathbf 0)\cdot \eta \vert C_n\frac{1+(\tau b)^{\frac{n-1}{2}}}{\tau^{\frac{n+1}{2}}K^{\frac{n-1}{2}}}e^{-\tau b}\notag\\
&+ \vert  \boldsymbol \varphi(\mathbf 0)\cdot \eta \vert \frac{\sigma (\mathbb S^{n-2})}{n-1}\tau^{-\frac{n+1}{2}}\left(\left( \frac{1}{K_{-}}\right)^{\frac{n-1}{2}}- \left( \frac{1}{K_{+}}\right)^{\frac{n-1}{2}}\right)\gamma(\tau b,\frac{n+1}{2})\label{eq:piktau1}\\
&+C_{n,\alpha,\omega}\| \boldsymbol \varphi\|_{C^\alpha}K^{-(\alpha+n)}+Ce^{-\tau b}K^{-(\beta+\frac{n+1}{2})}(K+\tau  )\|\EM u\|_{C^{1,\beta} }\notag.
\end{align}
By virtue of Lemma \ref{lem:rise reg}, multiplying both sides of \eqref{eq:piktau1} by  $K^{\frac{n-1}{2}}\tau^{\frac{n+1}{2}}e^{-\frac{\tau^2+\kappa_s^2}{4\tau K}}$  on  and utilizing the fact that $e^{\frac{\kappa_s^2}{4\tau K}}$ is bounded,  we can derive
\begin{align}
C_{n,L,M,\alpha,\beta,\lambda,\mu}\vert  \boldsymbol \varphi(\mathbf 0)\cdot \eta \vert 
 &\leq (1+(\tau b)^{\frac{n-1}{2}})e^{\frac{\tau}{4K}-\tau b}+K^{-\varsigma}e^{\frac{\tau}{4K}}\notag\\
 &+\| \boldsymbol \varphi\|_{C^\alpha}K^{-(\alpha+\frac{n+1}{2})}\tau^\frac{n+1}{2}e^{\frac{\tau}{4K}}\label{eq:piktau2}\\
&+ \tau^{\frac{n+1}{2}}K^{-(\beta+\frac{n+1}{2})}(K+\tau )e^{\frac{\tau}{4K}-\tau b}\|\boldsymbol{\varphi}\|_{H^{1} }.\notag
\end{align}
Let $\tau=4K(\ln K^\zeta)$, where $\zeta\in \mathbb R_+$ is a constant to be specified in what follows. Dividing by the constants and $\max(1,\|\boldsymbol{\varphi}\|_{C^\alpha},\|\boldsymbol{\varphi}\|_{H^1})$, then \eqref{eq:piktau2} can be reduced to   
\begin{align}
C_{n,D,L,M,\alpha,\beta,\lambda,\mu,\omega}&\frac{\vert \boldsymbol \varphi(\mathbf 0)\cdot \eta \vert}{\max(1,\|\boldsymbol{\varphi}\|_{C^\alpha},\|\boldsymbol{\varphi}\|_{H^1})^n} \leq(1+({\ln K^\zeta})^{\frac{n-1}{2}})K^{-3\zeta}
+K^{\zeta-\varsigma}\notag\\
&+K^{\zeta-\alpha}(\ln K^\zeta)^{\frac{n+1}{2}} 
+(\ln K^\zeta)^{\frac{n+1}{2}}K^{-\beta+1-3\zeta} \left(1+\ln K^\zeta \right)\label{eq:piktau3}
%&\leq (1+({\ln K^\zeta})^{\frac{n-1}{2}})K^{-3\zeta}+K^{\zeta-\varsigma}+K^{\zeta-\alpha}(\ln K^\zeta)^{\frac{n+1}{2}}\notag\\
%&+2(\ln K^\zeta)^{\frac{n+1}{2}}(1+\ln K^\zeta)K^{-\beta+1-3\zeta}.\label{eq:piktau3}
\end{align}
When $n=2$, setting $\zeta=\frac{1}{2}\min(\alpha,\varsigma)$, $\beta=1-\min(\alpha,\varsigma)$, and utilizing the inequality $(\ln x)^{t_1} \leq \frac{t_1 x^{t_2}}{t_2 e}$ for all $t_1, t_2 > 0$ and $x \geq e$, the right-hand side of \eqref{eq:piktau3} can be estimated by
\begin{equation}\label{eq:21}
C(\ln K)^\frac{n+1}{2}K^{-\frac{1}{2}\min(\alpha,\varsigma)},
\end{equation}
where $C$ is a positive constant independent of $K$. When $n=3$, let $\zeta=\frac{1}{2}\min(\alpha,\varsigma)+\frac{1}{6}$ and $\beta=\frac{1}{2}(1-\min(\alpha,\varsigma))$, we can derive that the right-hand side of  \eqref{eq:piktau3} can be estimated by 
\begin{equation}\notag %\label{eq:31}
C(\ln K)^\frac{n+1}{2}K^{-\frac{1}{2}\min(\alpha,\varsigma)+\frac{1}{6}}.
\end{equation}

{Since $\boldsymbol{\varphi}$ is real valued in $\Omega$, it is clear that 
\begin{equation}\label{eq:phi0}
\vert \boldsymbol\varphi(\mathbf 0)\vert  < \vert \boldsymbol\varphi(\mathbf 0)\cdot \eta \vert < \sqrt{3}\vert \boldsymbol\varphi(\mathbf 0)\vert.
\end{equation}}
 Combining \eqref{eq:21}-\eqref{eq:phi0} with \eqref{eq:piktau3}, we obtain
  \begin{equation} \label{eq:con1}
\frac{\vert \boldsymbol \varphi(\mathbf 0)\vert}{\max(1,\|\boldsymbol{\varphi}\|_{C^\alpha(\overline{\Omega})^n},\|\boldsymbol{\varphi}\|_{H^1(\Omega)^n})} \leq \mathcal C(\ln K)^{\frac{n+1}{2}}K^{-\frac{1}{2}\min(\alpha,\varsigma)}\ for\ n=2
\end{equation}
and 
\begin{equation} \label{eq:con2}
\frac{\vert \boldsymbol \varphi(\mathbf 0)\vert}{\max(1,\|\boldsymbol{\varphi}\|_{C^\alpha(\overline{\Omega})^n},\|\boldsymbol{\varphi}\|_{H^1(\Omega)^n})}  \leq \mathcal C(\ln K)^{\frac{n+1}{2}}K^{-\frac{1}{2}\min(\alpha,\varsigma)+\frac{1}{6}}\ for\  n=3, 
\end{equation}
which contradicts \eqref{eq:f1} (and \eqref{eq:f2}). 

{When $\EM f=\chi_\Omega\boldsymbol{\varphi}$ is a complex-valued function, we can rewrite  it and $\EM u$ as
$$\EM f=\EM f_\Re+{\rm i}\EM f_\Im=\chi_{\Omega}\boldsymbol{\varphi}_\Re+{\rm i}\chi_{\Omega}\boldsymbol{\varphi}_\Im \ \mbox{and}\ \EM u=\EM u_\Re+{\rm i}\EM u_\Im,$$
where $\EM f_\Re, \EM f_\Im, \EM u_\Re$ and $\EM u_\Im$ are real-valued functions. Then we have 
\begin{align}
\mathcal L\EM u_\Re+\omega^2\EM u_\Re=\EM f_\Re,\quad 
\mathcal L\EM u_\Im+\omega^2\EM u_\Im=\EM f_\Im,\notag
\end{align}
and $\EM u_\Re=T_\nu \EM u_\Re=\EM 0$, $\EM u_\Im=T_\nu \EM u_\Im=\EM 0$ on $\partial \Omega$. Then we can derive that both $\boldsymbol{\varphi}_\Re$ and $\boldsymbol{\varphi}_\Im$ satisfy \eqref{eq:con1} or \eqref{eq:con2}, which concludes that \eqref{eq:con1} and \eqref{eq:con2} hold for $\boldsymbol{\varphi}$ when it is complex-valued.  
} 

The proof is complete.
\end{proof}

\section{Geometrical characterizations of radiating and non-radiating  medium scatterer and  transmission eigenfunctions}\label{sec:3}
In this section, we derive geometric characterizations of inhomogeneous medium scatterers in Theorems \ref{thm:medsmall} and \ref{thm:medKpoint1}, which are proved by employing approaches similar to those used for the source cases in Theorems \ref{thm:small} and \ref{thm:source rad}. Furthermore, as applications, we also establish geometric properties of transmission eigenfunctions.

\subsection{The small inhomogeneous medium scatterer must be radiating.}
Before presenting the main result--Theorem \ref{thm:medsmall}, we give the following crucial proposition, which demonstrates that both the displacement field and the total field can be estimated by the incident wave.
\begin{prop}\label{prop:ust2}
    Let $\Omega\subset \mathbb R^n$ $(n=2,3)$ be an inhomogeneous medium scatterer with the physical configurations $(\lambda,\mu,V)$, where $V\in L^\infty (\mathbb R^n)$ and $supp (V)\subset \Omega$. Let $s=s_{n,\lambda,\mu}$ be a positive constant only depending on $n,\lambda,\mu$. Denote $\varepsilon=d(\Omega)\omega$ with  $\omega\in \mathbb R_{+}$. If $\varepsilon\|V\|_{L^\infty}\leq s_{n,\lambda,\mu}$ and $\EM u^i$ is an incident wave, then the displacement field $\EM u$ and the total field $\EM u^t$  to the medium scattering problem \eqref{eq:mesca} associated with  the radiation condition \eqref{eq:rad} satisfy  
     \begin{align}
        \|\EM u\|_{L^2(\Omega)}\leq \frac{\varepsilon \|V\|_{L^\infty}}{s- \varepsilon\|V\|_{L^\infty}}\|\EM u^i\|_{L^2(\Omega)} \ \mbox{and}\ \|\EM u^t\|_{L^2(\Omega)}\leq \frac{s}{s-\varepsilon \|V\|_{L^\infty}}\|\EM u^i\|_{L^2(\Omega)}.\label{eq:usi2}
    \end{align}
\end{prop}
\begin{proof}
Recall  the medium scattering problem  \eqref{eq:mesca}, by virtue of $\EM u^t=\EM u^i+\EM u$ and $\mathcal L \EM u^i+\omega^2 \EM u^i=\EM 0$, then one has 
$$ \mathcal L\EM u+\omega^2\EM u=\EM f,$$
where $\EM f=-\omega^2V\EM u^t$ and $\EM u$ satisfying Kupradze\ radiation\ condition. 
\iffalse
\begin{align}
    \begin{cases}\label{eq:medsca}
        \mathcal L\EM u+\omega^2\EM u=-\omega^2V\EM u^t\\
        Kupradze\ radiation\ condition.
    \end{cases}
\end{align}
\fi
It is clear that $\EM u^t\in H^2_{loc}(\mathbb R^n)^n$ satisfies  Lippmann-Schwinger equation
\begin{align}
    \EM u^t=\EM u^i+(\mathcal L+\omega^2I)^{-1}\EM f,\label{eq:LSeq}
\end{align}
where $I$ is  a $n\times n$ unity matrix.
By virtue of $\EM u^t\in H^2_{loc}(\mathbb R^n)^n$ and $V\in L^\infty(\mathbb R^n)$  we obtain that  $\EM f\in L^2(\mathbb R^n)^n$ and $supp (\EM f)\subset \Omega$ since $supp (V)\subset \Omega$. Therefore, 
    by Lemma \ref{lem:ufL2},   the following estimate holds:
    \begin{align}
        \|\EM u\|_{L^2(\Omega)}&=\|(\mathcal L+\omega^2I)^{-1}\EM f\|_{L^2(\Omega)}\leq C_{n,\lambda,\mu}d(\Omega)\omega^{-1}\|\EM f\|_{L^2(\Omega)}\notag\\
        &\leq C_{n,\lambda,\mu}\varepsilon\|V\|_{L^\infty(\Omega)}\|\EM u^t\|_{L^2(\Omega)}.\label{eq:usV2}
    \end{align}
    Underlying the assumption  $ \varepsilon \|V\|_{L^\infty(\Omega)}<s$,  it is noted that the Neumann series for the Lippmann-Schwinger equation \eqref{eq:LSeq} converges in $L^2(\Omega)$ since the operator norm above satisfying $C_{n,\lambda,\mu}\varepsilon\|V\|_{L^\infty(\Omega)}<1$. 
   From \eqref{eq:LSeq}, we set
    $$
        \EM u^t=\sum_{j=1}^\infty(-\omega^2(\mathcal L+\omega^2I)^{-1}V\cdot)^j\EM u^i,
 $$
    it follows from the convergence of \eqref{eq:LSeq} that 
\begin{align}
    \|\EM u^t\|_{L^2(\Omega)}\leq \frac{1}{1-C_{n,\lambda,\mu}\varepsilon \|V\|_{L^\infty(\Omega)}}\|\EM u^i\|_{L^2(\Omega)}.\label{eq:uti21}
\end{align}
Substituting \eqref{eq:uti21} into \eqref{eq:usV2}, we obtain the estimate for $\EM u$ in \eqref{eq:usi2}.

The proof is complete.
\end{proof}

%\begin{rem}\label{rem:u-relation}

%Although the analytical approach for the medium scattering problem parallels that of the source scattering problem,   a fundamental difference between these two problems. The inclusion of the unknown total field $\EM u^t$
%  on the right-hand side of \eqref{eq:esca} presents significant technical challenges in our analysis and calculations. Therefore, Proposition \ref{prop:ust2} is  crucial to our analysis, enabling the transformation from estimating unknown fields (both scattering and total) to dealing with the known incident field.
%\end{rem}

Here and in what follows,   for simplicity, we denote $\Upsilon$ as:
   \begin{align}
       \Upsilon(\varepsilon,\|V\|_{L^\infty})=\frac{\varepsilon\|V\|_{L^\infty}}{s- \varepsilon \|V\|_{L^\infty}},\label{eq:Upsilon}
   \end{align}
  it is clear that $\Upsilon$ is non-decreasing in both of its arguments. %Therefore, for any incident wave, we have 
  %$$\|\EM u\|_{L^2}\leq \Upsilon(d(\Omega)\omega,\|V\|_{L^\infty})\|\EM u^i\|_{L^2(\Omega)}.$$
Furthermore,  we denote $\|\EM f\|_{\widetilde {C}^{\delta}(\Omega)^n}$ as:
\begin{align}
    \|\EM f\|_{\widetilde {C}^{\delta}(\Omega)^n}=\|\EM f\|_{L^\infty (\Omega)^n}+\omega^{-\delta} [\EM f]_{\delta,\Omega}.\notag%\label{eq:fdelta}
\end{align}

The following theorem establishes geometric properties of a radiating medium scatterer, characterizing the quantitative relationship between its diameter and its physical parameters at the boundary.
\begin{thm}\label{thm:medsmall}
Let $n,\lambda,\mu, \mathcal M_{max}, \varepsilon_{max},\alpha$ be positive a-priori constants and $\Omega$ be a bounded Lipschitz domain with a connected complement in $\mathbb R^n$ $(n=2,3)$.
    Let $(\Omega;\lambda,\mu ,V)$ be an inhomogeneous medium  scatterer  and  let $\varepsilon =d(\Omega)\omega$, $\varepsilon\leq \varepsilon_{max}$. %with $\delta_M\in  \mathbb R_+$ being arbitrary fixed.
     Assume that $supp(V)\subset \Omega$ and $\|V\|_{\widetilde C^{\delta}(\overline{\Omega})^n}\leq \mathcal M_{max}$ with $\delta\in (0,1]$ for $n=2$ and $\delta \in (0,1/2]$ for $n=3$. Let $\EM u^i\in L^2_{loc}(\mathbb R^n)^n$ be an incident wave. If there holds that 
    \begin{align}
        \sup_{\EM p\in \partial \Omega} \frac{|V(\EM p)\EM u^i(\EM p)|}{\|V\|_{\widetilde{C}^\delta}\|\EM u^i\|_{\widetilde{C}^\delta}} > C\varepsilon^\delta(1+(1+\Upsilon(\varepsilon_{max},\mathcal M_{max}))(1+\varepsilon)\varepsilon^{\frac{n}{2}} ),\label{eq:vui}
        %\varepsilon^\delta(1+(1+M)(1+\varepsilon_M)(1+\Upsilon(\varepsilon_M,M))\varepsilon_M ),\label{eq:vui}
     \end{align}
    then $(\Omega;\lambda,\mu,V)$ must be radiating,  where $C $ is a positive constant depending only on a-priori constants. 
\end{thm}
%\begin{lem}[\cite{DTLT}]\label{lem:usi}
 %   Consider the medium scattering problem \eqref{eq:mesca} with the unique solution $\EM u^t\in H_{loc}^1(\mathbb R^n)^n$. Let $\EM u^i$ be the incident wave and $\EM u$ be the scattered displacement wave.  Then 
 %   \begin{equation}\label{eq:usi}
 %       \|\EM u^s\|_{L^2(\Omega)^n}\leq C \omega^2\|V\|_{L^\infty(\Omega)^n}\|\EM u^i\|_{L^2(\Omega)^n}.
 %   \end{equation}
%\end{lem}
\begin{proof}%[The proof of Theorem \ref{thm:medsmall}.]
By contradiction, assume that $\EM u^\infty\equiv \EM 0$, it follows from   Rellich's Lemma that $\EM u=\EM 0$ in $\mathbb R^n\setminus \overline{\Omega}$. Note that  
\begin{align}
    \mathcal L \EM u+\omega^2\EM u=\EM f,\notag
\end{align}
    where source term $\EM f= -\omega^2V\EM u^t$ satisfies $supp (\EM f)\subset \Omega$ since $supp(V)\subset \Omega$. Let $\boldsymbol{\psi}:=\EM f|_\Omega-\omega^2\EM u|_\Omega$, then one has $\boldsymbol{\psi}=-\omega^2V\EM u^i$ on $\partial \Omega$ when $\EM u=0$ in $\mathbb R^n\setminus \overline{\Omega}$.  Similar to the argument presented  in  the proof of Theorem \ref{thm:small}, it is noted that $\int_{\Omega}\boldsymbol{\psi(\EM x)}\mathrm d\EM x=\EM 0$. Therefore, for any $\EM p\in \partial \Omega$, it yields that
    \begin{align}
        \omega^2|V(\EM p)\EM u^i(\EM p)|&=|\boldsymbol{\psi}(\EM p)|\leq \frac{1}{m(\Omega)} \left \vert \int_{\Omega}(\boldsymbol{\psi}({\EM p})-\boldsymbol{\psi(\EM x)})\mathrm d\EM x\right \vert \leq [\boldsymbol{\psi}]_{\delta,\Omega} d(\Omega)^\delta\notag\\
        %&\leq [\boldsymbol{\psi}]_{\delta,\Omega}\frac{1}{m(\Omega)}\int_{\Omega}|\EM p-\EM x|^\delta\mathrm d\EM x\notag\\
       &
       \leq d(\Omega)^\delta\omega^\delta\|\boldsymbol{\psi}\|_{\widetilde{C}^\delta}\leq \varepsilon^\delta( \|\EM f\|_{\widetilde C^\delta}+\omega^{ 2}\|\EM u\|_{\widetilde C^\delta})\notag\\
    & \leq \varepsilon^\delta\omega^2( \|V\|_{\widetilde C^\delta}\|\EM u^i\|_{\widetilde C^\delta}+ (1+\mathcal M_{max}) \|\EM u\|_{\widetilde C^\delta})\label{eq:meds2}.
    \end{align}
    By virtue of Lemma \ref{lem:uholder} and the fact that $\|V\|_{L^\infty}\leq \|V\|_{\widetilde{C}^\delta}$,   we obtain
    \begin{align}
        \|\EM u\|_{\widetilde{C}^\delta}
        &\leq C_{n,\lambda,\mu}\omega^{\frac{n}{2}-1}(\omega^{-1}+d(\Omega))\|\omega^2V\EM u^t\|_{L^2}\notag\\
        &\leq C_{n,\lambda,\mu}\omega^{\frac{n}{2}}(1+\varepsilon) \|V\|_{\widetilde C^\delta}\|\EM u^t\|_{L^2}.\label{eq:meds3}
    \end{align}
    Moreover, by \eqref{eq:Upsilon} and the estimate $\|\EM u^i\|_{L^2(\Omega)}\leq \sqrt{m(\Omega)}\|\EM u^i\|_{L^\infty(\Omega)}$, we have
    \begin{align}
        \|\EM u^t\|_{L^2(\Omega) }&\leq \|\EM u^i\|_{L^2 (\Omega)}+\|\EM u\|_{L^2(\Omega) }\leq (1+\Upsilon (\varepsilon,\|V\|_{L^\infty}))\|\EM u^i\|_{L^2(\Omega)}\notag\\
       % &\leq C_n(1+\Upsilon (\varepsilon,\|V\|_{L^\infty}))\varepsilon^\frac{n}{2}\omega^{-\frac{n}{2}}\|\EM u^i\|_{\widetilde C^{\delta}(\overline{\Omega})}\notag\\
        &\leq C_n(1+\Upsilon (\varepsilon_{max},\mathcal M_{max}))\varepsilon^\frac{n}{2}\omega^{-\frac{n}{2}}\|\EM u^i\|_{\widetilde C^{\delta}(\overline{\Omega})}. \label{eq:meds4}
    \end{align}
Combining \eqref{eq:meds3} and \eqref{eq:meds4}  with \eqref{eq:meds2}, it yields that 
\iffalse
\begin{align}
    [\boldsymbol{\psi}]_{\delta,\Omega}&\leq \omega^{\delta+2}\|V\|_{\widetilde C^\delta}\|\EM u^i\|_{\widetilde C^\delta}\notag\\
    &+ C_{n,\lambda,\mu} \omega^{\delta+2}(1+M)(1+\varepsilon)(1+\Upsilon(\varepsilon_{max},\mathcal M_{max}))\varepsilon^{\frac{n}{2}} \|V\|_{\widetilde C^\delta}\|\EM u^i\|_{\widetilde C^\delta}\notag\\
    %omega^{\delta+2}\kappa_s^{\frac{n}{2}-1}(\kappa_p^{-1}+d(\Omega))\omega^2\|V\|_{\widetilde C^\delta}(1+\omega^2M)\varepsilon^\frac{n}{2}\omega^{-\frac{n}{2}}\|\EM u^i\|_{\widetilde C^{\delta}(\Omega)}\notag\\
    &\leq C \omega^{\delta+2}(1+(1+\Upsilon(\varepsilon_{max},\mathcal M_{max}))(1+\varepsilon) \varepsilon^\frac{n}{2} )\|V\|_{\widetilde C^\delta}\|\EM u^i\|_{\widetilde C^\delta},\label{eq:meds5}
    %(1+(1+\omega^2M)(1+\varepsilon)\varepsilon^\frac{n}{2})
\end{align}
where $C$ is a constant depending on a-priori constant.
Combining \eqref{eq:meds5} with \eqref{eq:meds1},  one has
\fi

\begin{align}
    \left\vert \frac{V(\EM p)\EM u^i(\EM p)}{\|V\|_{\widetilde{C}^\delta}\|\EM u^i\|_{\widetilde{C}^\delta}}\right \vert \leq C\varepsilon^\delta(1+(1+\Upsilon(\varepsilon_{max},\mathcal M_{max}))(1+\varepsilon)\varepsilon^{\frac{n}{2}} ).\notag
\end{align}
Taking the supremum over $\EM p\in \partial \Omega$, we arrive at a contradiction with \eqref{eq:vui}.

The proof of complete.
\end{proof}

By virtue of  Theorem \ref{thm:medsmall}, we obtain the following corollary directly.
\begin{cor}\label{cor:non-me1}
Under the same configurations as in Theorem \ref{thm:medsmall}, if the medium $(\Omega;\lambda,\mu, V)$ is non-radiating, namely $\EM u^\infty=\EM 0$, then one has
\begin{align}
        \sup_{\EM p\in \partial \Omega} \frac{|V(\EM p)\EM u^i(\EM p)|}{\|V\|_{\widetilde{C}^\delta}\|\EM u^i\|_{\widetilde{C}^\delta}} \leq C\varepsilon^\delta(1+(1+\Upsilon(\varepsilon_{max},\mathcal M_{max}))(1+\varepsilon)\varepsilon^{\frac{n}{2}} ),\notag
        %\varepsilon^\delta(1+(1+M)(1+\varepsilon_M)(1+\Upsilon(\varepsilon_M,M))\varepsilon_M ),\label{eq:vui}
     \end{align}
	where $C$ is a positive constant depending on a-priori constants.
\end{cor}

\subsection{Radiating property of an inhomogeneous medium with high-curvature boundary points} 
We now investigate the inhomogeneous medium scattering problem for scatterers containing admissible $K$-curvature points. Theorem~\ref{thm:medKpoint1} establishes an exact quantitative relationship between the physical parameters of an inhomogeneous medium and the curvature $K$ at admissible high-curvature boundary points.%Specifically, we are going to  localize and generalize the ``smallness" result to scatterers of arbitrary size.
\begin{thm}\label{thm:medKpoint1}
 Let $L,M,\omega,\lambda,\mu,D,\Xi,M_i, n,\alpha$ be the positive a-priori constants. 
    Let $(\Omega;\lambda,\mu,V)$ be an inhomogeneous medium scatterer in $\mathbb R^n$ $(n=2,3)$, where $\Omega$ is a bounded Lipschitz domain with a connected complement. Suppose that $\Omega$ has diameter at most $D$ and $supp(V)\subset \Omega$ and  $V\in C^{\alpha}(\overline{\Omega})^n\cap H^1(\Omega)^n$, {$\alpha \in (0,1)$ for $n=2$ and $\alpha \in (1/3,1)$ for $n=3$}.  Consider the medium  scattering problem \eqref{eq:mesca} associated with the radiation condition \eqref{eq:rad}. 
    Suppose that $\mathbf q\in \partial \Omega$ is an admissible $K$-curvature point with parameters $K,L,M,\varsigma$ and $K\geq e$.
    Suppose that  $\max(\|\boldsymbol {V}\|_{C^\alpha(\overline {\Omega})^n},
    \|\boldsymbol {V }\|_{H^1(\Omega)^n})\leq \Xi$ and $\max(\|\EM u^i\|_{H^1(\mathbb R^n)^n},\|\EM u^i\|_{C^\alpha(\mathbb R^n)^n})\leq M_i$.
    
    For any angular frequency $\omega\in \mathbb R_{+}$, if there holds that
\begin{equation}\label{eq:med1}
\vert V(\EM q)\EM u^i(\EM q)\vert \geq \mathcal C(\ln K)^{\frac{n+1}{2}}K^{-\frac{1}{2}\min(\alpha,\varsigma)} \quad \mbox{for}\ n=2
\end{equation}
or 
\begin{equation}\label{eq:med2}
\vert V(\EM q)\EM u^i(\EM q)\vert\geq \mathcal C(\ln K)^{\frac{n+1}{2}}K^{-\frac{1}{2}\min(\alpha,\varsigma)+\frac{1}{6}},\ \min(\alpha,\varsigma)\in(1/3,1)\quad \mbox{for}\ n=3,
\end{equation}
%where  constant $\mathcal C=\mathcal C(n,M,L,\alpha,\lambda,\mu,\varsigma,\omega)\in \mathbb R_{+}$, 
where $\mathcal C$ is a positive constant depending on   a-priori constants, then $(\Omega;\lambda,\mu,V )$ must be radiating. %the source radiates a non-trivial far-field patter
\end{thm}

\iffalse
\begin{thm}\label{thm:medKpoint2}
Consider the scattering configurations as described in Theorem \ref{thm:medKpoint1}. Furthermore, 
let $\mathbf u^\infty$ be the far-field pattern corresponding to the  displacement field $\EM u$. Suppose that  $\Omega\subset \overline {B_R}$, for some $R>diam (\Omega)>0$, $\|\mathbf u \|_{H^1(B_{2R})^n}\leq \mathcal M$ for $\mathcal M>0$, and $\|\EM u^t\|\geq \mathcal N$ for some $\mathcal N>0$, then the following estimate holds:
\begin{equation}\label{eq:stafarmedium}
\|\mathbf u^\infty\|_{L^2(\mathbb S^{n-1})}\geq \mathcal M  \left [\exp   \exp \{ C |V(\EM q)|^{-\frac{4}{n+1}} K^{\frac{2(n+9)}{n+1}} \right ]^{-1},
\end{equation}
where $C$ is a constant independent of  $K$.
\end{thm}
\fi

\begin{proof}
    By contradiction, assume that $(\Omega;\lambda,\mu,V )$ is non-radiating, i.e. $\EM u^\infty\equiv \EM 0$. %It follows from   Rellich's lemma that there holds that $\EM u=\EM 0$ in $\mathbb R^n\setminus \overline{\Omega}$, which indicates that $\EM u=T_\nu \EM u=\EM 0$ on $\partial \Omega$.
    By Rellich's lemma, we have $\mathbf{u} = \mathbf{0}$ in $\mathbb{R}^n \setminus \overline{\Omega}$, which implies $\mathbf{u} = T_\nu \mathbf{u} = \mathbf{0}$ on $\partial \Omega$.
    %Let $B_R$ be a ball of radius $R>D$ such that $\overline{\Omega}\Subset B_R$. Then we have  $\EM u=T_\nu \EM u=\EM 0$ on $\partial B_R.$
    Recall that $\EM u=\EM u^t-\EM u^i$ satisfies 
   $$
            \mathcal L\EM u+\omega^2\EM u=\EM f   $$
   where  $\EM f=-\omega^2V\EM u^t$.  it is clear that $supp(\EM f)\subset \Omega$ since $supp(V)\subset \Omega.$ By virtue of  $V\in H^1(\Omega)$ and $\EM u^t\in H^1(\Omega)$, it follows that $\EM f\in H^1(\Omega)$. Then by   Lemma \ref{lem:rise reg}, we obtain that $\EM u\in C^{1,\beta}(\overline {\Omega})$ and $\|\EM u\|_{C^{1,\beta}(\overline{\Omega})}\leq C\|\EM f\|_{H^1(\Omega)},$
    where $\beta\in(0,1]$ for $n=2$ and $\beta\in (0,1/2]$ for $n=3$. 
    
    %Now, we are going to show that $\|\EM f\|_{H^1(\Omega)^n}$ and $\|\EM f\|_{C^\alpha(\overline{\Omega})^n}$  depend only on a-priori constants rather than the curvature $K$. 
    We now show that the norms $\|\mathbf{f}\|_{H^1(\Omega)^n}$ and $\|\mathbf{f}\|_{C^\alpha(\overline{\Omega})^n}$ depend  solely on  a-priori constants, independent of the curvature parameter $K$.
    By \cite[Remark 2.1]{DTLT}, we know that $\|\EM u\|_{H^1(\Omega)^n}\leq C\|\EM u^i\|_{L^2(\Omega)^n}\leq C\|\EM u^i\|_{H^1(\Omega)^n},$ where $C$ is a constant depending on  the  a-priori constants.  Therefore, it follows that $\|\EM f\|_{H^1(\Omega)^n}$ is bounded by   a-priori constants, indicating that $\EM u\in C^{1,\beta}(\overline{\Omega})$ with a norm bounded by   a-priori constants. 
    Moreover, since $\EM u^i$ is H\"older-continuous, one has $\EM u^t$ is H\"older-continuous which implies that $\EM f\in C^{\alpha}(\overline{\Omega})$ with a norm bounded by   a-priori constants. Therefore, we have $\max(\|\EM f\|_{C^{\alpha(\overline{\Omega})}},\|\EM f\|_{H^1(\Omega)})\leq C$, where $C$ is a positive constant depending on   a-priori constants.

    Thus, by Theorem \ref{thm:source rad}, it yields that
    \begin{align}
        |\EM f(\EM q)|\leq C(\ln K)^{\frac{n+1}{2}}K^{-\frac{1}{2}\min(\alpha,\varsigma)}\ \mbox{for}\ n=2\notag
    \end{align}
    or 
      \begin{align}
        |\EM f(\EM q)|\leq C(\ln K)^{\frac{n+1}{2}}K^{-\frac{1}{2}\min(\alpha,\varsigma)+\frac{1}{6}}\ \mbox{for}\ n=3\notag.
    \end{align}
    Note that $\EM f=-\omega^2V\EM u^t$ and $\EM u=\EM 0$ on  $\partial \Omega$, which indicates   that $\EM f(\EM q)=-\omega^2V(\EM q)\EM u^i(\EM q)$. We then derive a contradiction with \eqref{eq:med1} or \eqref{eq:med2}, which implies that $(\Omega;\lambda, \mu, V)$ must be radiating.
    
        The proof is complete.
\end{proof}

The following corollary can be derived from Theorem \ref{thm:medKpoint1}, showing  a geometric characterization of non-radiating inhomogeneous medium scatterer  with admissible $K$-curvature  boundary points.
\begin{cor}\label{cor:non-me2}
Under the same setup as in Theorem \ref{thm:medKpoint1}, if the medium scatterer $\Omega$ is non-radiating, then there holds that
\begin{equation}\notag
\vert V(\EM q)\EM u^i(\EM q)\vert \leq \mathcal C(\ln K)^{\frac{n+1}{2}}K^{-\frac{1}{2}\min(\alpha,\varsigma)} \quad \mbox{for}\ n=2
\end{equation}
or 
\begin{equation}\notag
\vert V(\EM q)\EM u^i(\EM q)\vert\leq  \mathcal C(\ln K)^{\frac{n+1}{2}}K^{-\frac{1}{2}\min(\alpha,\varsigma)+\frac{1}{6}},\ \min(\alpha,\varsigma)\in(1/3,1)\quad \mbox{for}\ n=3,
\end{equation}
where $\mathcal C$ is a positive constant as in Theorem \ref{thm:medKpoint1}.
\end{cor}

\iffalse
\begin{proof}[The proof of Theorem \ref{thm:medKpoint2}.] 
Recall that the  displacement filed $\EM u$ satisfies
        $\mathcal L\EM u+\omega^2\EM u=\EM f,$
where $\EM f=-\omega^2V\EM u^t$ with $supp(\EM f)\subset \Omega$ since $supp(V)\subset \Omega.$ 
By employing a similar argument as in the proof of Theorem \ref{thm:medKpoint1}, there exists a constant 
$C$
 that depends on   a-priori constants such that $\max(\|\EM f\|_{C^{\alpha}(\overline{\Omega})^n},\|\EM f\|_{H^1( {\Omega})^n})\leq C$. Therefore, by Theorem \ref{thm:sta}, we have
\begin{align}\|\EM u^\infty\|_{L^2(\mathbb S^{n-1})}\geq \mathcal M \left[\exp\ \exp\{  C_1 |\EM f(\EM q)|^{-\frac{4}{n+1}}K^{\frac{2(n+9)}{n+1}}\}\right]^{-1},\label{eq:uinf-med}
\end{align}
where $C_1$ is also a constant that depends on   a-priori constants. Under the assumption $|\EM u^t|>\mathcal N$, then we can reduce \eqref{eq:uinf-med} into \eqref{eq:stafarmedium}.

The proof is complete.
    
\end{proof}
\fi

\subsection{The geometrical properties of the transmission eigenfunctions}
Theorems \ref{thm:tans1} and \ref{thm:trans2} present the geometrical properties of transmission eigenfunctions related to system \eqref{eq:trans} under two geometric structures: the scatterer either has a small diameter or exhibits high-curvature points, respectively.
\begin{thm}\label{thm:tans1}
        Let $(\Omega; \lambda, \mu, V)$ be an inhomogeneous medium scatterer in $\mathbb{R}^n$, $n \in {2, 3}$, where $\Omega$ is a bounded Lipschitz domain with a diameter of at most $D$ and a connected complement. 
        Suppose that $\varepsilon = d(\Omega) \omega$, $V \in \widetilde{C}^\delta(\overline{\Omega})^n$, and $\inf_{\partial \Omega} |V| > 0$, where $\delta \in (0, 1]$ for $n = 2$ and $\delta \in (0, 1/2]$ for $n = 3$. Let $\EM w, \EM v \in L^2(\Omega)^n$ be a pair of transmission eigenfunctions satisfying the transmission eigenvalue system \eqref{eq:trans} corresponding to the transmission eigenvalue $\omega$.
    %\begin{align}
    %    \sup_{\partial \Omega}\frac{|\EM w|}{\|\EM w\|_{\widetilde C^\delta}}\leq C\frac {\|V\|_{\widetilde C^\delta}}{\inf_{\partial \Omega}|V|}\varepsilon^\delta(1+(1+\delta)\varepsilon^{\frac{n}{2}}),\notag
   % \end{align}
    %  
      If $\|\EM w\|_{\widetilde C^\delta(\overline{\Omega})^n}\leq 1$, then there holds that
      \begin{align}
        \sup_{\partial \Omega}|\EM w|\leq C\frac {\|V\|_{\widetilde C^\delta}}{\inf_{\partial \Omega}|V|}\varepsilon^\delta(1+(1+\varepsilon)\varepsilon^{\frac{n}{2}}),\label{eq:transmall}
    \end{align}
    where $C=C_{n,\lambda,\mu}$ is a positive constant independent of $\omega$ and $\varepsilon$. 
\end{thm}

\begin{rem}
It is noted that if $\frac{\|V\|_{\tilde{C}^\delta}}{\inf_{\partial \Omega}|V|}$ is bounded, then $\EM w$ is nearly vanishing on the boundary of $\Omega$ for sufficiently  small $\varepsilon$. 	
\end{rem}

\begin{proof}
    Since $(\EM w,\EM v)$ is a pair of transmission eigenfunctions to transmission eigenvalue system \eqref{eq:trans}, it is clear that $\EM w-\EM v=\EM0$ on $\partial \Omega$.
    Extend $(\EM w-\EM v)$ and $\EM f=-\omega^2V\EM w$ by zero to $\mathbb R^n\setminus {\overline{\Omega}}$.  Then one has
 $\EM f=-\omega^2 V\EM w$, then one has 
    $$ \mathcal L(\EM w-\EM v)+\omega^2(\EM w-\EM v)=\chi_{\Omega}\EM f\quad \mbox{in}\quad\mathbb R^n,$$ 
and  $\EM w-\EM v$ is a trivially an outgoing solution. 
      Due to $(\EM w-\EM v)^\infty\equiv \EM 0$, then by Corollary \ref{cor:non-rads}, we derive that 
    $$\sup_{\partial \Omega}|\EM f|\leq C\|\EM f\|_{\widetilde{C}^\delta} \varepsilon^\delta(1+(1+\varepsilon)\varepsilon^{\frac{n}{2}}).$$
   Dividing by $\omega^2\inf_{\partial \Omega}|V|$ and utilizing $\|\EM w\|_{\tilde{C}^{\delta}(\Omega)^n}\leq 1$,  we   obtain  \eqref{eq:transmall}.

    The proof is complete.
\end{proof}

\begin{thm}\label{thm:trans2}
Let $L,M,\varsigma,D,\omega,\lambda,\mu, \Xi, c, \alpha\in(0,1)$ and $n\in\{2,3\}$ be some positive a-priori constants. 
Let  $\Omega\subset \mathbb R^n$ be a bounded Lipschitz domain with a connected complement and with a diameter at most $D$. Assume that $\mathbf q\in \partial \Omega$ is an admissible K-curvature point with parameters $L,M,\varsigma$ and $K\geq e$. Let $\EM w,\EM v\in L^2(\Omega)^n$ be a pair of transmission eigenfunctions to the transmission eigenvalue system \eqref{eq:trans} associated with $\omega$. Assume that   $V\in C^\alpha(\overline{\Omega})^n\cap H^1(\Omega)^n$ and $\max(1, \|V\|_{C^\alpha(\overline{\Omega})^n}, \|V\|_{H^1(\Omega)^n})<\Xi$ and $|V(\EM q)|>c>0$. If $\EM w\in C^\alpha(\overline{\Omega })^n\cap H^1(\Omega)^n$  and  $\max(\|\EM w\|_ {C^\alpha(\overline{\Omega })^n},\|\EM w\|_{H^1(\Omega)^n})\leq 1$, %where $\Omega_{\rho,b}$ is defined in Definition \ref{def:ad-hi-cur}, 
then one has
\begin{equation}\label{eq:trans1}
\vert \mathbf w(\mathbf q)\vert \leq \mathcal C (\ln K)^\frac{n+1}{2}K^{-\frac{1}{2}\min(\alpha,\varsigma)} \ \mbox{for}\ n=2
\end{equation}
and 
\begin{equation}\label{eq:trans2}
\vert \mathbf w(\mathbf q)\vert \leq \mathcal C  (\ln K)^\frac{n+1}{2}K^{-\frac{1}{2}\min(\alpha,\varsigma)+\frac{1}{6}},\ \min(\alpha,\varsigma)\in(1/3,1)\ \mbox{for}\ n=3,
\end{equation}
where $\mathcal C$ is a positive constant only depending on   a-priori constants.
\end{thm}

\begin{rem}\notag
It follows directly from Theorem \ref{thm:trans2} that the transmission eigenfunction $\mathbf w$ almost vanishes at the admissible $K$-curvature point when the curvature $K$ is sufficiently large.
\end{rem}

%\begin{rem}
%The vanishing of $\mathbf w$ near the admissible $K$-curvature boundary point $\mathbf p$ implies the vanishing of $\mathbf v$ by the same trace on $\partial \Omega$ of $\mathbf w$ and $\mathbf v$.
%\end{rem}

\begin{proof}[Proof of Theorem \ref{thm:trans2}.]
Denote $\mathbf u=\mathbf w-\mathbf v$, then the transmission eigenvalue system \eqref{eq:trans} can be reduced as:
\begin{equation}\notag
\begin{cases}
\mathcal L \mathbf u+\omega^2\mathbf u=-\omega^2V\mathbf w\hspace*{0.65cm}\mbox{in}\ \Omega,\\
T_\nu\mathbf u=\mathbf u=\mathbf 0\hspace*{2cm}\mbox{on}\ \partial \Omega.
\end{cases}
\end{equation}
Let $\mathbf f=-\omega^2V\mathbf w$, it is clear that $\EM f\in C^{\alpha}(\overline{\Omega})^n\cap H^1(\Omega)^n$ due to $\EM w\in C^{\alpha}(\overline{\Omega})^n\cap H^1(\Omega)^n$ and  $V\in C^{\alpha}(\overline{\Omega})^n\cap H^1(\Omega)^n$.
%For given that  $\mathbf w\in H^2(\Omega)^n$, we infer that $\mathbf w\in C^\gamma(\overline {\Omega})^n$ where $\gamma\in (0,1]$ for $n=2$ and $\gamma\in (0,1/2]$ for $n=3$.   We then have $\mathbf f\in C^{\alpha}(\overline{\Omega})^n,\alpha=\min(\alpha_1,\gamma)$. Without loss of generality, let $\alpha=\alpha_1$ for $n=2$ and $\alpha=\min(\alpha_1,1/2)$ for $n=3$.  
By Corollary  \ref{cor:non-rad}, it follows that $\vert \mathbf f(\mathbf q)\vert \leq \mathcal C(\ln K)^\frac{1+3}{2}K^{-\frac{1}{2}\min(\alpha,\varsigma)}$ for  $n=2$  and  $\vert \mathbf f(\mathbf q)\vert \leq \mathcal C(\ln K)^\frac{n+1}{2}K^{-\frac{1}{2}\min(\alpha,\varsigma)+\frac{1}{6}},\ \min(\alpha,\varsigma)\in (1/3,1)$ for $n=3$.
Consequently, under the assumption that $|V(\mathbf{q})|>c>0$, we obtain \eqref{eq:trans1} and \eqref{eq:trans2}.

The proof is complete.
\end{proof}

\section{Uniqueness results for inverse source and medium scattering problems}\label{sec:inverse}
In this section, we establish the local and global uniqueness results for the inverse source and medium scattering problems using  a single far-field pattern measurement, where the scatterers have either a small diameter (in terms of the wavelength) or high-curvature boundary points.
\subsection{Inverse problems for the source scattering problem}
 In this subsection, we establish the local and global uniqueness results for determining the small support of the source, as well as the local uniqueness results for identifying scatterers with high-curvature boundary points. 
 Theorem \ref{thm:in2loc} provides a local uniqueness result, demonstrating that if the sources $\mathbf{f}_{i} = \chi_{\Omega_i} \boldsymbol{\varphi}_i , (i=1,2)$ associated with two scatterers $\Omega_i$ radiate the same far-field patterns, then the difference of two scatterers $\Omega_1 \Delta \Omega_2 := (\Omega_1 \setminus {\Omega_2}) \cup (\Omega_2 \setminus {\Omega_1})$ cannot contain a very small component satisfying \eqref{eq:smallo}.
 
\begin{thm}\label{thm:in2loc}
Let $\Omega_{i} \subset \mathbb{R}^n\ (i=1,2,\  n=2,3)$ be two bounded Lipschitz domains, both with connected complements $\mathbb{R}^n \setminus \overline{\Omega_i}$, and the corresponding sources $\boldsymbol{\varphi}_i \in C^\alpha(\overline{\Omega}_i)^n$, where $\alpha \in (0,1]$ for $n=2$ and $\alpha \in (0,1/2]$ for $n=3$, respectively. Let $\mathbf u_i\in H^2_{loc}(\mathbb R^n)^n$ be outgoing solutions to the source problems
\begin{align}\notag
\mathcal L\mathbf u_i+\omega^2\mathbf u_i&=\chi_{\Omega_i}\boldsymbol\varphi_i\ (i=1,2).\notag
\end{align}
If $\mathbf u_1^{\infty}=\mathbf u_2^{\infty}$, then the difference $\Omega_1\Delta\Omega_2$ cannot contain any  component $\Omega_0$ satisfying  
\begin{equation}\label{eq:smallo}
\frac{\sup \|\boldsymbol \varphi\|_{L^2(\partial \Omega_0)}}{{\omega^{-\alpha}}[\boldsymbol \varphi]_{\alpha,\Omega_0}+\|\boldsymbol \varphi\|_{L^\infty(\Omega_0)}}>C\varepsilon^{\alpha}(1+(1+\varepsilon)\varepsilon^{n/2}),
\end{equation}
and there exists an unbounded path $\gamma\subset \mathbb R^n\setminus \overline{\Omega_1\cap \Omega_2}$ connecting $\Omega_0$ to infinity,  
where $\varepsilon=d(\Omega_0)\omega$, $\boldsymbol \varphi=\boldsymbol \varphi_1|_{\Omega_0}$ for $\Omega_0\subset (\Omega_1\setminus \overline{\Omega}_2)$ or $\boldsymbol \varphi=\boldsymbol \varphi_2|_{\Omega_0}$ for $\Omega_0\subset (\Omega_2\setminus \overline{\Omega}_1)$.
\end{thm}

\begin{proof}
We prove this theorem by contradiction and assume that there exist a component $\Omega_0\subset \Omega_1\Delta\Omega_2$. Without loss of generality, suppose that $\Omega_0\subset (\Omega_1\setminus \overline{ \Omega}_2)$.  It is clear that there is a bounded Lipschitz domain $G\subset \mathbb R^n$ such that $(\Omega_1\cup\Omega_2)\subset G$ and  its complement is connected. Then one has $\Omega_0$ is also a component of $G$. Denote $\mathbf w=\mathbf u_2-\mathbf u_{1}$, then one has
$$ 
\mathcal L \mathbf w+\omega^2\mathbf w=\chi_{G}(\chi_{\Omega_2}\boldsymbol \varphi_2-\chi_{\Omega_1}\boldsymbol \varphi _1),
$$
where $\mathbf w$ satisfies the Kupradze radiation condition and $\mathbf w^\infty=\mathbf 0.$ It is easily seen that the  source term above is $-\chi_{\Omega_0}\boldsymbol \varphi_1$ on $\Omega_0$. By virtue of Corollary \ref{cor:non-rads}, we have
$$\frac{\sup \|\boldsymbol \varphi_1\|_{L^2(\partial \Omega_0)^n}}{\omega^{-\alpha}[\boldsymbol \varphi_1]_{\alpha,\Omega_0}+\|\boldsymbol \varphi_1\|_{L^\infty(\Omega_0)^n}}\leq C\varepsilon^{\alpha}(1+(1+\varepsilon)\varepsilon^{n/2}),$$
which contradicts \eqref{eq:smallo}.

 The proof is complete.
\end{proof}

By virtue of Theorem \ref{thm:in2loc}, we can directly state the following corollary, which shows that for a fixed operating frequency, if the supports of two different sources are small and disjoint, they cannot produce the same far-field pattern.
\begin{cor}\label{cor:inv1}
Let $\Omega_i\ (i=1,2)$ be two bounded Lipschitz domains in $\mathbb R^n$ $(n=2,3)$ with connected complements $\mathbb R^n\setminus \overline{\Omega}_i$. Consider the source problems: $\mathcal L \EM u_i+\omega^2\EM u_i=\chi_{\Omega_i}\boldsymbol{\varphi}_i$ with $\boldsymbol \varphi_i\in C^\alpha(\overline{\Omega}_i)^n$, where $\alpha\in (0,1]$ for $n=2$ and $\alpha\in(0,1/2]$ for $n=3$.  Let $\varepsilon_{min}$ be defined as follows:
\begin{align}
&\min\left(\frac{\sup_{\partial \Omega_1} |\boldsymbol \varphi_1|}{\omega^{-\alpha}[\boldsymbol \varphi_1]_{\alpha,\Omega_1}+\|\boldsymbol \varphi_1\|_{L^\infty(\Omega_1)^n}}, \frac{\sup_{\partial \Omega_2}|\boldsymbol \varphi_2|}{\omega^{-\alpha}[\boldsymbol \varphi_2]_{\alpha,\Omega_2}+\|\boldsymbol \varphi_2\|_{L^\infty(\Omega_2)^n}}\right)\notag\\
&\quad=C\varepsilon_{min}^{1/2}(1+(1+\varepsilon_{min})\varepsilon_{min}^{n/2})\label{eq:e0}
\end{align}
or smaller. If $\omega d(\Omega_i)<\varepsilon_{min} \ (i=1,2)$ and $\overline{\Omega}_1 \cap \overline{\Omega}_2=\emptyset$, then $\mathbf u_1^\infty\not=\mathbf u_2^\infty.$
\end{cor}

\begin{proof}
By contradiction, assume that $\EM u_1^\infty=\EM u_2^\infty$. Since $\overline{\Omega}_1 \cap \overline{\Omega}_2=\emptyset$, a similar argument to that in the proof of Theorem \ref{thm:in2loc} allows us to derive that 
$$\frac{\sup_{\partial \Omega_i} |\boldsymbol \varphi_i|}{\omega^{-\alpha}[\boldsymbol \varphi_i]_{\alpha,\Omega_i}+\|\boldsymbol \varphi_i\|_{L^\infty(\Omega_i)^n}}\leq C\varepsilon_{min}^{1/2}(1+(1+\varepsilon_{min})\varepsilon_{min}^{n/2})\ \mbox{for}\ i=1,2$$
which contradicts \eqref{eq:e0}. 

The proof is complete.
\end{proof}

Corollary \ref{cor:inv3} establishes a  global unique identifiability result for inverse scattering problems. Specifically, if two source supports, each comprising a finite number of mutually disjoint small scatterers, generate identical far-field patterns, then they must have the same number of components. Moreover, after re-indexing the components, the pairwise intersections of corresponding components are non-empty. Prior to this, we define the admissible class $\mathcal{A}$.

%Corollary \ref{cor:inv3} can be proved by modifying the proof of Theorem \ref{thm:in2loc}, which is omitted here.

\begin{defn}\label{def:ad1}
	Let $\Omega\subset \mathbb R^n$ $(n=2,3)$ be a bounded Lipschitz domain with a connected complement $\mathbb R^n\setminus {\overline{\Omega}}$. Let $\EM u\in H_{loc}^2(\mathbb R^n)^n$ be the unique outgoing solution to 
	$\mathcal L \EM u+\omega^2\EM u=\EM f$, with source $ \EM f=\chi_{\Omega}\boldsymbol{\varphi}.$
		 Suppose that $\Omega$ be a collection of the mutually disjoint scatterers:
	$$\Omega=\bigcup_{j=1}^{N}\Omega^j\ \mbox{and}\   \overline{\Omega}^i\cap\overline{\Omega}^j=\emptyset\ \mbox{for}\ i\not =j\ (i,j=1,\dots,N\in \mathbb R^n),$$
	 then   $(\Omega;\EM f)$ belongs to an admissible class $\mathcal A$ if the following admissible conditions fulfilled:
	
	\begin{itemize}
		\item [(a)]  $\boldsymbol \varphi\in C^\alpha(\overline{\Omega})^n$ with $\alpha\in (0,1]$ for $n=2$ and $\alpha\in (0,1/2]$ for $n=3$.
		\item [(b)] The diameter of $\Omega$ is denoted by $\varepsilon = d(\Omega) \omega$ and satisfies
		\begin{equation}\notag
\frac{\sup \|\boldsymbol \varphi\|_{L^2(\partial \Omega )}}{{\omega^{-\alpha}}[\boldsymbol \varphi]_{\alpha,\Omega }+\|\boldsymbol \varphi\|_{L^\infty(\Omega )}}>C\varepsilon^{\alpha}(1+(1+\varepsilon)\varepsilon^{n/2}),
   \end{equation}
   where $C$ is a positive constant depending on $n,\lambda,\mu, \alpha.$
	\end{itemize}
\end{defn}

\begin{cor}\label{cor:inv3}
Let $(\Omega_i;\EM f_i)$ be two sources belonging to the admissible class $\mathcal A$ in $\mathbb R^n\ (n=2,3)$, where 
	$\Omega_1=\bigcup_{j=1}^N \Omega_1^{j}, \ \Omega_2=\bigcup_{\ell=1}^M \Omega_2^{\ell}$. Let $\varepsilon_{min}$ be given in \eqref{eq:e0}.  
	%$\varepsilon_1=d(\Omega_1)\omega$, $\varepsilon_2=d(\Omega_2)\omega$ and $\max(\varepsilon_1,\varepsilon_2)\leq \varepsilon_{min}$ with $\varepsilon_{min}$ given in \eqref{eq:e0}.  
Suppose that  each $\Omega_1^j$ and $\Omega_2^\ell$ has a diameter of at most $\varepsilon_{min}\omega^{-1}$ 
and  $d(\Omega_1^{j_1},\Omega_1^{j_2})>2\varepsilon_{min}\omega^{-1}$ for $j_1\not=j_2$, $d(\Omega_2^{\ell_1},\Omega_2^{\ell_2})>2\varepsilon_{min}\omega^{-1}$ for $\ell_1\not=\ell_2$.

Let $\EM u_i$ and $\EM u^\infty_i$ be the displacement fields and the far-field patterns corresponding to $(\Omega_i;\EM f_i)$, $i=1,2$, respectively. If $\EM u^\infty_1=\EM u^\infty_2$, then it follows that  $M=N$, and under re-indexing, $\overline{\Omega}_1^{j}\cap \overline{\Omega}_2^j\not=\emptyset$, $(j=1,\dots,N)$. 
\end{cor}

\begin{proof}
	We will prove this theorem by contradiction. Without loss of generality, assume that there is a component $\Omega_1^{j_0}\subset \Omega_1$ such that $\overline{\Omega}_1^{j_0}\cap\overline{\Omega}_2=\emptyset.$ Denote $\tilde {\EM u}=\EM u_1-\EM u_2$, then we have   $\tilde {\EM u}^\infty =\EM 0$ and 
	$$\mathcal L\tilde {\EM u}+\omega^2\tilde {\EM u}=\EM f_1\quad \mbox{in}\quad  \Omega_1^{j_0}.$$
By virtue of Corollary \ref{cor:non-rads}, one has
       $$\frac{\sup \|\boldsymbol \varphi\|_{L^2(\partial \Omega_1^{j_0})}}{{\omega^{-\alpha}}[\boldsymbol \varphi]_{\alpha,\Omega_1^{j_0}}+\|\boldsymbol \varphi\|_{L^\infty(\Omega_1^{j_0})}}\leq C\varepsilon_{min}^{\alpha}(1+(1+\varepsilon_{min})\varepsilon_{min}^{n/2}).$$
which contradicts the admissible condition (b) in Definition \ref{def:ad1}. Therefore, we obtain that  $\overline{\Omega}_1^{j_0}\cap\overline{\Omega}_2\not=\emptyset$. Furthermore, since the distance between two components is twice as large as their diameters, it is clear that a component of $\Omega_1$ can interact with at most one component of $\Omega_2$. Therefore, we conclude that $M = N$.

The proof is complete.
\end{proof}

Next, we shall establish the local uniqueness results for determining the support of a source with admissible $K$-curvature points using a single far-field pattern measurement. 
%First, denote the difference between two scatterers $\Omega_1$ and $\Omega_2$ as:
%\begin{align}\label{eq:defference}
%\Omega_1\aleph\Omega_2:=(Co(\Omega_1)\setminus Co( \Omega_2))\cup(Co(\Omega_2) \setminus Co(\Omega_1)),
%\end{align}
%where $Co(\Omega_i)\ (i=1,2)$ is the convex hull of $\Omega_i.$ 
Theorem \ref{thm:in3} demonstrates that if two admissible sources produce the same far-field pattern, then the difference between their supports cannot contain a high-curvature point connected to infinity. We first define the admissible class $\mathcal B$ for the support of a source that possesses admissible $K$-curvature points.

\begin{defn}\label{def:ad2}
Let $\Omega\in \mathbb R^n\ (n=2,3)$ be a bounded Lipschitz domain with a connected complement. Consider the source problem $\mathcal L\EM u+\omega^2\EM u=\EM f$,  $\EM f=\chi_{\Omega}\boldsymbol{\varphi}$. 	Let $L,M, n,\lambda,\mu,\omega,M,\Xi_M,\alpha$ be the positive a-priori constants. Suppose that $\Omega$ has an admissible $K$-curvature point with parameters $K,L, M, \varsigma$. We say that a source $(\Omega;\EM f)$ belongs to an admissible class $\mathcal B$ if the following conditions are satisfied:
\begin{itemize}
	\item [(a)] $\boldsymbol{\varphi}\in C^\alpha(\Omega)$ and $\max(\| \boldsymbol \varphi \|_{C^\alpha(\overline{\Omega })},\| \boldsymbol \varphi \|_{H^1( {\Omega })})<\Xi_M$, where $\alpha\in (0,1)$ for $n=2$ and $\alpha\in(1/3,1)$ for $n=3$. 
	%Furthermore,  $$\Xi_m < \inf_{\partial \Omega} |\boldsymbol{\varphi}|,\quad.$$
	\item[(b)] $\boldsymbol{\varphi}$ has a lower bound at the admissible $K$-curvature point associated with  $K$:
	$$ \vert \boldsymbol{ \varphi}(\mathbf q)\vert \geq \mathcal R(\ln K)^{\frac{n+1}{2}}K^{-\frac{1}{2}\min(\alpha,\varsigma)}\ \mbox{for}\ n=2,$$
or
$$\vert \boldsymbol {\varphi}(\mathbf q)\vert \geq\mathcal R(\ln K)^{\frac{n+1}{2}}K^{-\frac{1}{2}\min(\alpha,\varsigma)+\frac{1}{6}},\ \varsigma \in(1/3,1)\ \mbox{for}\ n=3,$$
where $\mathcal R$ is a positive constant depending on   a-priori constants. 
\end{itemize}
\end{defn}

\begin{thm}\label{thm:in3}
Let $(\Omega_i;\EM f_i)$ be two sources belonging to the admissible class $\mathcal B$ in $\mathbb R^n\ (n=2,3)$ with $\EM f_i=\chi_{\Omega_i}\boldsymbol{\varphi}_i$ $(i=1,2)$. Let 
 $\mathbf u_i^\infty$ be the far-field patterns corresponding to $\mathbf u_i$ and $M_\epsilon>1$ be an a-priori constant. %Assume that there exists a positive constant $\Xi>0$ such that  $\max(\| \boldsymbol \varphi_i\|_{C^\alpha(\overline{\Omega_i}},\| \boldsymbol \varphi_i\|_{H^1( {\Omega_i}})<\Xi$. 
 If $\mathbf u_1^\infty=\mathbf u_2^\infty$, then the difference $\Omega_1\Delta\Omega_2$  cannot have  an admissible $K$-curvature point $\mathbf q$ with the parameters $K,L,M,\varsigma$ such that neither $d(\mathbf q,\Omega_1)>M_\epsilon/K$ for $\mathbf q\in \partial \Omega_2\setminus \partial \Omega_1$ nor $d(\mathbf q,\Omega_2)>M_\epsilon/K$ for $\mathbf q\in \partial \Omega_1\setminus \partial \Omega_2$, and there exists an unbounded path $\gamma\subset \mathbb R^n\setminus \overline{\Omega_1\cup\Omega_2}$ connecting $\EM q$ to infinity. %can connect to infinity through $\mathbb R^n\setminus \overline{\Omega_1\cup\Omega_2}$.
\end{thm}

\begin{proof}
 Without loss of generality, and by contradiction, suppose that $\Omega_1\setminus \Omega_2$ has an admissible $K$-curvature point $\mathbf{q} \in \partial \Omega_1\setminus \partial \Omega_2$ and satisfies $d(\mathbf{q}, \Omega_2) > M_\epsilon/K$. Then, we know that there is a bounded domain $\Omega_{\rho,b}$, as defined in Definition \ref{def:ad-hi-cur}, containing $\mathbf{q}$, and $\overline{\Omega}_{\rho,b} \cap \overline{\Omega}_2 = \emptyset$.
 
Define $\tilde {\mathbf u}=\mathbf u_1-\mathbf u_2 $, then one has $\mathcal L\tilde{\mathbf u}+\omega^2\tilde{\mathbf u}=\mathbf 0$ in $\mathbb R^n\setminus{\overline{\Omega_1\cup\Omega_2}}$. By  Rellich's Lemma, it is evident that $\tilde {\mathbf u}\equiv\mathbf 0 $ in $\mathbb R^n\setminus{\overline{\Omega_1\cup\Omega_2}}$ since $\mathbf u_1^\infty=\mathbf u_2^\infty$. Additionally, there is a bounded Lipschitz domain $G$ such that  $(\Omega_1\cup\Omega_2) \subset G$  and 
$$\mathcal L\tilde{\mathbf u}+\omega^2\tilde{\mathbf u}=\boldsymbol \varphi\ \mbox{in}\ G,$$
where $\boldsymbol \varphi=\chi_{\Omega_1}\boldsymbol \varphi_1-\chi_{\Omega_2}\boldsymbol \varphi_2.$ We can observe that $\boldsymbol{\varphi}(\mathbf{q}) = \boldsymbol{\varphi}_1(\mathbf{q})$. 
Since $\Omega_{\rho,b} \subset G$, Corollary~\ref{cor:non-rad} yields 
$$\vert \boldsymbol \varphi_1(\mathbf q)\vert \leq \mathcal R(\ln K)^{\frac{n+1}{2}}K^{-\frac{1}{2}\min(\alpha,\varsigma)}\ \mbox{for}\ n=2$$
or
$$\vert \boldsymbol \varphi_1(\mathbf q)\vert \leq \mathcal R(\ln K)^{\frac{n+1}{2}}K^{-\frac{1}{2}\min(\alpha,\varsigma)+\frac{1}{6}},\ \min(\alpha,\varsigma)\in(1/3,1)\ \mbox{for}\ n=3,$$
where $\mathcal R$ is a positive constant independent of $K$. Therefore, we arrive a contradiction with the admissible condition (b) in Definition \ref{def:ad2}.

The proof is complete.
\end{proof}

\subsection{Inverse problems for the medium scattering problem}
Finally, we briefly study the inverse problems for the medium scattering problem related to \eqref{eq:mesca}. The inverse problems for the medium scattering problem in this paper focus on uniquely determining the shape of the inhomogeneous medium scatterer from the knowledge of the far-field pattern using a single measurement, regardless of its physical configurations. As applications of Theorems \ref{thm:tans1} and \ref{thm:trans2}, we establish the local and global uniqueness results for determining inhomogeneous medium scatterers.
%By reducing the inverse problems for the medium scattering problem \eqref{eq:mesca} to the source problem \eqref{eq:lame}, we can extend the results for inverse problems from Theorems \ref{thm:in2} and \ref{thm:in3} to the medium scattering scenario through appropriate modifications. Therefore, we omit the proofs of Theorems \ref{thm:medinv1} and \ref{thm:medinv2}.
Similar to the definitions of admissible classes $\mathcal{A}$ and $\mathcal{B}$ for sources, we shall define the admissible classes $\mathcal{A}'$ and $\mathcal{B}'$ for inhomogeneous medium scatterers.
\begin{defn}\label{def:medad1}
	Let $n,\lambda,\mu, \mathcal M_{min}, \mathcal M_{max},\delta$ be the positive  a-priori constants. Let $(\Omega;\lambda,\mu,V)$ be an inhomogeneous medium scatterer in $\mathbb R^n$ ($n=2,3$) with a connected complement and satisfy the medium scattering problem \eqref{eq:mesca}.  Let $\EM u^i\in L^2_{loc}(\mathbb R^n)^n$  be an incident wave and $\EM u^t$ be the total field. Suppose that $\Omega$ be a collection of the mutually disjoint scatterers:
	$$\Omega=\bigcup_{j=1}^{N}\Omega^j\ \mbox{and}\   \overline{\Omega}^i\cap\overline{\Omega}^j=\emptyset\ \mbox{for}\ i\not =j\ (i,j=1,\dots,N\in \mathbb R^n),$$
	 then the medium scatterer $(\Omega;\lambda,\mu,V)$ belongs to an admissible class $\mathcal A^{'}$ if the following admissible conditions are fulfilled:
	
	\begin{itemize}
	\item[(a)]	$supp (V)\subset \Omega$,  $\inf_{\partial \Omega}|V|\geq \mathcal M_{min}$ and $\|V\|_{\tilde{\Omega}^\delta(\overline{\Omega})^n}\leq \mathcal M_{max}$, where $\delta\in (0,1]$ for $n=2$ and $\delta \in (0,1/2]$ for $n=3$. 
	\item [(b)] Let $\varepsilon=d(\Omega)\omega$, the total field $\EM u^t$ satisfies that
	$\sup_{\partial \Omega}|\EM u^t|\geq  C \varepsilon^{\delta}(1+(1+\varepsilon)\varepsilon^{\frac{n}{2}}),$
  %Let $\Upsilon$ be given by \eqref{eq:Upsilon}, $V$ and $\EM u^i$ satisfies 
	%\begin{align}
     %   \sup_{\EM p\in \partial \Omega} \frac{|V(\EM p)\EM u^i(\EM p)|}{\|V\|_{\widetilde{C}^\delta}\|\EM u^i\|_{\widetilde{C}^\delta}} > C\varepsilon^\delta(1+(1+\Upsilon(\varepsilon_{max},\mathcal M_{max}))(1+\varepsilon)\varepsilon^{\frac{n}{2}} ),\label{eq:vui}
        %\varepsilon^\delta(1+(1+M)(1+\varepsilon_M)(1+\Upsilon(\varepsilon_M,M))\varepsilon_M ),\label{eq:vui}
     %\end{align}
     where $C$ is a positive constant depending on  a-priori constants.
	\end{itemize}
\end{defn}

In Theorem \ref{thm:medinv1}, we establish a global uniqueness result demonstrating that if two medium scatterers consist of a finite number of disjoint small scatterers and produce the same far-field pattern, then they have the same number of components.
\begin{thm}\label{thm:medinv1}
	Let $(\Omega_i;\lambda,\mu,V_i)$ $(i=1,2)$ be two inhomogeneous medium scatterers belonging to the admissible class $\mathcal {B}^{'}$, where 
	$\Omega_1=\bigcup_{j=1}^N \Omega_1^{j}, \ \Omega_2=\bigcup_{\ell=1}^M \Omega_2^{\ell}$. Let $\varepsilon_{min}$ be given in 
	$$\min \left(\sup|\EM u_1^t|_{\partial \Omega_1}, \sup|\EM u_2^t|_{\partial \Omega_2}\right)=C \varepsilon_{min}^{\delta}(1+(1+\varepsilon_{min})\varepsilon_{min}^{\frac{n}{2}}),$$ where $C$ is a positive constant independent of  $\varepsilon$. 
	 %$\varepsilon_1=d(\Omega_1)\omega$, $\varepsilon_2=d(\Omega_2)\omega$ and $\varepsilon_{min}=\min(\varepsilon_1,\varepsilon_2)$.  
Suppose that  each $\Omega_1^j$ and $\Omega_2^\ell$ has a diameter of at most $\varepsilon_{min}\omega^{-1}$ and $d(\Omega_1^{j_1},\Omega_1^{j_2})>2\varepsilon_{min}\omega^{-1}$ for $j_1\not=j_2$, $d(\Omega_2^{\ell_1},\Omega_2^{\ell_2})>2\varepsilon_{min}\omega^{-1}$ for $\ell_1\not=\ell_2$. 
\iffalse
where $\varepsilon_{min}$ is given by
\begin{align}
 %\sup_{\partial \Omega}|\EM w|\leq C\frac {\|V\|_{\widetilde C^\delta}}{\inf_{\partial \Omega}|V|}\varepsilon^\delta(1+(1+\varepsilon)\varepsilon^{\frac{n}{2}}),\label{eq:transmall}
	%&\min\left(\sup_{\EM p_1\in \partial \Omega_1^j} \frac{|V(\EM p_1)\EM u^i(\EM p_1)|}{\|V\|_{\widetilde{C}^\delta}\|\EM u^i\|_{\widetilde{C}^\delta}}, \sup_{\EM p_2\in \partial \Omega_2^\ell} \frac{|V(\EM p_2)\EM u^i(\EM p_2)|}{\|V\|_{\widetilde{C}^\delta}\|\EM u^i\|_{\widetilde{C}^\delta}}\right)\notag\\
	\min\left (\sup_{\partial \Omega_1^j}|\EM u_1^t|, \sup_{\partial \Omega_2^\ell}|\EM u^t_2|\right)
	=C\varepsilon_{min}^\delta(1+(1+\varepsilon_{min})\varepsilon_{min}^{\frac{n}{2}}).\notag
	%C\varepsilon_{\min}^\delta(1+(1+\Upsilon(\varepsilon_{max},\mathcal M_{max}))(1+\varepsilon)\varepsilon^{\frac{n}{2}} )\notag
\end{align}
\fi
	 
	Let $\EM u^t_i$ and $\EM u^\infty_i$ be the total fields and the far-field patterns corresponding to $\Omega_i$, $i=1,2$, respectively. If $\EM u^\infty_1=\EM u^\infty_2$, then one has $M=N$ and under re-indexing, $\overline{\Omega}_1^{j}\cap \overline{\Omega}_2^j\not=\emptyset$, ($j=1,\dots,N$). 
	\end{thm}

\begin{proof}
We will prove this theorem by contradiction. Without loss of generality, assume that there is a component $\Omega_1^{j_0}\subset \Omega_1$ such that $\overline{\Omega}_1^{j_0}\cap\overline{\Omega}_2=\emptyset.$ Then one has 
\begin{equation}
	\begin{cases}\notag
	\mathcal L \EM u_1^t+\omega^2(1+V_1)\EM u_1^t=\EM 0\hspace{0.7cm}\mbox{in}\ \Omega_1^{j_0},\\
	\mathcal L \EM u_2^t+\omega^2\EM u_2^t=\EM 0\hspace{2.1cm}\mbox{in}\ \Omega_1^{j_0},\\
	\EM u_1^t=\EM u_2^t, \ T_\nu \EM u_1^t=\T_\nu \EM u_2^t \hspace{1cm}\mbox{on}\ \partial \Omega_1^{j_0}.
\end{cases}
\end{equation}
By virtue of Theorem \ref{thm:tans1}, one has
       $$ \sup_{\partial \Omega_1}|\EM u^t_1|\leq C  \varepsilon_{min}^\delta(1+(1+\varepsilon_{min})\varepsilon_{min}^{\frac{n}{2}}),$$
which contradicts the admissible condition (b) in Definition \ref{def:medad1}. Therefore, we obtain that  $\overline{\Omega}_1^{j_0}\cap\overline{\Omega}_2\not=\emptyset$. Similar to the argument in the proof of Corollary \ref{cor:inv3}, we conclude that $M = N$.
The proof is complete.
\end{proof}

Next, we introduce the definition of the admissible class $\mathcal{B}'$ for inhomogeneous medium scatterers with admissible $K$-curvature points and establish a local uniqueness result. The proof can be obtained by combining a similar approach to Theorem \ref{thm:medinv1} with Theorem \ref{thm:trans2}. Therefore, we omit it here.

\begin{defn}\label{def:ad-me2}
Let $L,M,\varsigma,n,\omega,\lambda,\mu, D,\Xi,M_i, \alpha$ be the positive a-priori constants. Let $(\Omega;\lambda,\mu,V)$ be an inhomogeneous medium scatterer in $\mathbb R^n\ (n=2,3)$ satisfying the medium scattering problem \eqref{eq:mesca}, where $\Omega$ is a bounded Lipschitz domain with a connected complement. Suppose that $\Omega$ has a diameter of at most $D$ and  $\EM q\in \partial \Omega$ is an admissible $K$-curvature point with parameters $K,L,M,\varsigma$, $K\geq e$. Let $\EM u^i$ be an incident wave and satisfy   $\max(\|\EM u^i\|_{H^1(\mathbb R^n)^n},\|\EM u^i\|_{C^\alpha(\mathbb R^n)^n})\leq M_i$. 
Then the medium scatterer $(\Omega;\lambda,\mu,V)$ is said to belong to the admissible class $\mathcal B^{'}$ if the following admissible conditions are satisfied:
\begin{itemize}
\item [(a)]$supp(V)\subset \Omega$, $V\in C^\alpha(\overline{\Omega})^n\cap H^1(\Omega)^n$ and $\max(\|\boldsymbol {V}\|_{C^\alpha(\overline {\Omega})^n},
    \|\boldsymbol {V }\|_{H^1(\Omega)^n})\leq \Xi$, where $\alpha\in (0,1)$ for $n=2$ and $\alpha \in (1/3,1)$ for $n=3$.
\item [(b)]
The total field $\EM u^t\in C^\alpha(\overline{\Omega})^n\cap H^1(\Omega)^n$ and  
\begin{equation}\notag
\vert \EM u^t(\EM q)\vert \geq \mathcal C(\ln K)^{\frac{n+1}{2}}K^{-\frac{1}{2}\min(\alpha,\varsigma)} \quad \mbox{for}\ n=2
\end{equation}
or 
\begin{equation}\notag
\vert \EM u^t(\EM q)\vert\geq \mathcal C(\ln K)^{\frac{n+1}{2}}K^{-\frac{1}{2}\min(\alpha,\varsigma)+\frac{1}{6}},\ \min(\alpha, \varsigma)\in(1/3,1)\quad \mbox{for}\ n=3,
\end{equation}
where $\mathcal C$ is a positive constant depending on   a-priori constants.
\end{itemize}
\end{defn}

\begin{thm}\label{thm:medinv2}
Let $(\Omega_i; \lambda, \mu, V_i)$ $(i=1,2)$ be two medium scatterers in $\mathbb{R}^n$ $(n=2,3)$ and belonging to the admissible class $\mathcal{B}'$. Let $\EM u_i^\infty$ be the far-field patterns corresponding to $\Omega_i$ associated with the incident wave $\EM u^i$. Let $M_\epsilon > 1$ be an a priori constant. If $\mathbf{u}_1^\infty = \mathbf{u}_2^\infty$, then there cannot exist an admissible $K$-curvature point $\mathbf{q}$ with the parameters $K, L, M, \varsigma$ such that neither $d(\mathbf{q}, \Omega_1) > M_\epsilon/K$ for $\mathbf{q} \in \partial \Omega_2 \setminus \partial \Omega_1$ nor $d(\mathbf{q}, \Omega_2) > M_\epsilon/K$ for $\mathbf{q} \in \partial \Omega_1 \setminus \partial \Omega_2$,  and there exists an unbounded path $\gamma\subset \mathbb R^n\setminus \overline{\Omega_1\cup\Omega_2}$ connecting $\EM q$ to infinity.
%$\EM q$ can connect to infinity through $\mathbb R^n\setminus \overline{\Omega_1\cup\Omega_2}$.
\end{thm}

\section*{Acknowledgements}
The work of H. Diao is supported by by National Natural Science Foundation of China  (No. 12371422) and the Fundamental Research Funds for the Central Universities, JLU. The work of X. Fei is supported by NSFC/RGC Joint Research Grant No. 12161160314. The work of H. Liu is supported by the Hong Kong RGC General Research Funds (projects 11311122, 11300821, and 11303125), the NSFC/RGC Joint Research Fund (project  N\_CityU101/21), the France-Hong Kong ANR/RGC Joint Research Grant, A-CityU203/19.


\begin{thebibliography}{99}

      % \bibitem{Ang}
%           T. S. Angell and A. Kirsch, {\it The conductive boundary condition for Maxwells equations}, SIAM J. Appl. Math., {\bf 52(6)} (1992), 1597--1610.
%
%       \bibitem{kang1}
%           H. Ammari, E. Bretin, J. Garnier, H. Kang, H. Lee and A. Wahab, {\it Mathematical methods in elasticity imaging}, Princeton University Press, (2015).

\bibitem{AR}
    G. Alessandrini  and L. Rondi, {\it Determining a sound-soft polyhedral scatterer by a single far-field measurement}, Proceedings of the American Mathematical Society, {\bf 133(6)}(2005), 1685--1691.

	
%\bibitem{BDLM}
%          Z. Bai, H. Diao, H. Liu and Q. Meng, {\it Stable determination of an elastic medium scatterer by a single far-field measurement and beyond}, Calculus of Variations and Partial Differential Equations {\bf 61(5)} (2022), 170.
       
          
%\bibitem{BCL}
%G. Bao, C. Chen and P. Li, {\it Inverse random source scattering for elastic waves},  SIAM J. Numer. Anal., {\bf 55} (2017),   2616--2643. 

       \bibitem{B2018}
           E. Bl{\aa}sten, {\it Nonradiating sources and transmission eigenfunctions vanish at corner and edges}, {SIAM J.Math.Anal}, {\bf50(6)} (2018), 6255--7270.

       \bibitem{BLY}
           E. Bl{\aa}sten and Y.-H. Lin, {\it Radiating and non-radiating sources in elasticity }, Inverse Problems, {\bf 35(1)} (2019), 015005.

    %   \bibitem{BL2017}   E. Bl{\aa}sten and H. Liu, {\it On vanishing near corners of transmission eigenfunctions}, J. Funct. Anal., {\bf 273(11) } (2017), 3616--3632.

       %\bibitem{BL2020} E.Bl{\aa}sten and H. Liu, {\it Recovering piecewise constant refractive indices by a single far-field pattern}, Inverse Problems, {\bf 36(8)} (2020), 085005.

     %  \bibitem{BL2021-1} E.Bl{\aa}sten and H. Liu,  {\it On corners scattering stably and stable shape determination by a single far-field pattern}, Indiana Univ. Math. J., {\bf 70(3)} (2021), 907--947.

      % \bibitem{BL2017}
       %    E. Bl{\aa}sten and H. Liu, {\it On vanishing near corners of transmission eigenfunctions}, J. Funct. Anal, {\bf 273(11)} (2017), 3616--3632.

   %  \bibitem{BL21}
    %   E. Bl{\aa}sten and H. Liu, {\it On corners scattering stably and stable shape determination by a single far-field pattern}, Indiana University Mathematics Journal {\bf70(3)} (2021): 907--947.

       \bibitem{BL2021}
           E. Bl{\aa}sten and H. Liu, {\it Scattering by curvatures, radiationless sources, transmission eigenfunctions, and inverse scattering problems}, SIAM Journal on Mathematical Analysis, {\bf 53(4)}(2021), 3801-3837.	

      % \bibitem{BLX}
        %   E. Bl{\aa}sten, H. Liu, and J. Xiao, {\it On an electromagnetic problem in a corner and its applications}, {Analysis \& PDE}, {\bf 14(7)} (2021), 2207--2224.
           
       
           \bibitem{BPS}    
          E. Bl{\aa}sten, L. P\"aiv\" arinta and J. Sylvester,   {\it Corners always scatter}, Comm. Math. Phys, {\bf 331}(2014), 725--753.

       \bibitem{BS}
           E. Bl{\aa}sten and J. Sylvester, {\it Translation-invariant estimates for operators with simple characteristics}, Journal of Differential Equations, {\bf 263(9)}(2017), 5656-5695.


     

      % \bibitem{DR1998}
%           D. Colton and R. Kress, {\it Inverse acoustic and electromagnetic scattering theory}, Berlin: Springer, (1998).

       \bibitem{CCH}
F. Cakoni, D. Colton, and H. Haddar, {\it Inverse scattering theory and transmission eigenvalues,} Society for Industrial and Applied Mathematics, 2022.


       \bibitem{CV}
           F. Cakoni and M. Vogelius, {\it Singularities almost always scatter: Regularity results for nonscattering inhomogeneities}, Comm. Pure Appl. Math.,  {\bf 76(12)}(2023), 4022--4047. 

       \bibitem{CVX23}
F. Cakoni, M. S. Vogelius and J. Xiao, {\it On the regularity of non-scattering anisotropic inhomogeneities,} 
Arch. Ration. Mech. Anal., {\bf 247(3)}(2023), Paper No. 31, 15 pp.

   % \bibitem{CDLZ1}
    % X. Cao, H. Diao, H. Liu and J. Zou, {\it  On nodal and generalized singular structures of Laplacian eigenfunctions and applications to inverse scattering problems}, Journal de Math\'ematiques Pures et Appliqu\'ees,  {\bf 143}(2020), 116--161.

    % \bibitem{CDLZ2}
    %X. Cao, H. Diao, H. Liu and J. Zou, {\it  On novel geometric structures of laplacian eigenfunctions in $R^3$ and applications to inverse problems}, SIAM Journal on Mathematical Analysis, {\bf 53(2)}(2021), 1263--1294.

       \bibitem{CX}
           F. Cakoni and J. Xiao, {\it On corner scattering for operators of divergence and applications to inverse scattering}, Communications in Partial Differential Equation, {\bf 46(3)} (2021), 413--441.

    \bibitem{CK}
      D. Colton and R. Kress, {\it Inverse acoustic and electromagnetic scattering theory}, vol. 93 of Applied
Mathematical Sciences, Springer, New York, third ed., (2013).
           
             \bibitem{DR2018}
           D. Colton and R. Kress, {\it Looking back on inverse scattering theory}, SIAM Review, {\bf 60(4)}(2018), 779-807.

       \bibitem{Da}
           G. Dassios, {\it The Atkinson-Wilcox expansion theorem for elastic waves}, Quart. Appl. Math. {\bf 46(2)} (1988), 285--299.

       %\bibitem{DCL2021}  H. Diao, X. Cao and H. Liu, {\it On the geometric structures of transmission eigenfunctions with a conductive boundary condition and applications}, Comm. Partial Differential Equations, {\bf 46(4)}(2021), 630--679.

      % \bibitem{DFL}
      %     H. Diao, X. Fei and H. Liu, {\it  Local geometric properties of conductive transmission eigenfunctions and applications}, European Journal of Applied Mathematics, Published online 2024:1-32. doi:10.1017/S0956792524000287.

      % \bibitem{DFLY}
      %     H. Diao, X. Fei, H. Liu and K. Yang, {\it Visibility, invisibility and unique recovery of inverse electromagnetic problems with conical singularities}, Inverse Problems and Imaging, {\bf 18(3)}(2024), 541--570.

    \bibitem{DLbook}
   H. Diao and H. Liu, {\it Spectral Geometry and Inverse Scattering Theory}, Springer, Cham, 2023.
      
      % \bibitem{DLWY}
       %    H. Diao, H. Liu, X. Wang and K. Yang, {\it On vanishing and localizing around corners of electromagnetic transmission resonances}, Partial Differ. Equ. Appl., {\bf 2} (2021), 1--20.

      \bibitem{DLS2021}
          H. Diao, H. Liu and B. Sun, {\it On a local geometric property of the generalized elastic transmission eigenfunctions and application}, Inverse Problems, {\bf37(10)}(2021), 105015.

        \bibitem{DTLT} 
         H. Diao, R. Tang, H. Liu and J. Tang, {\it Unique determination by a single far-field measurement for an inverse elastic problem}, Inverse Probl. Imaging, {\bf 18(6)}(2024), 1405--1430. 

     %\bibitem{DLWY}
    %H. Diao, H. Liu, X. Wang and K. Yang, {\it  On vanishing and localizing around corners of electromagnetic transmission resonances}, Partial Differential Equations and Applications, 2: 1--20(2021).

   % \bibitem{DLZZ}
    %H. Diao, H. Liu, L. Zhang, and J. Zou, {\it Unique continuation from a generalized impedance edge corner for Maxwell’s system and applications to inverse problems}, Inverse Problems, {\bf37} (2021), pp. Paper No. 035004, 32.

       \bibitem{EH18}
J. Elschner and G. Hu, {\it Acoustic scattering from corners, edges and circular cones},  Arch. Ration. Mech. Anal., {\bf 228(2)}(2018), 653--690.

       \bibitem{FKS}
           R. Farwig, H. Kozono and H. Sohr, {\it  On the Helmholtz decomposition in general unbounded domains}, Archiv der Mathematik, {\bf 88}(2007), 239--248.
           
       \bibitem{GilTru83}
       D. Gilbarg and N. S. Trudinger,{\it Elliptic partial differential equations of second order}, Vol. 224 of Grundlehren der Mathematischen Wissenschaften [Fundamental Principles of Mathematical Sciences], Springer-Verlag, Berlin, 2nd edition, 1983.
       
      \bibitem{Hahner1}
          P. H\"ahner, {\it  A uniqueness theorem in inverse scattering of elastic waves}. IMA J. Appl. Math. 51, 201–-215(1993).

   %  \bibitem{Ike1}
  %   M. Ikehata and H. Itou, {\it  Extracting the support function of a cavity in an isotropic elastic body from a single set of boundary data}, Inverse Problems,  {\bf 25(10)}(2009), 105005.

   %  \bibitem{Ike2}
    % M. Ikehata and H. Itou, {\it Reconstruction of a linear crack in an isotropic elastic body from a single set of measured data}, Inverse Problems, {\bf 23(2)}(2007), 589.

      \bibitem{P1998}
           P. H\"ahner, {\it  On acoustic, electromagnetic, and elastic scattering problems in inhomogeneous media}, Habilitation Thesis, Universit\"at G\"ottingen, Mathematisches Institut, 1998.
        
     %  \bibitem{JLZ}
      %     Y. Jiang, H. Liu, J. Zhang and K. Zhang, {\it  Spectral patterns of elastic transmission eigenfunctions: boundary localization, surface resonance, and stress concentration}, SIAM Journal on Applied Mathematics, {\bf 83(6)}(2023), 2469--2498.
           
        \bibitem{KSS}
 P.  Kow, M. Salo and H. Shahgholian, {\it On scattering behavior of corner domains with anisotropic inhomogeneities,}  SIAM J. Math. Anal., {\bf 56(4)}(2024), 4834--4853.

        \bibitem{KW21}
            P. Kow and J. Wang, {\it On the characterization of nonradiating sources for the elastic waves in anisotropic inhomogeneous media}, SIAM Journal on Applied Mathematics, {\bf 84(1)} (2021), 1530--1551.

       \bibitem{Hu1}
           L. Li, G. Hu and J. Yang,  {\it Piecewise-analytic interfaces with weakly singular points of arbitrary order always scatter}, Journal of Functional Analysis, {\bf 284(5)}(2023), 109800.
     
      % \bibitem{LRX19}
       %    H. Liu, L. Rondi and J. Xiao,  {\it Mosco convergence for $ H $(curl) spaces, higher integrability for Maxwell's equations, and stability in direct and inverse EM scattering problems}, Journal of the European Mathematical Society, {\bf 21(10)}(2019), 2945--2993.
           
    %\bibitem{LT22}
    %   H. Liu and H.C. Tsou, {\it Stable determination by a single measurement, scattering bound and regularity of transmission eigenfunctions}, Calculus of Variations and Partial Differential Equations {\bf 61(3))}(2022): 91.

    \bibitem{LZ08}
       H. Liu and J. Zou, {\it Uniqueness in an inverse acoustic obstacle scattering problem for both
sound-hard and sound-soft polyhedral scatterers}, Inverse Problems, 22 (2006), 515–-524,

    \bibitem{MC}
      W. Mclean, {\it Strongly Elliptic Systems and Boundary Integral Equation}, Cambridge University Press, Cambridge (2000).
      
  %\bibitem{NW06}
  %G. Nakamura and J.-N. Wang. {\it Unique continuation for the two-Dimensional anisotropic elasticity system and its applications to inverse problems}, Transactions of the American Mathematical Society, {\bf 358(7)}(2006), 2837–-2853. %JSTOR, http://www.jstor.org/stable/3845555. Accessed 21 May 2025.
  
   \bibitem{PSV}
     L. P\"aiv\" arinta, M. Salo, and E. Vesalainen, {\it Strictly convex corners scatter}, Rev. Mat. Iberoam., {\bf 33(4)}(2017), 1369--1396.

  %   \bibitem{RSS}
  %    L. Rondi, E. Sincich and M. Sini, {\it Stable determination of a rigid scatterer in elastodynamics}, SIAM J. Math. Anal. {\bf 53(2)}(2021), 2660-–2689. 

\bibitem{SS21}
M. Salo and H. Shahgholian, {\em Free boundary methods and non-scattering phenomena,} Res. Math. Sci., {\bf 8(4)}(2021), Paper No. 58, 19 pp.

%\bibitem{SSini}
%E. Sincich and M. Sini, {\it Local stability for soft obstacles by a single measurement},  Inverse Probl. Imaging, {\bf 2(2)}(2008), 301--315.



       %\bibitem{YZ} G. Yuan and Y. Zhao, {\it Increasing stability for the inverse source problem in elastic waves with attenuation}, European Journal of Applied Mathematics, {\bf 34(4)}(2023), 896--910.
          
          
       



    	
    	\end{thebibliography}
	\end{document}